\documentclass[pagesize=pdftex,10pt,a4paper,oneside,openright,final]{article}
\usepackage[english]{babel}
\usepackage[utf8]{inputenc}		
\usepackage{verbatim}
\usepackage{supertabular}
\usepackage{stmaryrd}
\usepackage{mathrsfs}
\usepackage{longtable}
\usepackage{varioref}
\usepackage{calc}
\usepackage{ifthen}
\usepackage{indentfirst}
\usepackage{fancyhdr}
\usepackage{enumerate}
\usepackage{float}
\usepackage{color}
\usepackage[usenames,dvipsnames]{xcolor}
\usepackage{makeidx}
\usepackage{graphicx}
\graphicspath{{./}{./Numerics/}{%
./Numerics/HermesKnot_Menger_000600E_FrozenIntrinsicGagliardoMetric/HermesKnot_Menger_000600E_FrozenIntrinsicGagliardoMetric/%
}{%
./Numerics/SquareKnot_Menger_000600E_FrozenIntrinsicGagliardoMetric/SquareKnot_Menger_000600E_FrozenIntrinsicGagliardoMetric/%
}{%
./Numerics/TorusKnot_5_3_Symmetric_Menger_000240E_FrozenIntrinsicGagliardoMetric/TorusKnot_5_3_Symmetric_Menger_000240E_FrozenIntrinsicGagliardoMetric/%
}{%
./Numerics/TorusKnot_5_3_Asymmetric_Menger_000240E_FrozenIntrinsicGagliardoMetric/TorusKnot_5_3_Asymmetric_Menger_000240E_FrozenIntrinsicGagliardoMetric/%
}{%
./Numerics/TorusKnot_5_3_Asymmetric_Menger_000240E_TrustRegion/TorusKnot_5_3_Asymmetric_Menger_000240E_TrustRegion/%
}}
\makeatletter
\def\input@path{{./}{./Numerics/}}
\makeatother
\usepackage{multicol}
\usepackage[BCOR4mm,DIV10]{typearea}	
\usepackage{titlesec}			
\usepackage[refpage]{nomencl/nomencl}
\usepackage{amssymb}
\usepackage{amsxtra}
\usepackage{amsmath} 
\usepackage{amsfonts}
\usepackage{tikz}
\usepackage{tikz-cd}
\usepackage{esint}
\usepackage[thref, amsthm, hyperref, thmmarks, amsmath]{ntheorem}
\usepackage{etoolbox}
\usepackage{lmodern}
\usepackage[T1]{fontenc}

\DeclareRobustCommand{\newterm}[2][]{\emph{#2}\ifthenelse{\equal{#1}{}}{\index{#2}}{\index{#1}}}

\DeclareRobustCommand{\[}{\begin{equation*}}
\DeclareRobustCommand{\]}{\end{equation*}}

\newcommand{\abs}[1]{\left\lvert #1 \right\rvert}
\newcommand{\norm}[2][]{\left \lVert #2 \right \rVert_{#1}}
\newcommand{\halfnorm}[2][W]{\left[#2\right]_{#1}}
\newcommand{\hatProduct}[1]{\widehat{\left[#1\right]}}

\newcommand{\AL}{\mathcal{A}}
\newcommand{\BL}{\mathcal{B}}

\DeclareMathOperator{\hol}{H\ddot ol}

\DeclareMathOperator{\loc}{loc}

\DeclareMathOperator{\intM}{intM}
\DeclareMathOperator{\ir}{ir}
\DeclareMathOperator{\evi}{EVI}
\DeclareMathOperator{\AC}{AC}
\DeclareMathOperator{\BiLip}{BiLip}

\DeclareMathOperator{\epi}{epi}
\newcommand{\g}{\ensuremath{\gamma}}
\newcommand{\Jj}{\ensuremath{\mathfrak{J}}}
\newcommand{\GLL}{\ensuremath{\mathcal{GL}}}
\newcommand{\Mpq}[1][p,q]{\ensuremath{\intM^{\left(#1\right)}}}
\newcommand{\Mpp}[1][p,p]{\ensuremath{\intM^{\left(#1\right)}}}
\newcommand{\Rpq}[1][p,q]{R^{\left(#1\right)}}
\newcommand{\RpqX}[5]{ \frac {\left(\abs{#5-#3} \abs{#4-#3} \abs{#5-#4} \right)^{#1}} {\abs{\left(#5 - #3 \right) \wedge \left(#4 - #3 \right)}^{#2}}}
\newcommand{\restrict}[2]{\left . #1 \right\rvert_{#2}}
\newcommand{\numberthis}{\refstepcounter{equation}\tag{\theequation}}

\newcommand{\ball}[2][0]{\ensuremath{B_{#2}\left(#1\right)}}

\renewcommand{\ker}{\NL}

\newlength{\adjustedalignmentskip}

\renewcommand{\rho}{\varrho}

\allowdisplaybreaks[1]

\makeatletter
\renewcommand{\thm@@thmline@name}[5]{\noindent\hyper@linkstart{link}{#5}#1 #2\hyper@linkend \ifblank{#3}{}{\ (#3)}, Page #4:\vspace{1.5cm}\\}%
\makeatother

\usepackage[pdftex=true,breaklinks=true,bookmarks=true,linktocpage=true,pdfpagelabels=true,plainpages=false,pagebackref=true]{hyperref}	

\usepackage[all]{hypcap}		

\theoremstyle{plain}
\newtheorem{lemma}{Lemma}[section]
\newtheorem{theorem}[lemma]{Theorem}
\newtheorem{proposition}[lemma]{Proposition}
\newtheorem{corollary}[lemma]{Corollary}
\newtheorem{definition}[lemma]{Definition}

\renewtheorem*{theorem*}{Theorem}
\renewtheorem*{lemma*}{Lemma}
\renewtheorem*{definition*}{Definition}

\theoremstyle{definition}

\numberwithin{equation}{section}
\newcommand{\disstitle}{A speed preserving Hilbert gradient flow for 
generalized integral Menger curvature} 
\newcommand{\disstitletextpdf}{A Projected Gradient Flow for the Generalized Integral Menger Curvature in the Hilbert Case} 
\title{\disstitle}
\author{Jan Knappmann\thanks{Institut für Mathematik, RWTH Aachen University, Templergraben 55, D-52062 Aachen, Germany} \and Henrik Schumacher\footnotemark[1]
\and Daniel Steenebr\"ugge\footnotemark[1] \thanks{Corresponding author email: \href{mailto:steenebruegge@instmath.rwth-aachen.de}{steenebruegge@instmath.rwth-aachen.de}} \and Heiko von der Mosel\footnotemark[1]}

\hypersetup{
pdfauthor={Jan Knappmann, Henrik Schumacher, Daniel Steenebr\"ugge, Heiko von der Mosel},
pdftitle={\disstitletextpdf},
pdfproducer={LaTeX with hyperref package},
pdfkeywords={},
hypertexnames=true,
pdfstartview={FitV},
pdfstartpage={1},
pdfhighlight={/I},
pdfborder={0 0 0 0},
bookmarksopen=false,
bookmarksnumbered=true,
pdfpagemode={UseOutlines},
unicode=false,
breaklinks=true,
frenchlinks=false,
a4paper=true,
pdftoolbar=true,
pdfmenubar=true,
pdfcenterwindow=true,
pdffitwindow=true,
pdfnewwindow=true,
colorlinks=true,
linkcolor=blue,
citecolor=blue,
urlcolor=blue,
filecolor=blue
}

\DeclareMathOperator*{\essinf}{ess\,inf}

\DeclareMathOperator{\Lip}{Lip}

\renewcommand{\epsilon}{\varepsilon}

\newcommand{\operp}{\raisebox{0.25ex}{\ensuremath{\,{\scriptstyle\mathop{\bigcirc\kern-0.71em\perp}}\,}}}

\newcommand{\menge}[1]{\left\{ #1 \right\}}				

\newcommand{\R}{\mathbb{R}}
\newcommand{\N}{\mathbb{N}}
\newcommand{\Z}{\mathbb{Z}}

\newcommand{\B}{\mathcal{B}}

\newcommand{\HL}{\mathcal{H}}
\newcommand{\RL}{\mathcal{R}}

\newcommand{\X}{\mathcal{X}}

\newcommand{\EL}{\mathcal{E}}

\newcommand{\LL}{\mathcal{L}}

\newcommand{\NL}{\mathcal{N}}

\newcommand{\LIP}{\textnormal{Lip}}
\newcommand\Id{{{\rm Id}}}
\newcommand{\difference}[1]{\Delta_{#1}}
\renewcommand\S{{\mathbb S}}
\newcommand\SIGMA{{\Sigma}}

\newcommand{\Fo}{\,\,\,\text{for }\,\,}
\newcommand{\Foa}{\,\,\,\text{for all }\,\,}

\newcommand{\AND}{\,\,\,\text{and }\,\,}



\renewcommand*{\nompageref}[1]{\if@printpageref\pagedeclaration{#1}\fi\dotfill\hyperpage{#1}\nomentryend\endgroup}	

\RequirePackage{ifthen}
\renewcommand{\nomgroup}[1]{\bigskip
    \ifthenelse{\equal{#1}{A}}{\item[\textbf{\large A} \quad \textbf{Vector and Set Notation}] \item[]}{
        \ifthenelse{\equal{#1}{B}}{\item[\textbf{\large B} \quad \textbf{Functional Notation}]  \item[]}{}}{
        \ifthenelse{\equal{#1}{C}}{\item[\textbf{\large C} \quad \textbf{Function Spaces}]  \item[]}{}}{
        \ifthenelse{\equal{#1}{D}}{\item[\textbf{\large D} \quad \textbf{Measures}]  \item[]}{}}
    {\hrule height 0.4pt \hbox to \textwidth{\rule[-0.10cm]{3.2cm}{0.10cm}\hfill}}\vspace{-0.5cm}}

\newcommand{\ceq}{:=}

\newcommand{\dd}{\mathrm{d}}
\newcommand{\nabs}[1]{|#1|}
\newcommand{\nnorm}[1]{\|#1\|}

\newcommand{\ninnerprod}[1]{\langle#1\rangle}

\newcommand{\Triangulation}{T}

\newcommand{\Polygon}{P}
\newcommand{\EdgeLengths}{\ell}

\newcommand{\Edges}{E}
\newcommand{\Edge}{I}
\newcommand{\Vertices}{V}
\newcommand{\Vertex}{v}
\newcommand{\Midpoint}{m}

\newcommand{\up}[1]{#1^{\uparrow}}
\newcommand{\down}[1]{#1^{\downarrow}}

\newcommand{\Energy}{\mathcal{E}}
\newcommand{\Curve}{\gamma}
\newcommand{\ConstraintMap}{\Sigma}
\newcommand{\Metric}{G}
\newcommand{\RieszOp}{\mathcal{J}}
\newcommand{\HilbertSpace}{\mathcal{H}}
\newcommand{\AmbDim}{n}
\newcommand{\AmbSpace}{{\R^\AmbDim}}
\newcommand{\Circle}{{\R\slash\Z}}

\newcommand{\VertexCount}{N}

\newcommand{\LandO}{O}

\newcommand{\SaddlePointMatrix}{\mathcal{A}}
\newcommand{\BarycenterContraint}{\mathcal{C}}
\newcommand{\transp}{^\top}

\begin{document}
\pagenumbering{arabic}\setcounter{page}{1}
\pagestyle{fancy}

\maketitle

\thispagestyle{empty}

\begin{abstract}
    We establish long-time existence for a projected Sobolev gradient flow of generalized integral Menger curvature in the Hilbert case, and provide $C^{1,1}$-bounds in time for the solution that only depend on the initial curve.
    The self-avoidance property of integral Menger curvature guarantees that the knot
    class of the initial curve is preserved under the flow, and the projection ensures that each curve along the flow is parametrized with the same speed as the initial configuration. Finally, we describe how to simulate this flow numerically with substantially higher efficiency than in the corresponding numerical $L^2$ gradient descent or other optimization methods.    
\end{abstract}

\section{Introduction}
Integral Menger curvature is one of several geometrically defined curvature 
functionals that are used in geometric knot theory to separate different 
knot classes by infinite energy barriers. It is defined as the triple
integral 
\begin{equation}\label{eq:integral-menger}
\mathcal{M}_p(\gamma):=\iiint_{(\R/\Z)^3}
\frac {\abs{\gamma'(u_1)}\abs{\gamma'(u_2)} \abs{\gamma'(u_3)}}
{R^p\left(\gamma(u_1),\gamma(u_2),\gamma(u_3)\right)} \, \dd u_1 \, \dd u_2 \, \dd u_3
\end{equation}
evaluated on absolutely continuous closed curves $\gamma
\in \AC(\R/\Z,\R^n)$, 
where $R(x,y,z)$ denotes the circumcircle
radius of three points $x,y,z\in\R^n$. 
For exponents $p$ above scale-invariance, that is for $p>3$, 
this geometric curvature energy has regularizing properties. It was shown in
\cite{strzelecki-etal_2010} for $n=3$
that any locally homeomorphic arc length
parametrization $\Gamma$ of $\gamma$ with $\mathcal{M}_p(\gamma)<\infty$ is 
an embedding or a multiple cover of the image manifold of class $C^{1,1-(3/p)}$, which can be viewed as a geometric
Morrey-Sobolev embedding theorem. This interpretation
was later confirmed by S. Blatt's characterization of finite energy curves
as exactly those arc length parametrized
embeddings that are of fractional Sobolev-Slobodecki\v{\i}\footnote{
    For the definition
    and several useful
    properties of periodic Sobolev-Slobodecki\v{\i} spaces
    we refer to
    Appendix~\ref{section:Sobolev}.}
regularity $W^{2-(2/p),p}$, which precisely embeds into
$C^{1,1-(3/p)}$ \cite{blatt_2013a}. The combination of 
its self-avoidance and regularizing effects can be used to 
minimize integral Menger curvature
within any prescribed
given tame knot class,  and to bound classic knot invariants
and therefore the number of knot classes representable below given
energy values. This makes $\mathcal{M}_p$ a valuable instrument
in geometric knot theory. We refer to the survey
\cite{StrzekeckivonderMosel:2013:MengerSurvey} for more information
on the r{\^o}le of integral Menger curvature in the context of knot theory.

The Euler-Lagrange equations for integral Menger curvature were first derived
by T. Hermes \cite{hermes_2014}. The principal part of this complicated
integro-differential equation seems too degenerate to expect 
smoothness of
critical points of $\mathcal{M}_p$. Therefore, Blatt and Ph.\ Reiter 
\cite{BlattReiter:2015:Menger} came up with
the idea of splitting the powers in the numerator and denominator
of the fraction defining the circumcircle radius $R(\cdot,\cdot,\cdot)$
in \eqref{eq:integral-menger}, which led them to introduce
the family of \emph{generalized integral Menger curvatures}
\begin{equation}\label{eq:gen-integral-menger}
\Mpq{}(\g) 
:= \iiint_{(\R / \Z)^3} \frac {\abs{\gamma'(u_1)}\abs{\gamma'(u_2)} 
    \abs{\gamma'(u_3)}} {\Rpq\left(\gamma(u_1),\gamma(u_2),\gamma(u_3)\right)} 
\, \dd u_1 \, \dd u_2 \, \dd u_3,
\end{equation}
where 
\[\Rpq(x,y,z):=\RpqX{p}{q}{x}{y}{z} \quad\textnormal{for 
    $x,y,z\in\R^n $}\]
generalizes the circumcircle's  diameter\footnote{Throughout the paper $a\wedge b$
    denotes the exterior product of the two vectors $a,b\in\R^n$, which 
    reduces to the usual cross product $a\times b\in \R^3$ for $n=3$.} 
so that one has $\mathcal{M}_p=2^p
\Mpp{}.$ It turns out that finite energy curves of class $C^1$
possess an arc length
parametrization of class $W^{(3p-2)/q-1,q}$ if $q>1$ and $p\in
(\tfrac23 q +1,q+\tfrac23 )$,  
which implies that also $\intM^{\left( p,q\right)}$
can be minimized in tame knot classes; see
\cite[Theorems 1 \& 2]{BlattReiter:2015:Menger}. In the more specific 
Hilbert space case, i.e. for $q=2$ and
$p\in (\frac73 , \frac 83)$
(thus excluding the original integral Menger curvature, unfortunately),
the variational equations become more accessible for regularity arguments,
so that Blatt and Reiter could show $C^\infty$-smoothness for any 
critical point of $\intM^{\left(p,2\right)}$ \cite[Theorem 4]{BlattReiter:2015:Menger}.

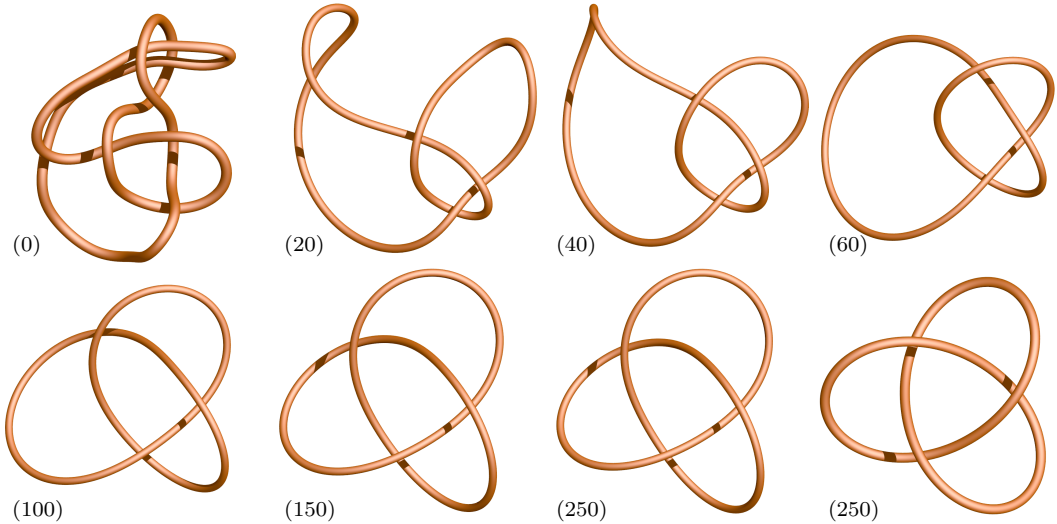
\begin{figure}
    \begin{center}
        \newcommand{\myincludegraphics}[2]{\begin{tikzpicture}
            \node[inner sep=0pt] (fig) at (0,0) {\includegraphics{#1}};
            \node[above right= 0.05ex] at (fig.south west) {\footnotesize(#2)};    
            \end{tikzpicture}
        }%
        \presetkeys{Gin}{
			trim = 45 40 50 50, 
             clip = true,  
             angle = 0,
             width = 0.24\textwidth
        }{}%
        \capstart%
        \myincludegraphics{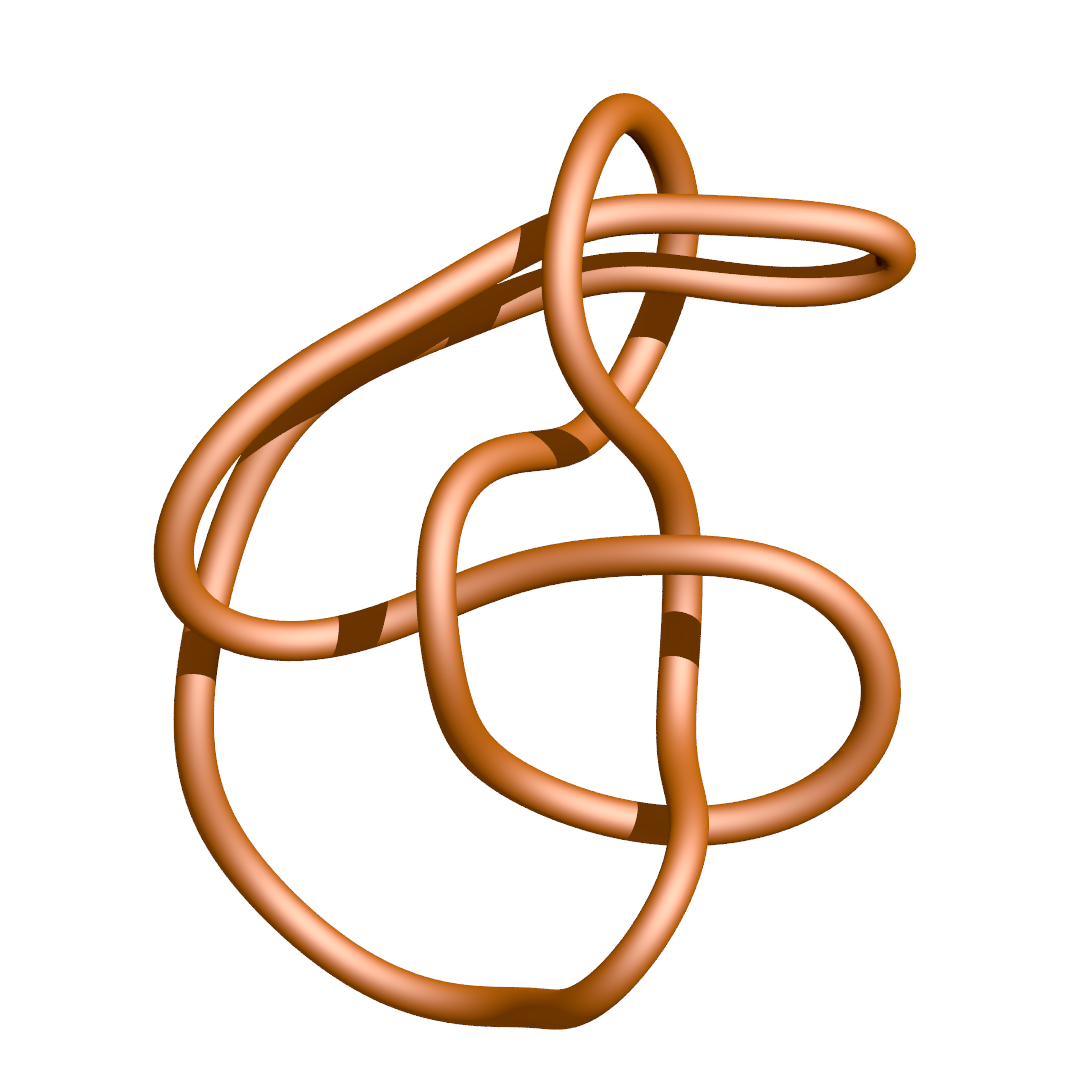}{0}%
        \myincludegraphics{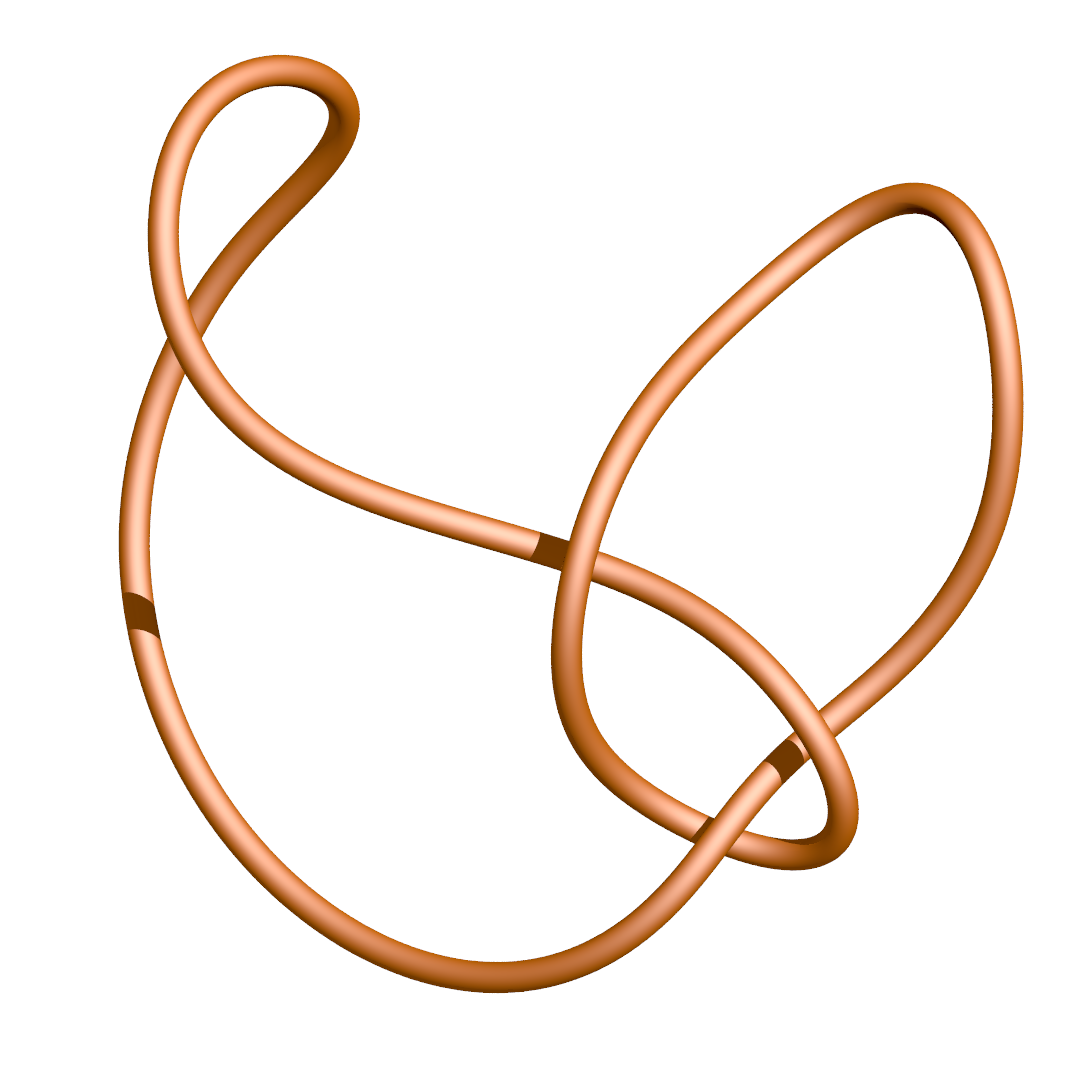}{20}%
        \myincludegraphics{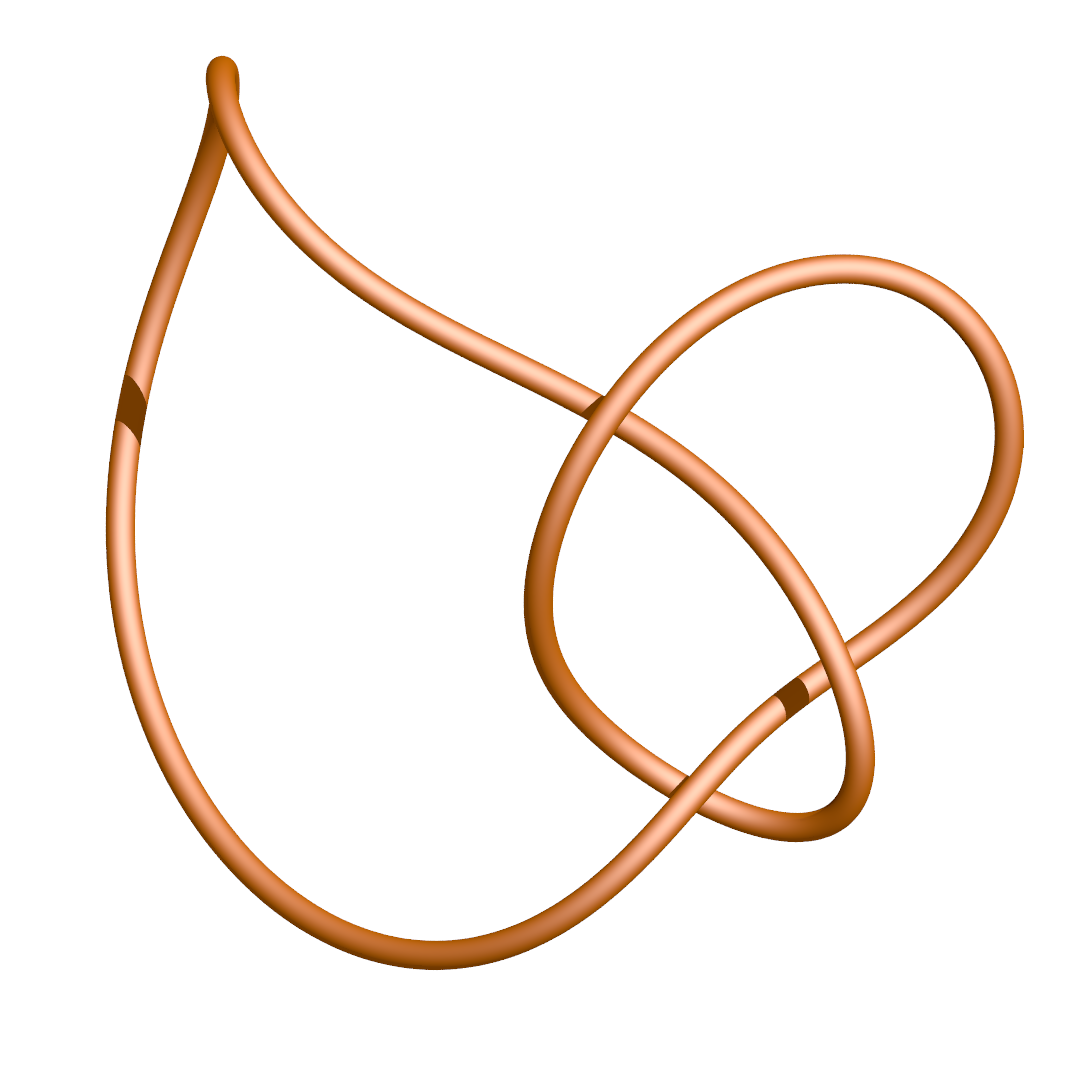}{40}%
        \myincludegraphics{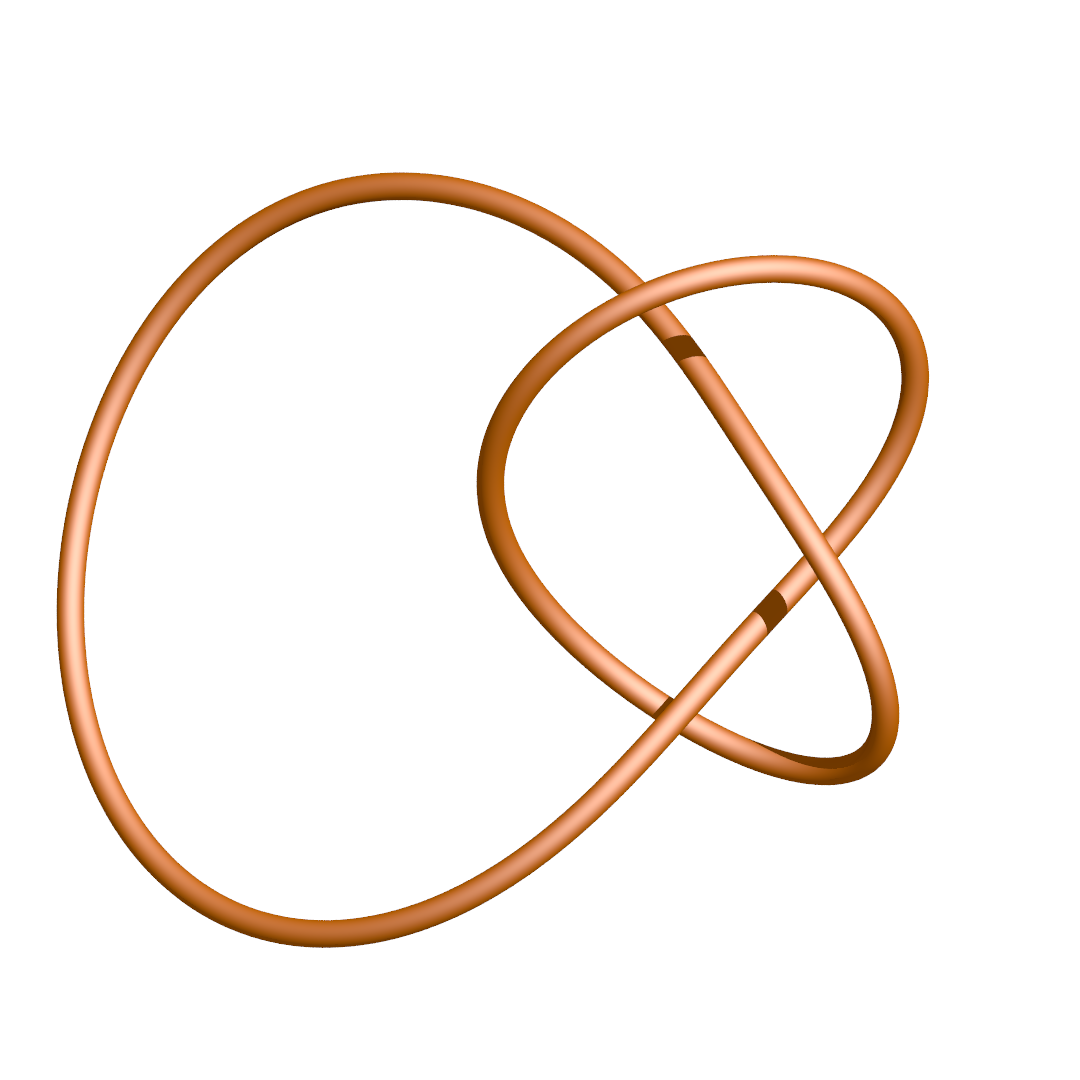}{60}%
        \\
        \myincludegraphics{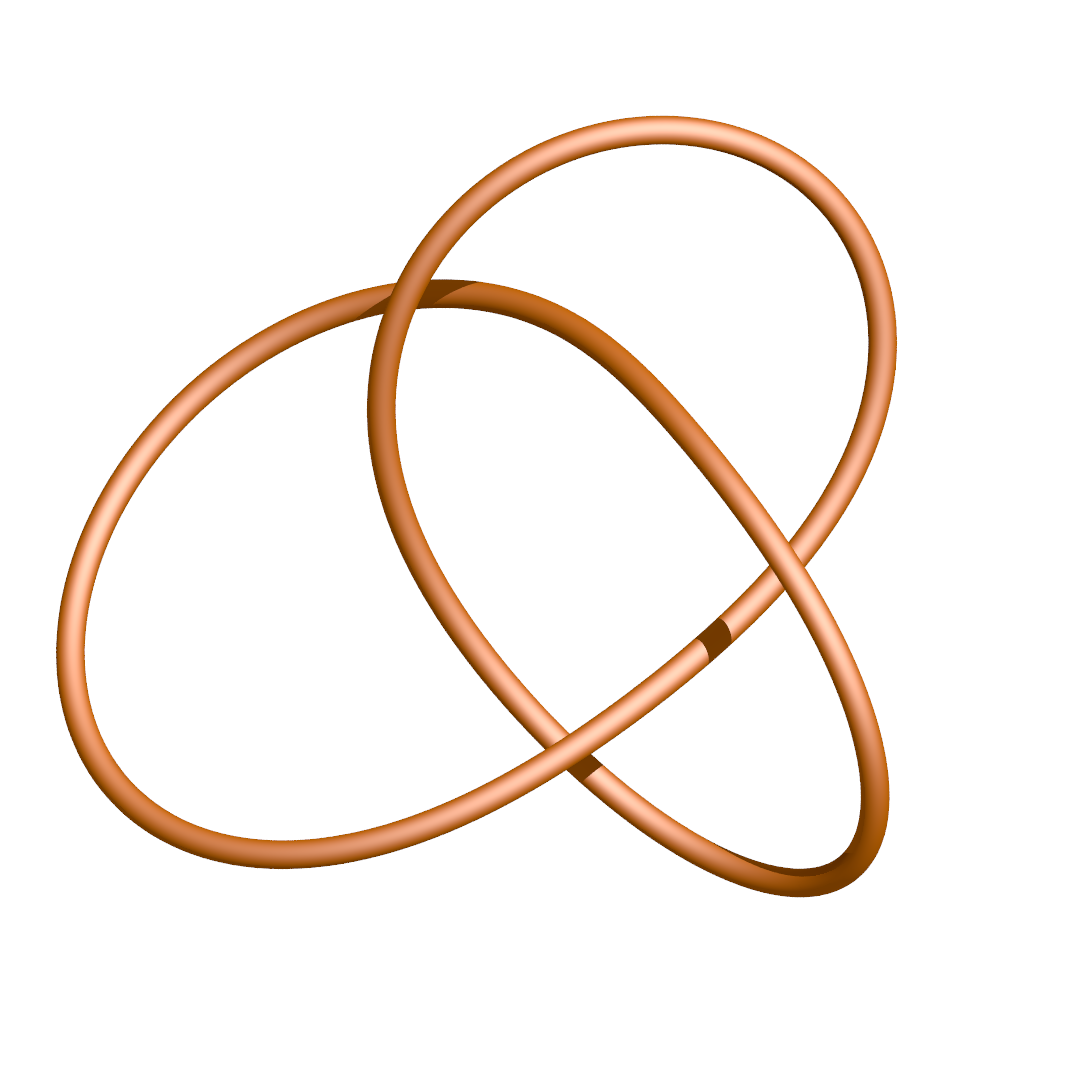}{100}%
        \myincludegraphics{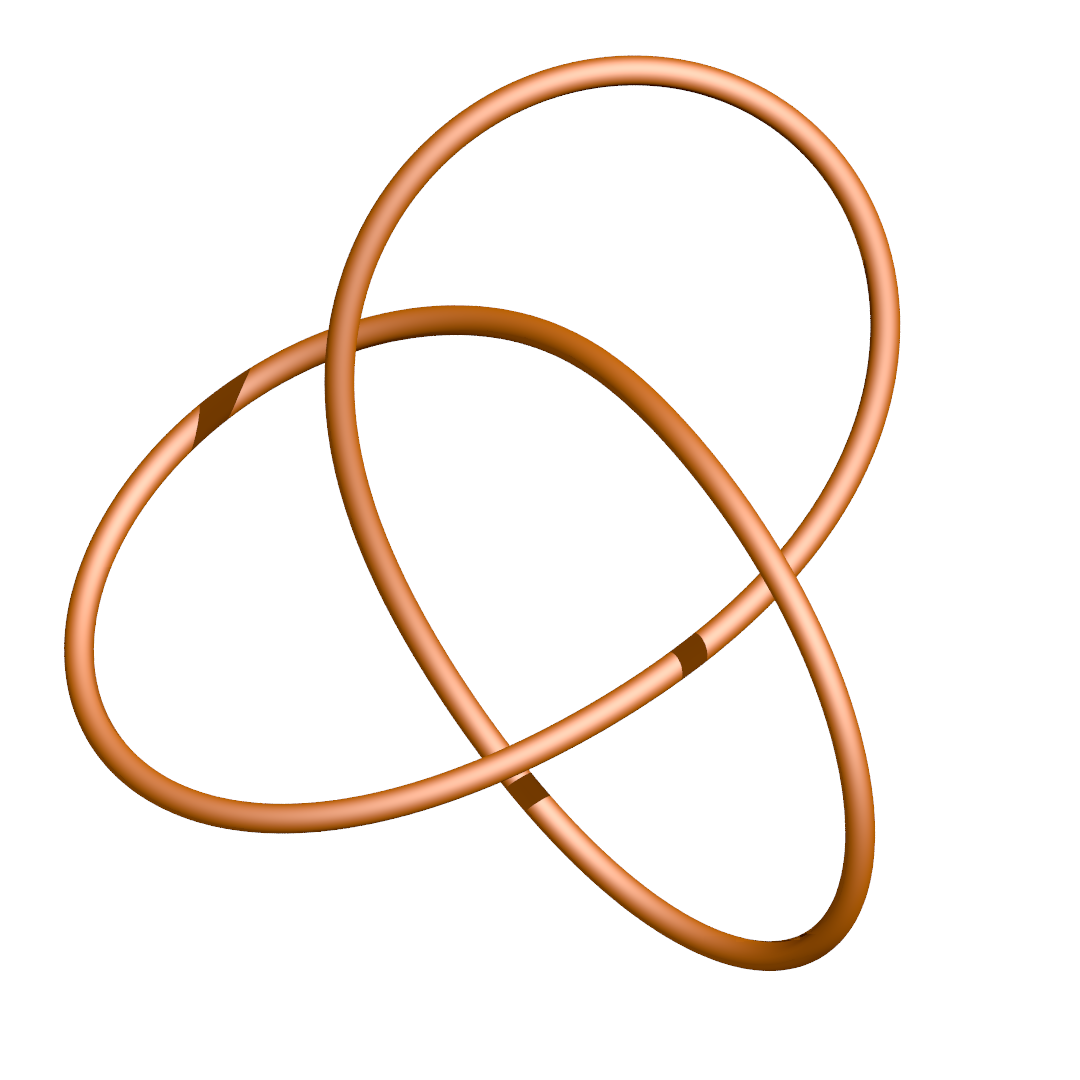}{150}%
        \myincludegraphics{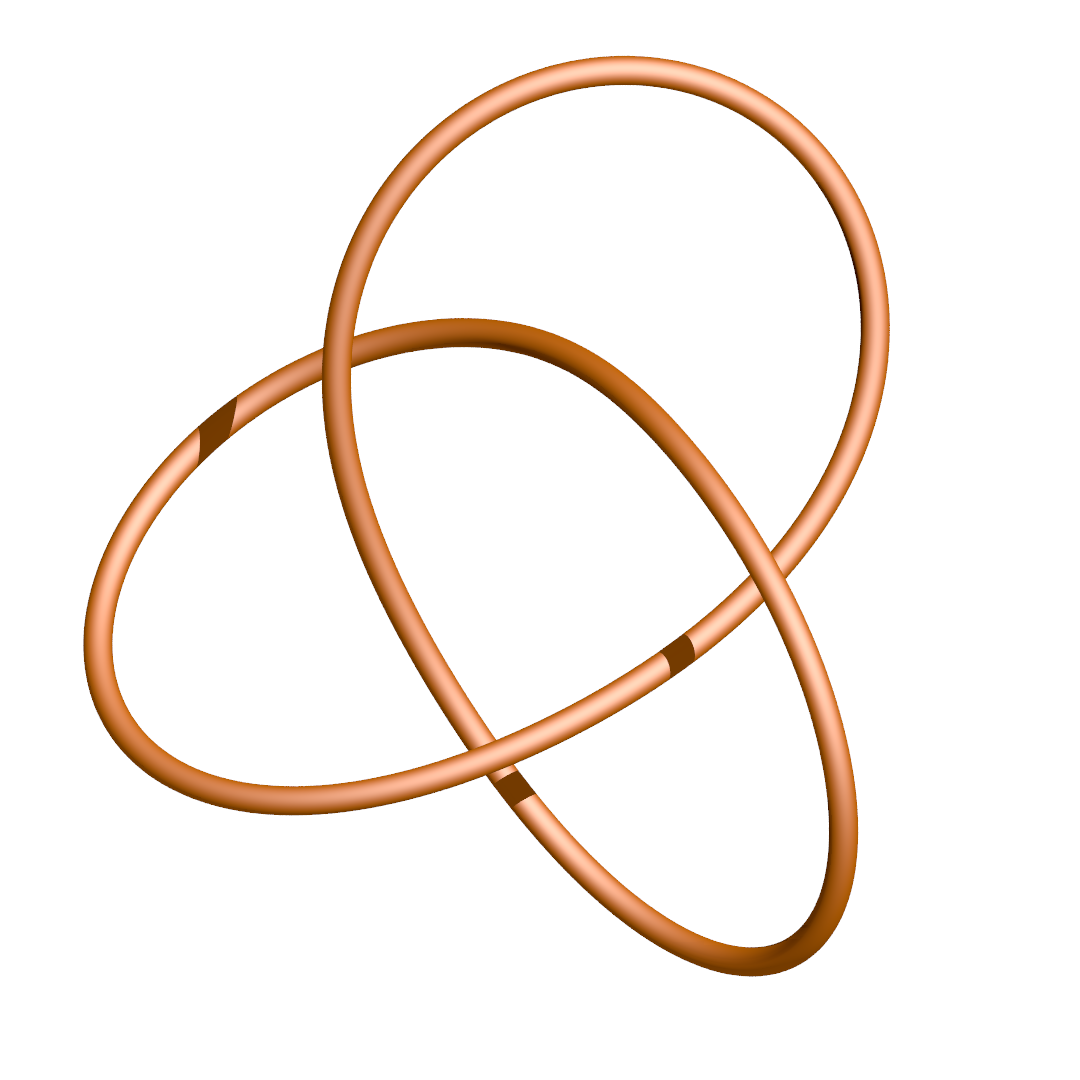}{250}%
        \myincludegraphics{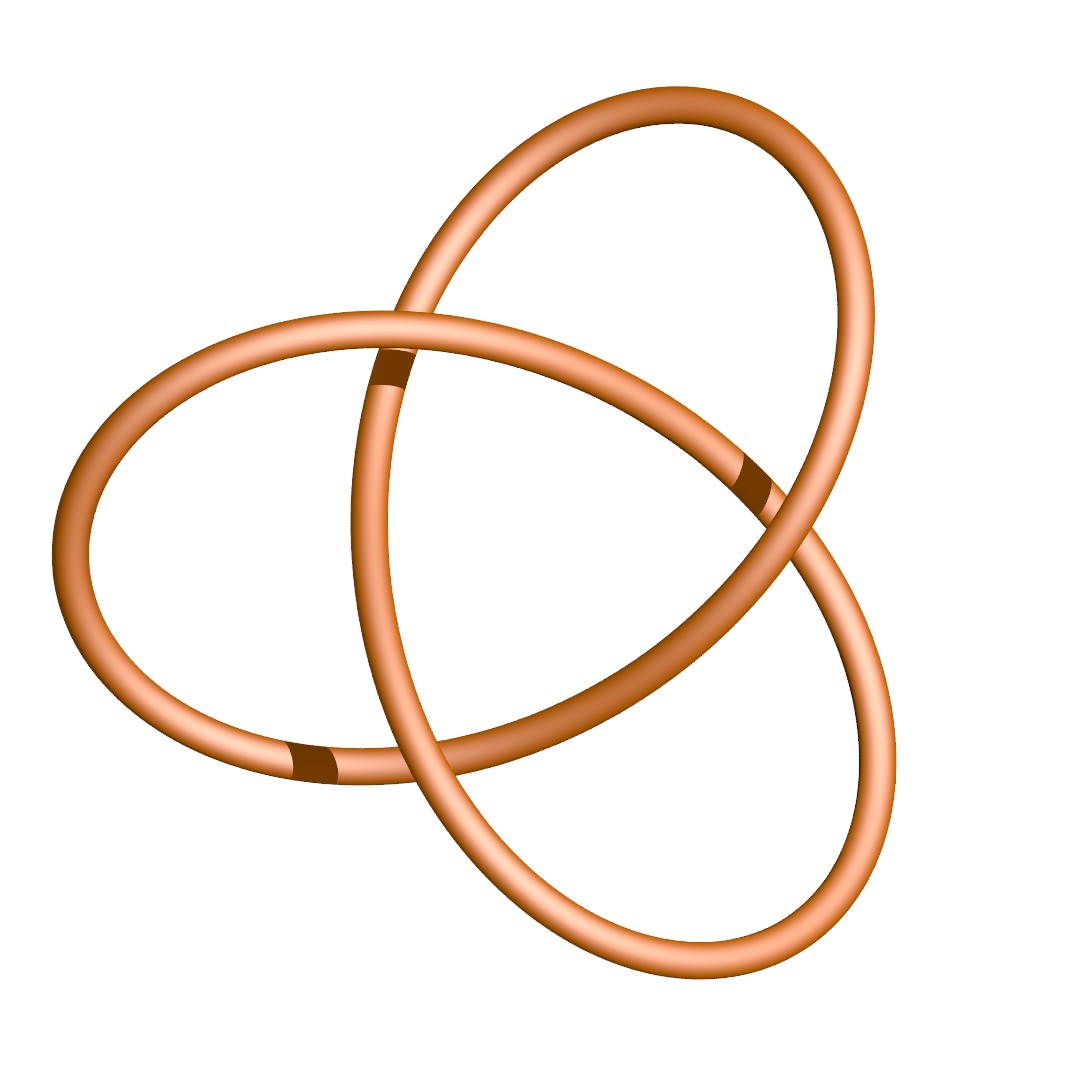}{250}%
    \end{center}
    \caption{Select iterations of the projected Sobolev gradient flow with line search. (Iteration numbers are indicated in parentheses.) The discrete model is a polygonal line with 600 edges. The obtained minimizer (bottom right) is shown from a further angle in order to reveal its three-fold rotational symmetry.}
    \label{fig:Flow_HermesKnot}
\end{figure}

It is natural to ask if one can set up a gradient flow 
\begin{equation}
\label{eq:gradflow}
\begin{cases}
\frac{ \partial} {\partial t} \gamma(\cdot,t) &= -\nabla 
\Mpq[p,2](\gamma(\cdot,t))\quad\textnormal{on $\R/\Z$ for $t>0$,}\\
\gamma(\cdot,0) &= \gamma_0.
\end{cases}
\end{equation}
in
the appropriate
Hilbert space $\mathcal{H} :=W^{\frac32 p-2,2}(\R/\Z,\R^n)$, 
deforming any given
initial knotted curve $\gamma_0$ in that space to a critical point of 
$\intM^{\left(p,2\right)}$ within the knot class represented
by $\g_0$.
Note that we consider the full Sobolev-Slobodecki\v{\i} gradient here, 
in other words, \eqref{eq:gradflow} represents 
the $W^{\frac32 p-2,2}$-gradient flow for
the generalized integral Menger curvature $\Mpq[p,2]$.
For a visualization of this flow, see \autoref{fig:Flow_HermesKnot}.
The energy is indeed continuously differentiable
on regular embeddings of class
$W^{\frac32 p-2,2}(\R/\Z,\R^n)$ according to 
\cite[Theorem 3]{BlattReiter:2015:Menger}, and we show in 
Section \ref{sec:2} by analyzing its
second variation that its differential  $D\Mpq[p,2]$
is even
locally Lipschitz continuous. By the classic Picard-Lindel\"of Theorem
for ordinary differential equations in Banach spaces this 
can be turned into a short time 
existence result for \eqref{eq:gradflow}, 
together with an explicit lower bound on the
maximal time of existence, depending only on the Sobolev-Slobodecki\v{\i}
seminorm of the tangent $\g_0'$
and on the bilipschitz constant of the initial curve $\g_0$; see 
\thref{theorem:shortTimeExistence}
of Section
\ref{section:ShortTimeExistence}.
For arc length parametrized curves both 
these quantities can be controlled in terms of the energy according
to \cite[Theorem~1 \& Proposition~2.1]{BlattReiter:2015:Menger}. 
Unfortunately, it seems hard
to keep the velocities $|\g'(\cdot,t)|$  of the curves $\g(\cdot,t)$
bounded away from zero along the evolution \eqref{eq:gradflow}, 
or to control the fractional Sobolev seminorms
of the time-dependent reparametrizations to arc length,
which would allow us to continue
the flow up to infinite time. The
right-hand side $-\nabla\intM^{\left( p,2\right) }=-
\RieszOp^{-1}
(D\intM^{
    \left(p,2\right)})$
of \eqref{eq:gradflow}  is 
not explicit because of the underlying Riesz-isomorphism $
\RieszOp\colon \mathcal{H}\to\mathcal{H}^*$ for the Hilbert space $\mathcal{H}=W^{\frac32 p-2,2}(\R/\Z,\R^n)$.
This makes 
it hard to derive evolution equations
for the velocities of the evolving curves. Instead we overcome
this difficulty of possibly degenerating parametrizations by projecting the
energy gradient onto the   null-space of a constraint's gradient, thus preserving the initial velocity $|\g_0'(\cdot)|$ along this projected 
flow.  This idea, which  goes back to J. W. Neuberger 
\cite[Chapter 6]{Neuberger:1997:SobolevGradients}, was used by the
second author
in cooperation with S. Scholtes and M. Wardetzky 
in the context of discrete elasticae 
\cite{ScholtesSchumacherWardetzky:2019:DiscreteElasticae}.

Let us briefly describe Neuberger's approach in an abstract setting
first. For two real 
Hilbert spaces $\HL_1,\HL_2$, and an open set 
$\mathcal{O}\subseteq \HL_1$, consider an energy 
functional $E \colon \mathcal{O} \to \R$ and a constraint
mapping $S \colon \mathcal{O} \to \HL_2$, both Fréchet differentiable.
If we aim for 
a curve $x \colon [0,T] \to \mathcal{O}$ that chooses the direction of 
steepest descent for $E$ under the condition that it conserves an initial value of $S $, we may use the differential equation
\[
\tag{PrFlow}
\label{equation:ODEP}
\frac d {dt} x(t) = - \nabla_{S} E(x(t)), \quad t>0,
\]
where the projected gradient on the right-hand side is defined as 
\begin{equation}\label{eq:proj-abstract-gradient}
\nabla_{S} E(y):= \Pi_{\NL(DS (y))} \nabla E(y)\quad\Fo y\in 
\mathcal{O}.
\end{equation}
Here,  $\Pi_{\NL(DS (y))}$ denotes 
the orthogonal projection onto the (closed) null space of 
the differential $DS (y)$ of the constraint mapping $S$ at 
$y \in \mathcal{O}$.
Indeed, if $x$ satisfies the projected evolution equation
\eqref{equation:ODEP}, we simply differentiate the constraint with
respect to time to obtain
\[
\textstyle\frac d {dt} S (x(t)) 
= DS (x(t))\frac {d} {dt} x(t)
\overset{\eqref{equation:ODEP}}{=} - DS (x(t)) \Pi_{\NL(DS (x(t)))} 
\nabla E(x(t))
= 0
\]
and therefore, 
\[
\label{eq:invarianceBoundaryCondition}
\numberthis
S (x(t))=S (x(0)) \quad\textnormal{for all $t\ge 0$}.
\]
In Section \ref{sec:4} we analyze in the abstract Hilbert space setting under
which circumstances the orthogonal projection $A\mapsto \Pi_\mathcal{\NL(A)}$
for a bounded linear operator $A\colon \HL_1\to\HL_2$ 
is Lipschitz continuous in order to keep the projected flow
\eqref{equation:ODEP} in the realm of the
Picard-Lindel\"of existence theory.

Our actual choice of constraint in the present context is, as in 
\cite[Section 2]{ScholtesSchumacherWardetzky:2019:DiscreteElasticae}, the  
\emph{logarithmic strain}  $S:=\SIGMA$ defined as
\begin{equation}\label{eq:log-constraint}
\SIGMA (\g):= \log \left( |\g'(t)|\right) \in 
\mathcal{H}_2:=W^{\frac 3 2 p - 3,2}
(\R/\Z) 
\end{equation}
for curves $\g$ contained in the open\footnote{In \thref{cor:bilip2} we provide a
    quantitative version of the fact that the regular embedded  curves
    form an open subset in this fractional Sobolev space.} subset 
\begin{equation}\label{eq:open-subset}
\mathcal{O}:=W^{\frac32 p-2,2}_{\textnormal{ir}}(\R/\Z,\R^n)
\subset\mathcal{H}_1:=\mathcal{H}=
W^{\frac32 p-2,2}(\R/\Z,\R^n)
\end{equation}
of injective regular curves in $\mathcal{H}$, 
i.e.,  for those  $\g\in\mathcal{H}$ of which the restrictions
$\g|_{[0,1)}$ are injective and the velocities $|\g'|$
are strictly positive 
on $\R/\Z$. 
With this choice we intend to guarantee
the conservation of the initial velocity throughout the  projected
gradient flow for $\Mpq[p,2]$. 
Section 
\ref{subsection:explicitBoundaryCondition} is therefore 
devoted to establishing
the sufficient conditions of Section \ref{sec:4}
for the logarithmic strain constraint  $\SIGMA $ to maintain 
uniform Lipschitz continuity. 
The main issue here is 
to prove the existence
of a bounded right inverse for  the constraint's differential $D\SIGMA
(\g)$
for  $\g\in W^{\frac32 p-2,2}_{\textnormal{ir}}(\R/\Z,\R^n)$.
In Section \ref{sec:6}
we combine the results of Sections \ref{sec:2} and 
\ref{subsection:explicitBoundaryCondition} to obtain a locally
Lipschitz continuous right-hand side of the projected evolution
for the generalized integral Menger curvatures $\Mpq[p,2]$.
In addition, we gain  explicit quantitative control over 
the size of the right-hand side's 
Lipschitz domain,
which allows us
to prove the following central 
long time existence result by means of standard
continuation arguments.
\begin{theorem}[Long time existence]
    \label{introtheorem:longTimeExistenceViaProjections}
    Let $p\in (\frac 7 3, \frac 8 3)$ and 
    $\gamma_0  \in W^{\frac 3 2 p-2,2}(\R/\Z,\R^n)$ be injective on $[0,1)$
    and suppose that $|\g_0'|>0$ on $\R/\Z$.
    Then, there is a unique\footnote{The solution $\g(\cdot,t)$ is even
        unique among all $W^{\frac32 p-2,2}$-valued mappings that
        are absolutely continuous in time for which the differential
        equation \eqref{eq:parabolic-map} holds only for a.e. $t>0$.} 
    map 
    \begin{equation}\label{eq:parabolic-map}
    t\mapsto \g(\cdot,t)
    \in C^1\big([0,\infty),W^{\frac32 p-2,2}(\R/\Z,\R^n)\big)
    \end{equation}
    satisfying 
    \begin{equation*}
    \label{introtheorem:longTimeExistenceViaProjections:eq:ODE}
    \numberthis
    \begin{cases}
    \frac{\partial}{\partial t} \gamma(\cdot,t) &= -\nabla_{\SIGMA} \Mpq[p,2](\gamma(\cdot,t))
    \quad\textnormal{on $\R/\Z$ for all $t>0$},\\
    \gamma(\cdot,0) &= \gamma_0,
    \end{cases}
    \end{equation*}
    where $\SIGMA $ is defined as in \eqref{eq:log-constraint}.
    Moreover, the energy is non-increasing in time, and
    the initial velocity and the barycenter
   are preserved, that is,
    \begin{equation}\label{eq:velocity-preserving}
    |\g'(\cdot,t)|=|\g_0'(\cdot)|\,\,\textnormal{on $\R/\Z$ and
        $\textstyle\int_{\R/\Z}\g(u,t)\,\dd u=\int_{\R/\Z}\g_0(u)\,\dd u$ \,\,for all $t>0$}.
    \end{equation}
    In particular, the length of $\g(\cdot,t)$ is constant in time
    and the curves $\g(\cdot,t)$ remain uniformly bounded, that is,
    $\mathcal{L}\left(\g(\cdot,t)\right)=\mathcal{L}\left(\g_0(\cdot)\right)$ 
    and $\|\g(\cdot,t)\|_{W^{\frac32 p-2,2}(\R/\Z,\R^n)}\le C$ for
    all $t>0$.
\end{theorem}
The  fractional Sobolev space $W^{\frac32 p-2,2}$ continuously
embeds into the H\"older space $C^{1,\alpha(p)}$ with H\"older exponent
$\alpha(p):=
\frac32 p- \frac1p -3\in (0,1)$
(see \thref{thm:morrey} in Appendix \ref{section:Sobolev}), which implies that
all curves $\g(\cdot,t)$ are also  uniformly bounded in $C^{1,\alpha(p)}$.
Moreover, the initial embedding $\g_0$ represents  a tame knot class $\mathcal{K}:=
[\g_0]$; see, e.g. \cite[Appendix I]{crowell-fox_1977}. Tame knot
classes are stable with respect to $C^1$-perturbations (see \cite{reiter_2005, denne-sullivan_2008, blatt_2009a}), so that the knot class  is preserved
along the flow.  
\begin{corollary}[Flow within prescribed knot class]
    The solution curves $\g(\cdot,t)$ in 
    \thref{introtheorem:longTimeExistenceViaProjections} 
    are injective on $[0,1)$ for all $t\ge 0$.
    If the initial embedding $\g_0$ represents the tame knot class $\mathcal{K}$, 
    i.e., $[\g_0]=\mathcal{K}$, then $[\g(\cdot,t)]=\mathcal{K}$ for all
    $t>0$. 
\end{corollary}
In addition, due to the embedding into $C^{1,\alpha(p)}$, the regularity
assumption $|\g_0'|>0$ is meant to hold pointwise 
everywhere on $\R/\Z$, which by
continuity and periodicity implies that  
\begin{equation}\label{eq:minimal-velocity-intro}
v_{\g_0}:=\min_{\R/\Z}|\g_0'| >0,
\end{equation}
and one can analyze how this minimal initial velocity $v_{\g_0}$
enters
the  $t$-dependence of the solution. 
\begin{theorem}[Regularity in time]\label{introtheorem:regularity}
    The solution
    $t \mapsto \gamma(\cdot,t)$ of \eqref{introtheorem:longTimeExistenceViaProjections:eq:ODE}
    in \thref{introtheorem:longTimeExistenceViaProjections}
    is of class $
    C^{1,1}\big([0,\infty),W^{\frac 3 2 p -2,2}(\R/\Z,\R^n)\big)
    $, and the Lipschitz constants for the mappings
    $t \mapsto \gamma(\cdot,t)$ and 
    $t \mapsto \frac {\partial  \gamma} {\partial t} (\cdot,t)$ 
    depend only on $n$ and
    $p$,  non-decreasingly on 
    the initial energy $\Mpq[p,2](\gamma_0)$ and 
    initial fractional seminorm
    $\halfnorm[\frac 3 2 p -3,2]{\gamma_0'}$, and non-increasingly
    on the minimal velocity
    $v_{\g_0}$.
\end{theorem}
At this point it is not clear yet whether the solution curves $\g(\cdot,t)$
converge to a projected critical point of $\Mpq[p,2]$ as $t\to\infty$. 
We do know, however,
that we have subconvergence of $\g(\cdot,t_k)$ to some limiting knot
$\g^*$  in the same fixed knot class for any sequence $t_k\to\infty$
as $k\to\infty.$
\begin{corollary}\label{cor:subconvergence}
    For the solution $\g(\cdot,t)$ of 
    \eqref{introtheorem:longTimeExistenceViaProjections:eq:ODE}
    in \thref{introtheorem:longTimeExistenceViaProjections}
    one has
    \begin{equation}\label{eq:limit-proj-gradient}
    E_\infty:=\lim_{t\to\infty}\Mpq[p,2](\g(\cdot,t))\in [0,
    \Mpq[p,2](\g_0)],\quad\lim_{t\to\infty}\nabla_\SIGMA\Mpq[p,2](\g(\cdot,t))=0,
    \end{equation}
    and for every sequence $t_k\to\infty$ as $k\to\infty$
    there is a subsequence
    $(t_{k_l})_l\subset (t_k)_k$ and a curve $\g^*\in W^{\frac32 p-2,2}_{\ir}
    (\R/\Z,\R^n)$ with $\Mpq[p,2](\g^*)\le E_\infty
    $ and with knot class $[\g^*]=[\g_0]$, such that
    $\g(\cdot,t_{k_l})$ converges weakly in $W^{\frac32 p-2,2}$ and
    strongly in $C^1$ to $\g^*$ as $l\to\infty.$
\end{corollary}
Analytical results regarding gradient flows for non-local 
self-avoidance energies seem scarce.  The most significant
contributions are the long time existence results of Blatt for
the $L^2$-flow of 
O'Hara's knot energies including the M\"obius energy 
\cite{blatt_2012b,blatt_2020a,blatt_2018a}. 
There is also  work on the $L^2$-flows of  linear combinations of
the classic bending energy and a (non-local)
self-avoidance term \cite{lin-schwetlick_2010}, \cite{vonbrecht-blair_2017},
but there the crucial
a priori estimates are obtained by virtue of the leading order curvature
energy.  
For a linear combination of a discrete bending and 
self-repulsive tangent-point-type energy, however, numerical analysis has been
performed recently by S. Bartels, Reiter, 
and J. Riege \cite{bartels-etal_2018,
    bartels-reiter_2018}. In his Ph.D. thesis, Hermes \cite{hermes_2014}
implemented a
numerical scheme for an $L^2$-type gradient
flow   of
integral Menger curvature $\mathcal{M}_p$
which was supplemented later
by A. Gilsbach \cite{gilsbach_2018}
to numerically minimize  under symmetry constraints, but there is
no analytical foundation yet for that flow with or without 
symmetry enforcement. The situation seems similar
for the publicly available numerical flows such as
R. Scharein's KnotPlot \cite{knotplot_2017} or 
Ridgerunner \cite{ashton-etal_2011} that are widely used
in the community of geometric or applied knot theory. 
Fractional Sobolev metrics for the Möbius energy have already been investigated by Reiter and the second author in \cite{2005.07448}. However, because the attendant energy space is not a Hilbert space, long time existence could not be shown there.
Similar metrics for tangent-point energies of curves have been discussed from a more experimental point of view by {Ch.}~Yu, K.~Crane, and the second author in \cite{2006.07859}. There are also more advanced discretization and solving strategies being developed to make the evaluation of the discrete energy and the discrete gradient more efficient. Some of the methods discussed there may also be applied to the integral Menger energy, but this shall not be our concern here.

Sobolev gradient flows are also easier to discretize than $L^2$ gradient flows.
The problem with the $L^2$-flow of an energy $E$ is that its governing equation is a parabolic partial differential equation and that certain \emph{Courant--Friedrichs--Lewy} conditions have to be fulfilled in order to make its explicit space-time discretization stable. Typically, these conditions are of the form $\tau \lesssim h^\alpha$ where $\tau >0$ denotes the time step size and $h>0$ denotes the spatial mesh size (here: the longest edge in a polyhedral discretization), and $\alpha >0$ is the order of the differential operator $\g \mapsto DE(\g)$. In our case, we have $\alpha = 2 s > 2$, hence a stable explicit time integration would require prohibitively small step size $\tau \lesssim h^{2s}$ for small $h>0$.
This can be mended by using implicit time integration, but this requires the second derivative of the energy and it introduces other costs like having to solve a nonlinear equation in each gradient step.
Instead, the discrete Sobolev flow comes without dependence of the step size $\tau$ on the mesh size $h$, allowing us to use inexpensive explicit time integrators and adaptive time stepping strategies. 
The latter provides sufficient stability so that the discrete flow with oversized time steps can be used as fairly efficient optimization routine. 
Even with a naive all-pairs discretization of the energy and dense matrix arithmetic, this allowed us to compute local discrete minimizers in \autoref{fig:Flow_HermesKnot}, \autoref{fig:Flow_SquareKnot} and \autoref{fig:Flow_TorusKnot} within less than 20 minutes on a consumer laptop. 
For the algorithmic details see Section~\ref{sec:Numerics}. There we also briefly explain how we employ refined time stepping techniques in order to guarantee that the discrete flow preserves the knot class.
This is necessary because i) the discrete energy is not a knot energy (because we use inexact quadrature rules) and because ii) \emph{finite} time steps may lead to pull-through even for true knot energies.

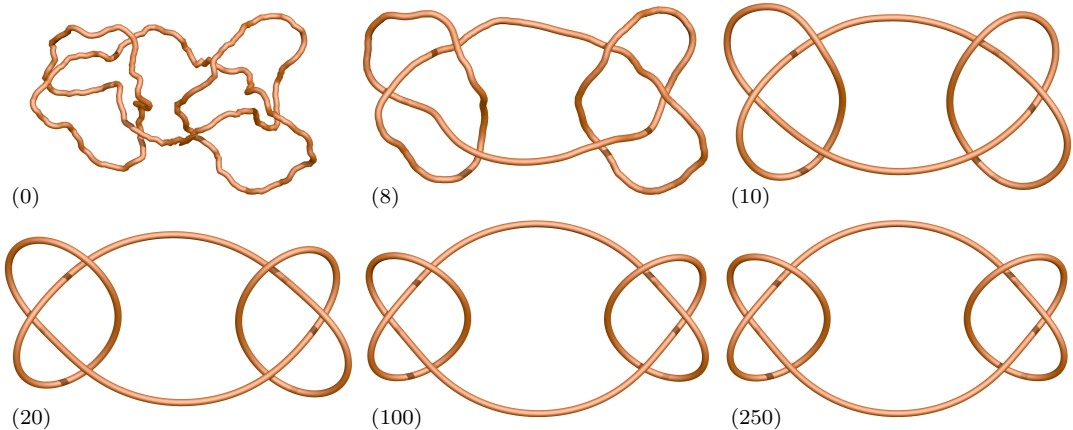
\begin{figure}
    \begin{center}
        \newcommand{\myincludegraphics}[2]{\begin{tikzpicture}
            \node[inner sep=0pt] (fig) at (0,0) {\includegraphics{#1}};
            \node[above = -0.15ex , right= 0.01ex] at (fig.south west) {\footnotesize(#2)};    
            \end{tikzpicture}
        }%
        \presetkeys{Gin}{
            trim = 100 50 100 55, 
            clip = true,  
            angle = 0,
            width = 0.32\textwidth
        }{}%
        \capstart%
        \myincludegraphics{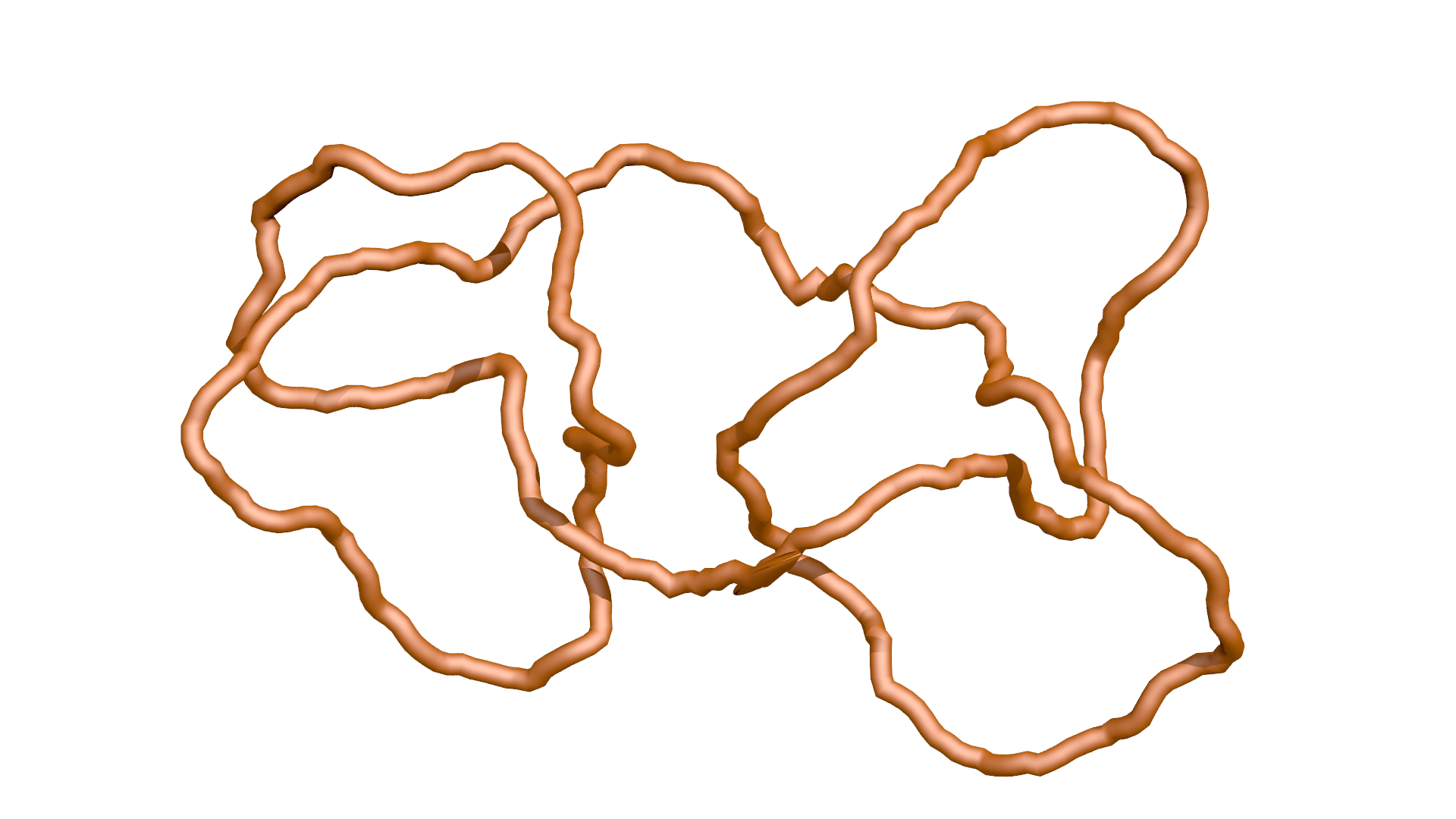}{0}%
        \myincludegraphics{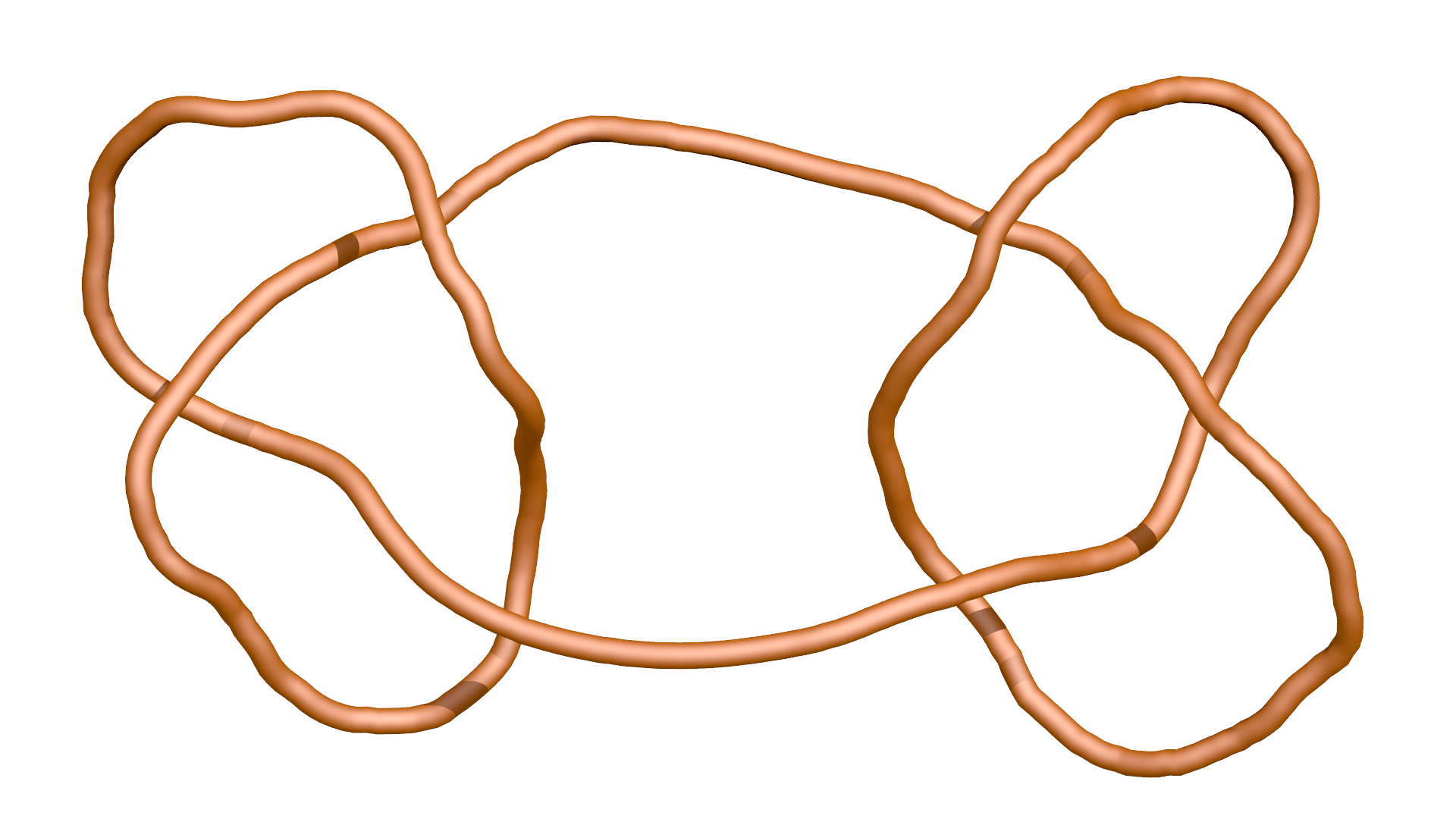}{8}%
        \myincludegraphics{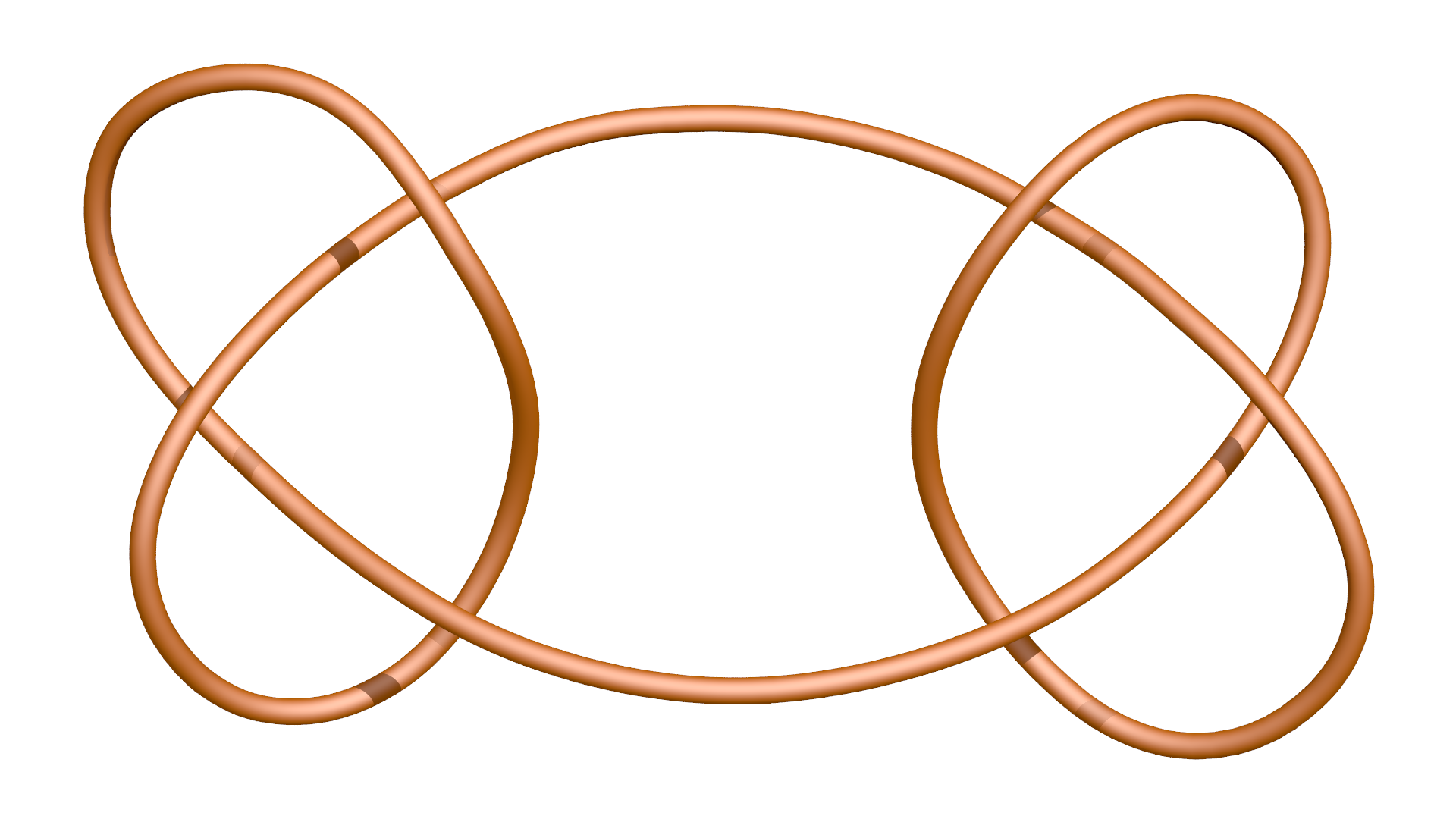}{10}%
        \\
        \myincludegraphics{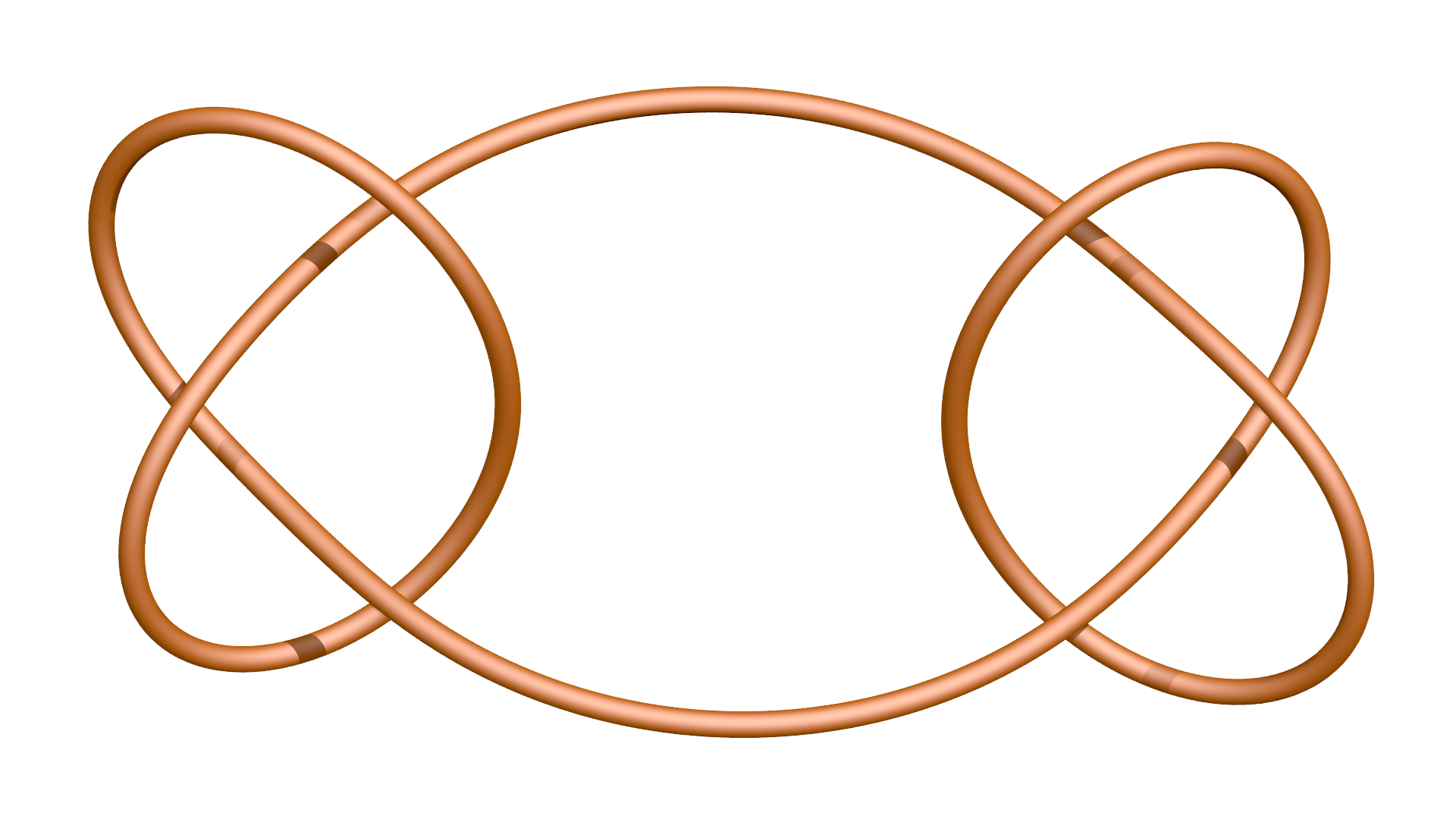}{20}%
        \myincludegraphics{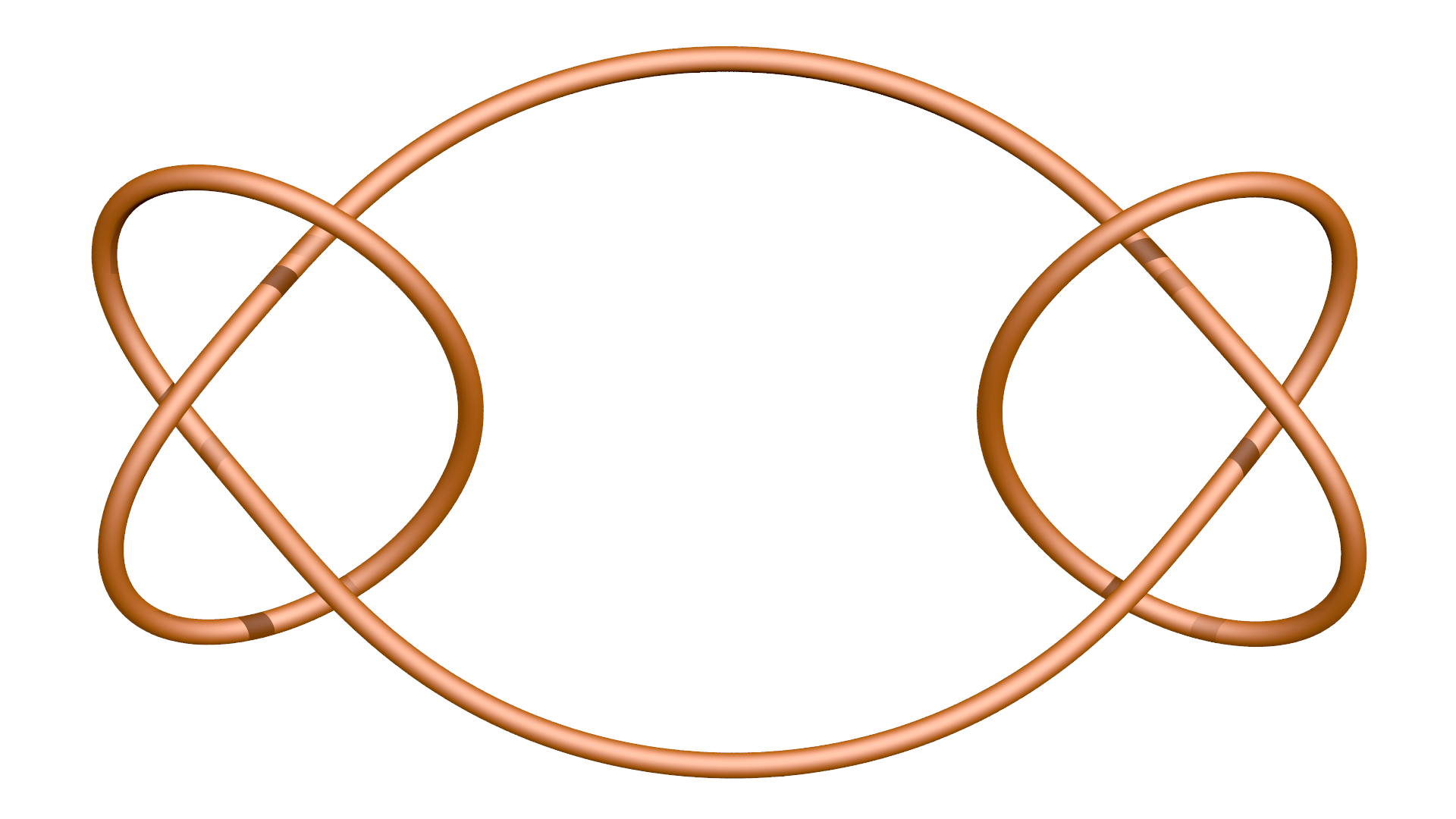}{100}%
        \myincludegraphics{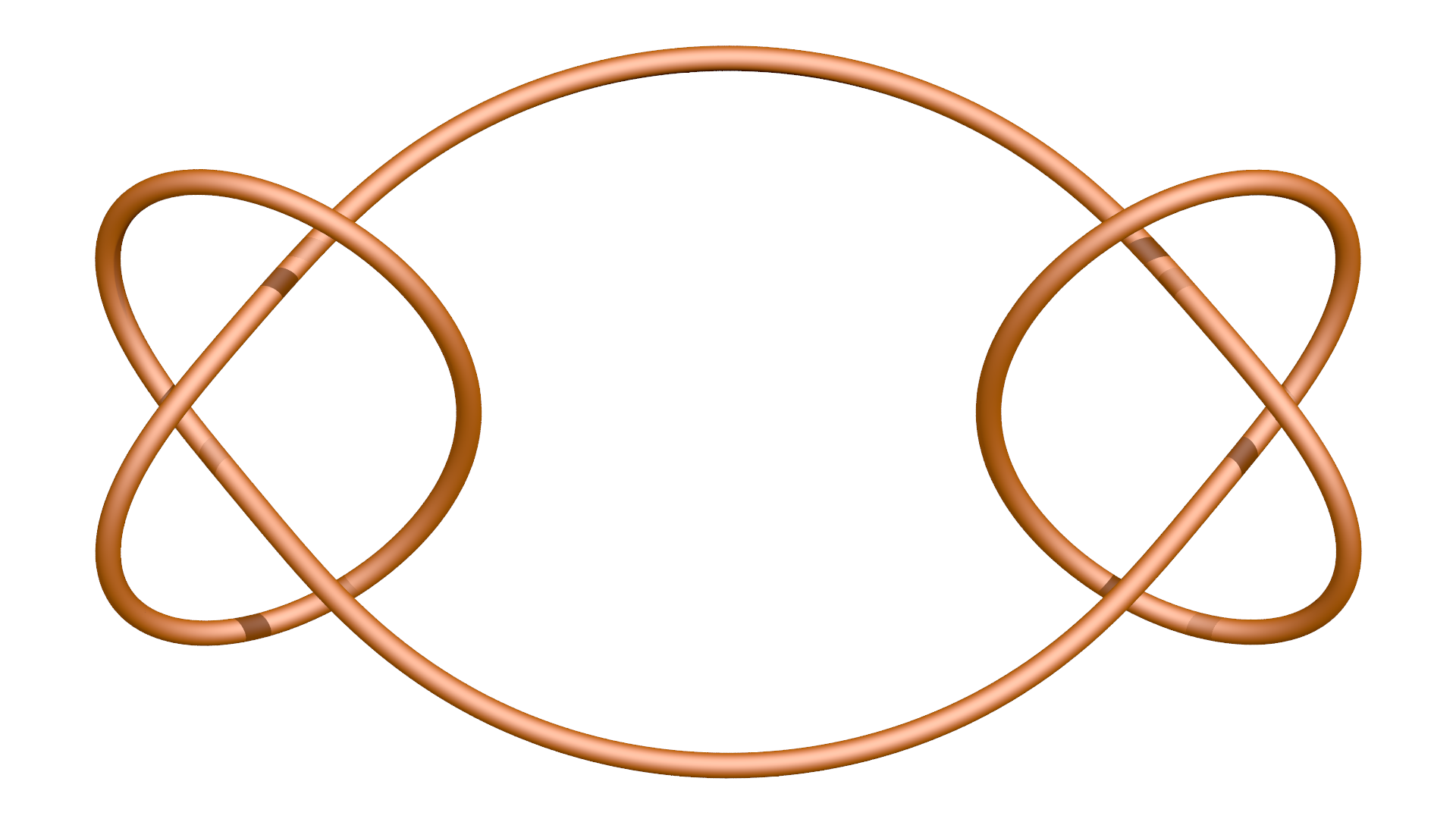}{250}%
    \end{center}
    \caption{The Sobolev gradient flow has also regularizing properties. Here a noisy, polyhedral square knot (connected sum of a left-handed and a right-handed trefoil) of 600 edges is chosen as initial configuration. (Again, numbers in parentheses indicate iteration numbers of the projected gradient descent.)}
    \label{fig:Flow_SquareKnot}
\end{figure}

\begin{figure}
    \begin{center}
        \newcommand{\myincludegraphics}[2]{\begin{tikzpicture}
            \node[inner sep=0pt] (fig) at (0,0) {\includegraphics{#1}};
            \node[above right= 0.125ex] at (fig.south west) {\footnotesize#2};    
            \end{tikzpicture}
        }%
        \presetkeys{Gin}{
            trim = 45 45 45 45, 
            clip = true,  
            angle = 0,
            width = 0.242\textwidth
        }{}%
        \capstart%
        \myincludegraphics{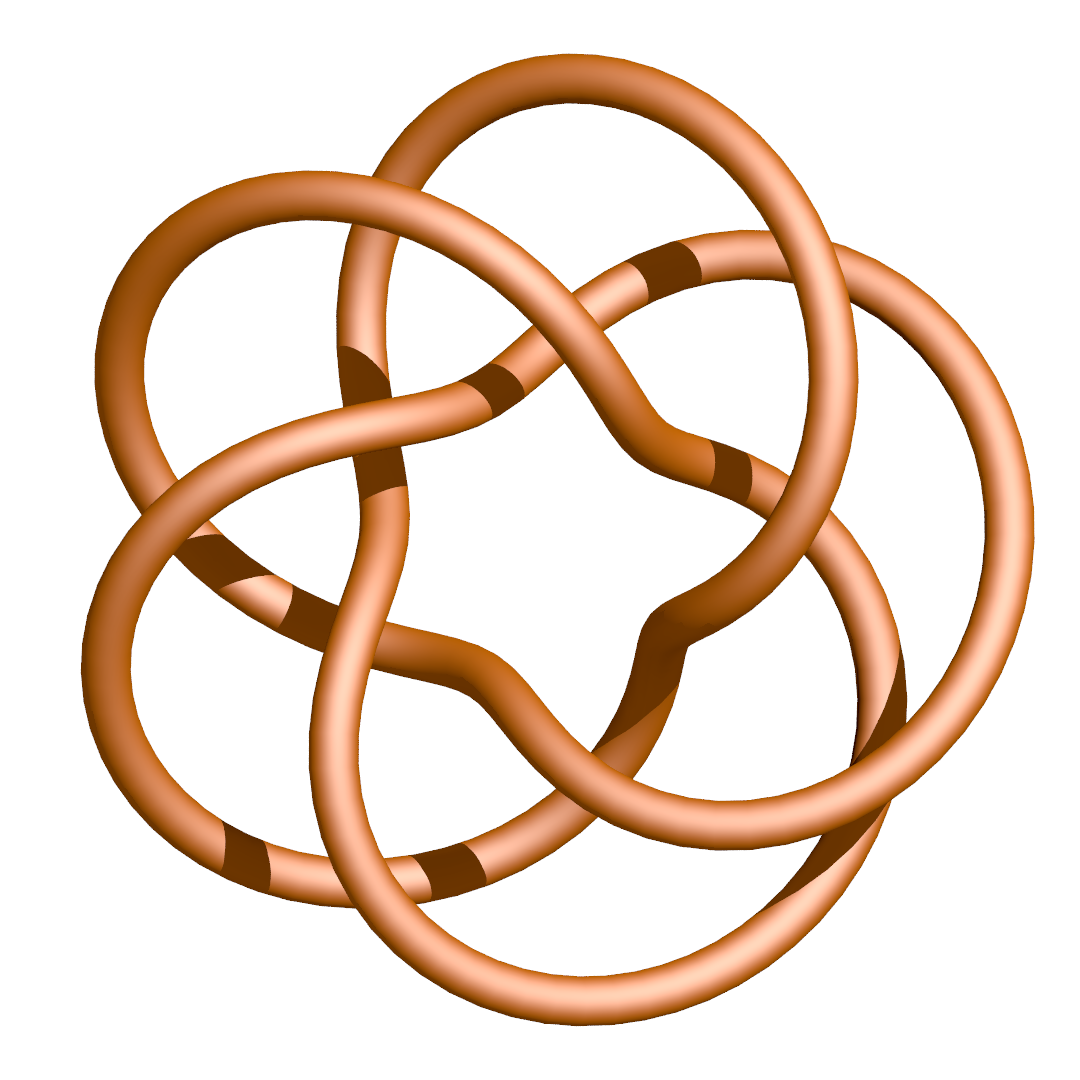}{(a)}%
        \myincludegraphics{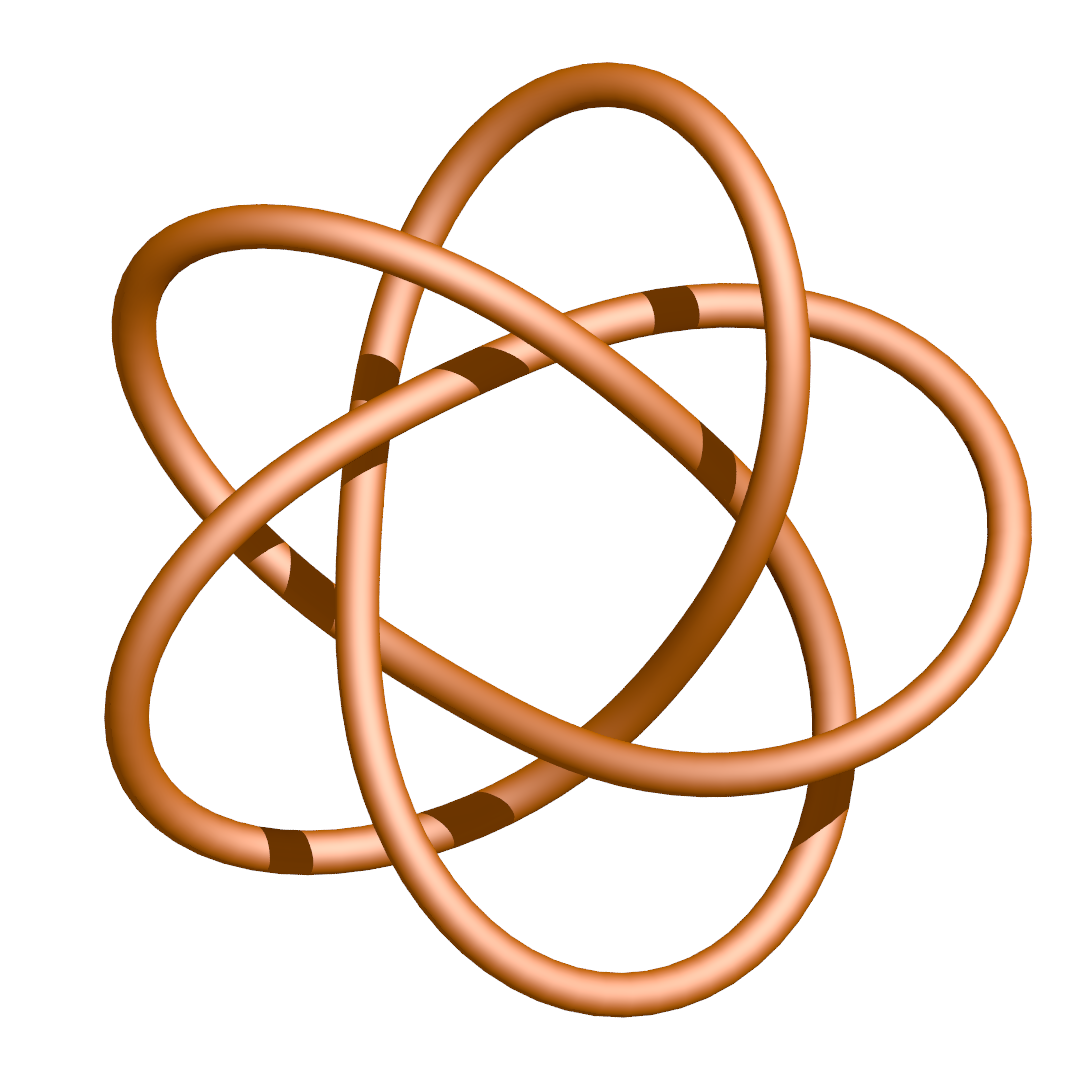}{(b)}%
        \myincludegraphics{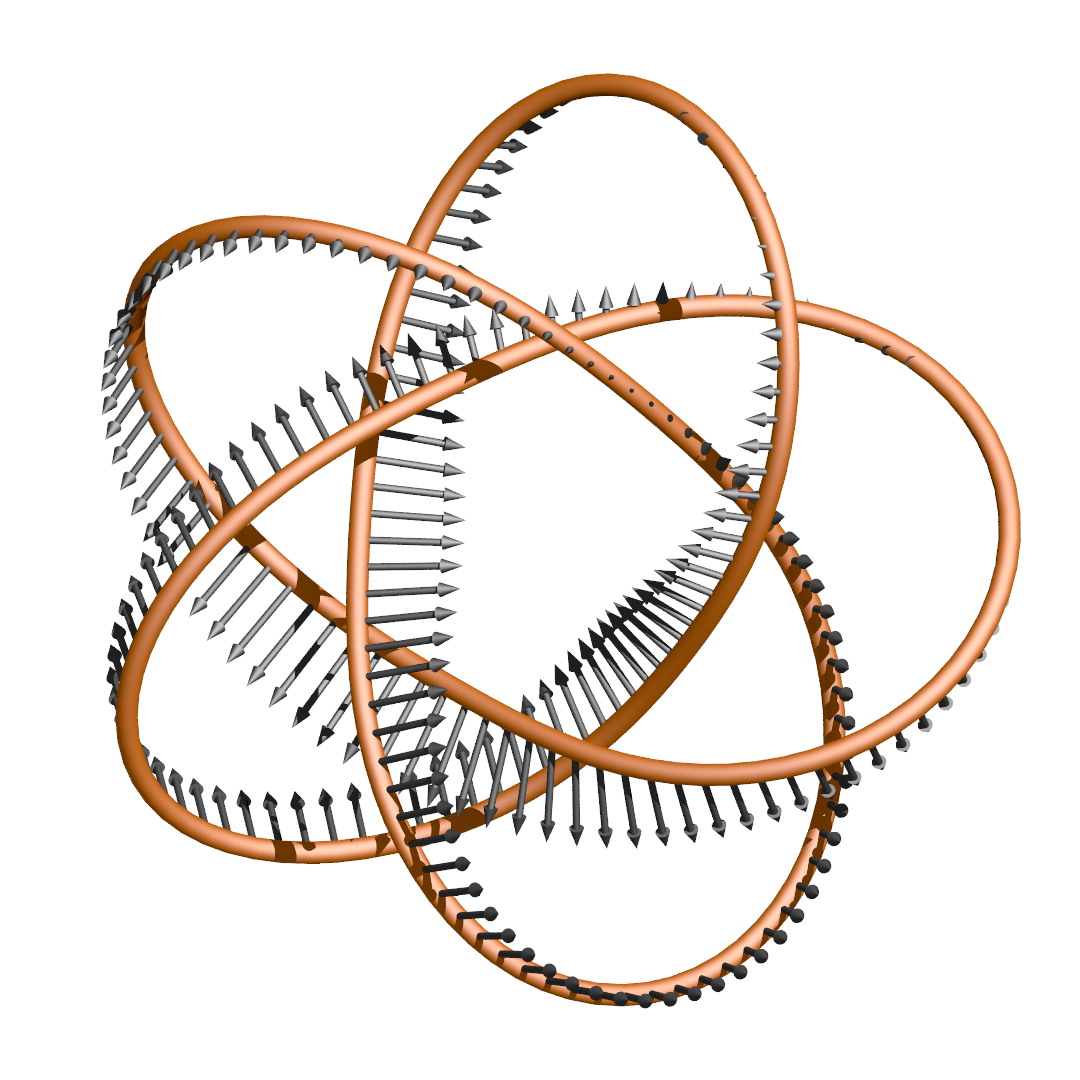}{(c)}%
        \myincludegraphics{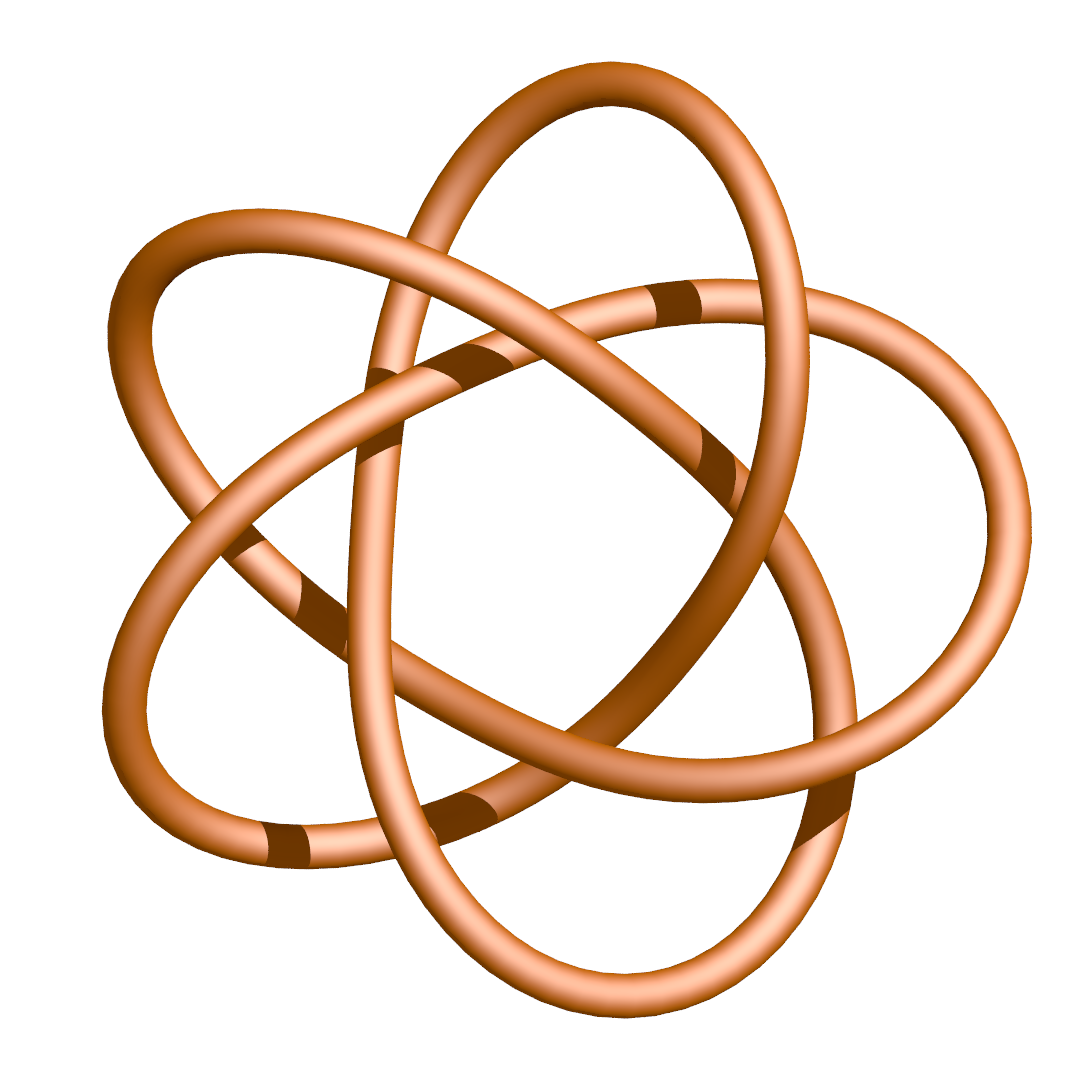}{(d)}%
        \\%
        \myincludegraphics{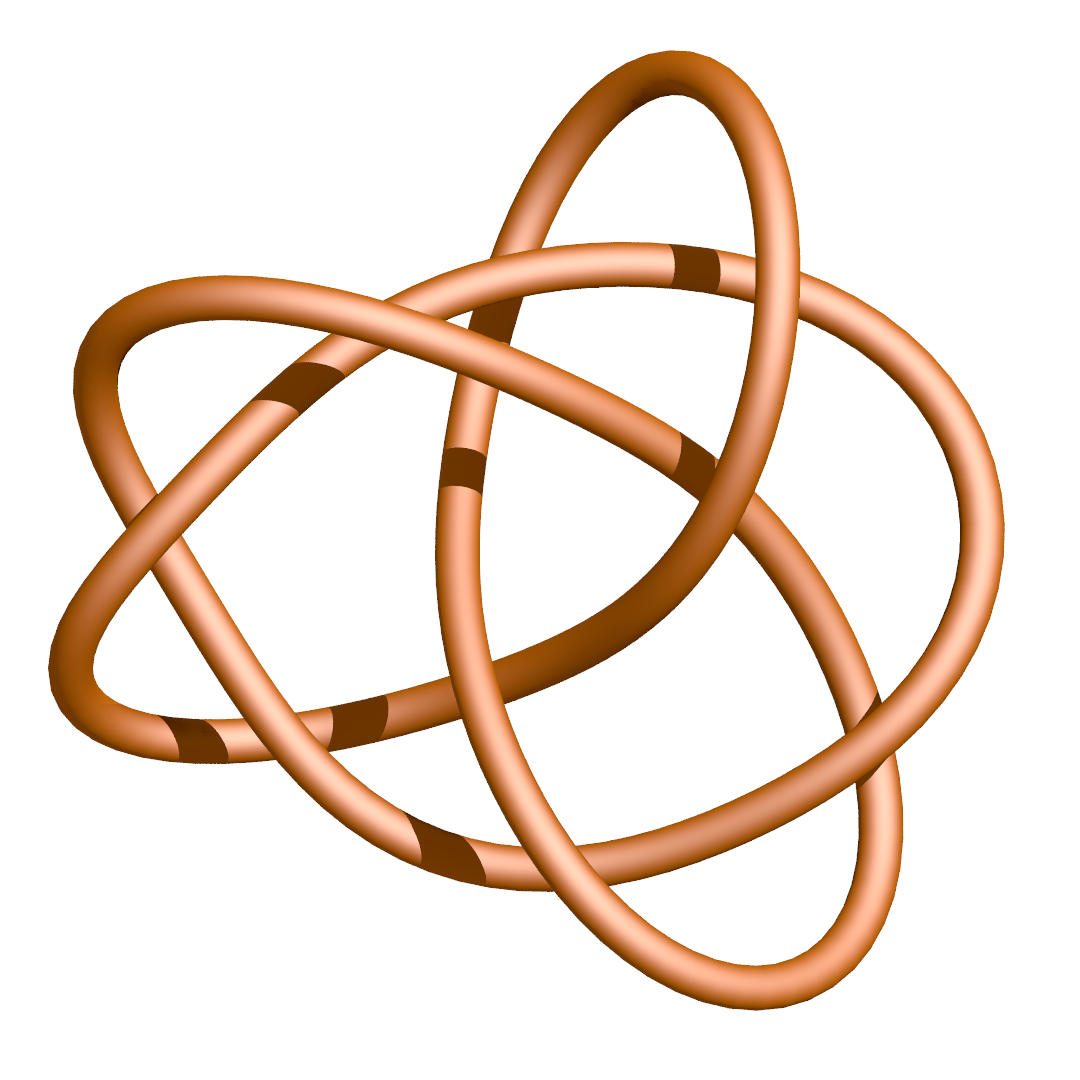}{(e)}%
        \myincludegraphics{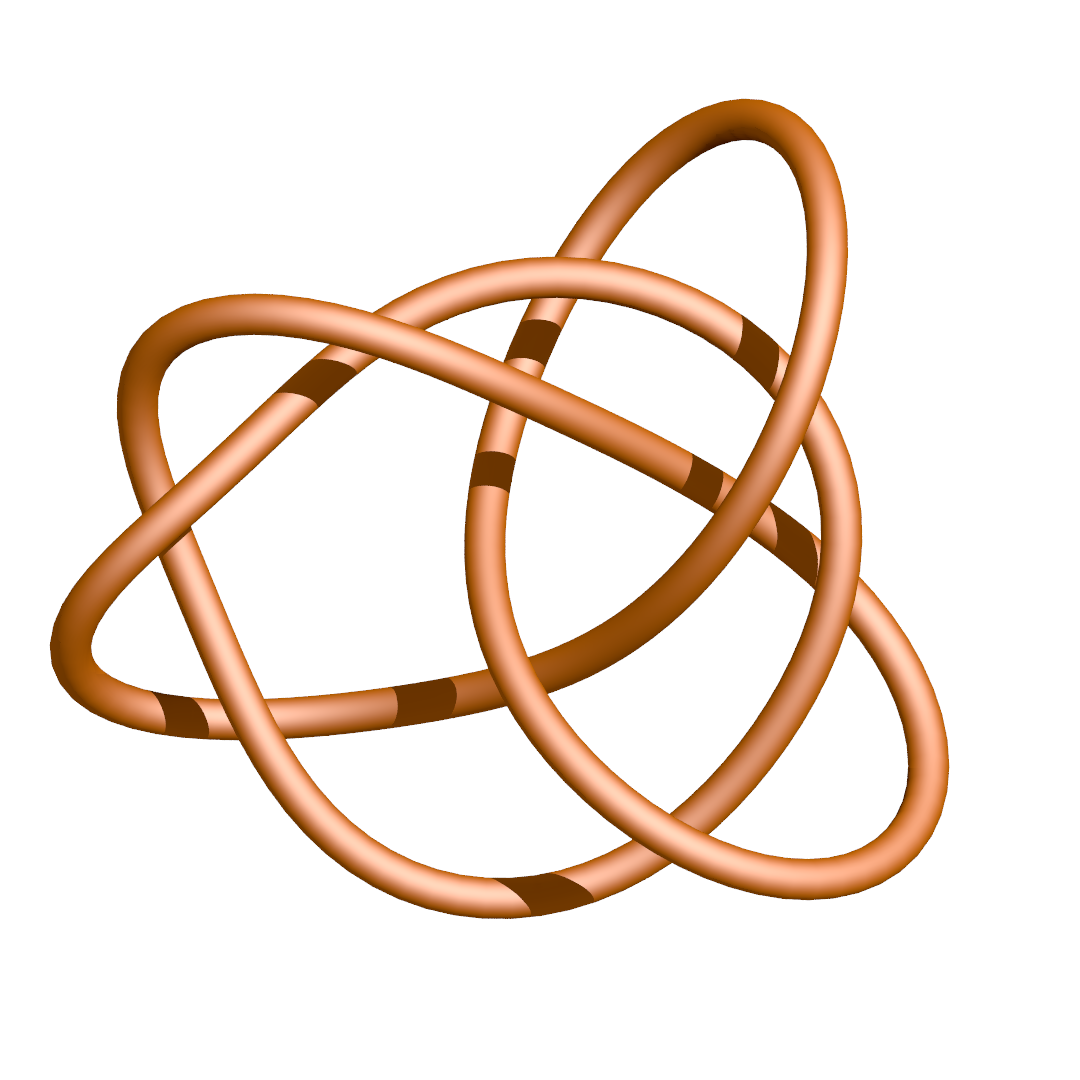}{(f)}%
        \myincludegraphics{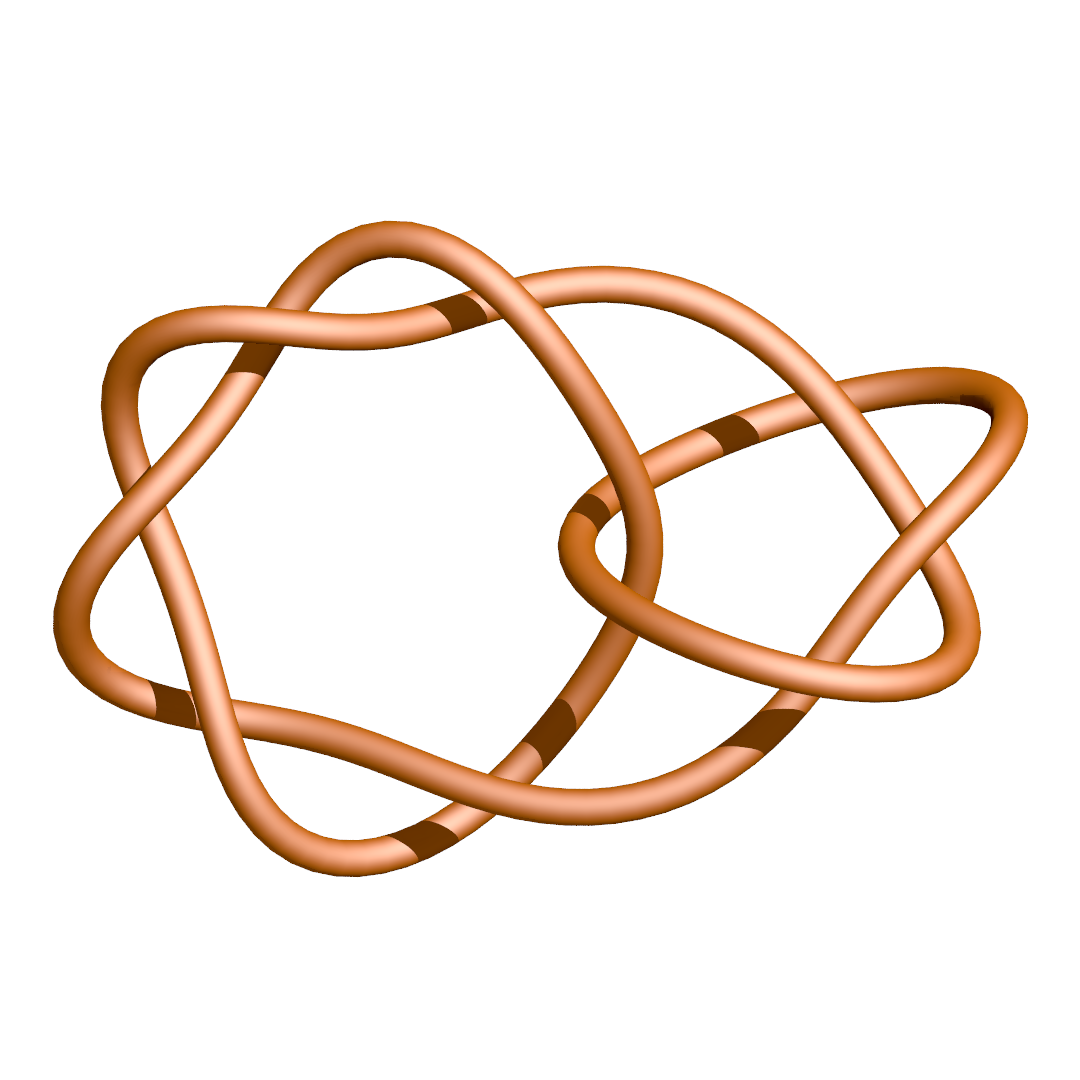}{(g)}%
        \myincludegraphics{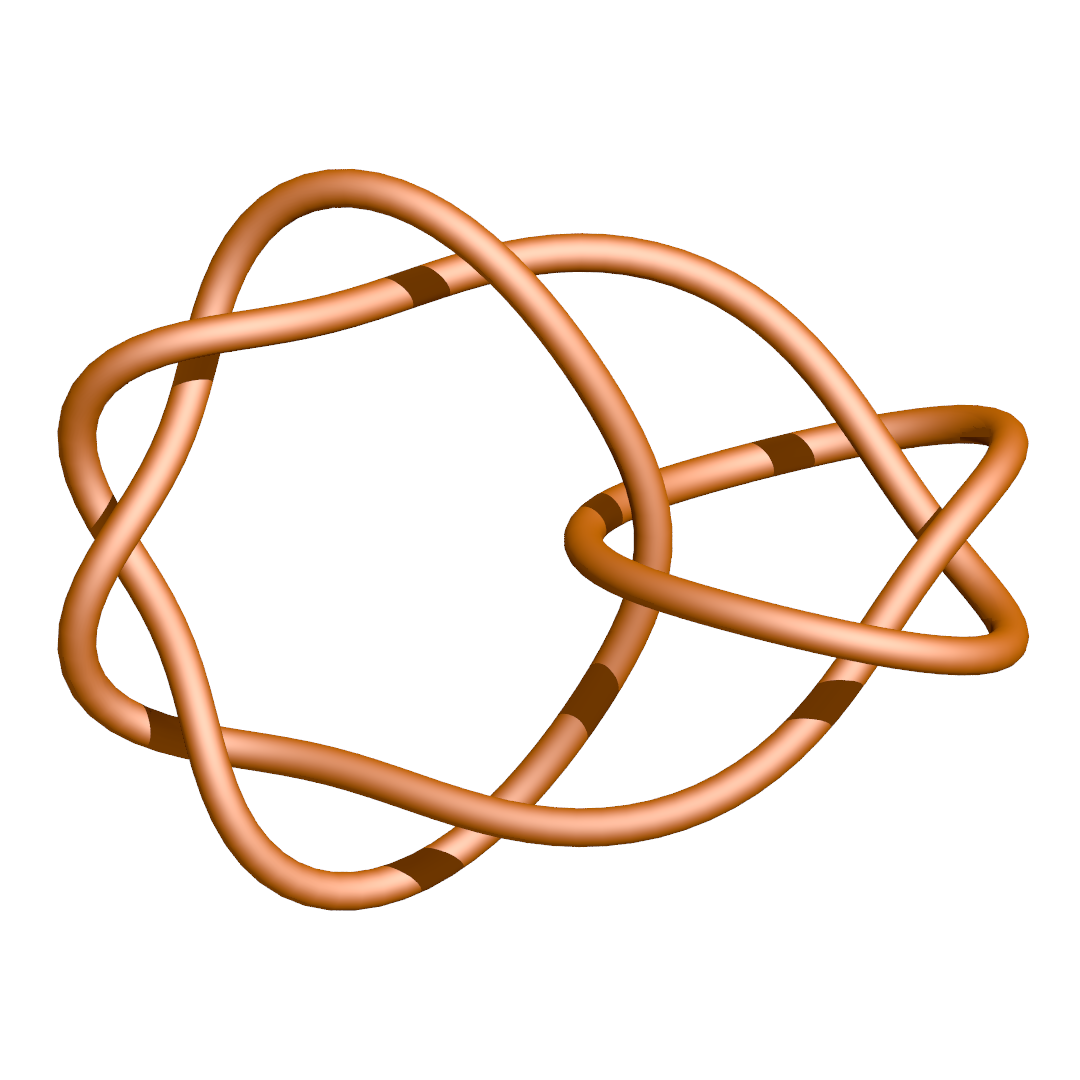}{(h)}%
    \end{center}
    \caption{Spontaneous symmetry breaking:
    (a) Symmetric 5-3 torus knot. (b) Symmetric critical point after 40 projected Sobolev gradient iterations. (c) The eigenvector to the smallest negative(!) eigenvalue of the constraint Hessian. (d) Small, almost invisible perturbation of (b) in direction of this eigenvector.
    (e) Continuing the projected Sobolev gradient flow for 100, (f) 200, and (g) 300 iterations. 
	From there on the energy landscape becomes rather shallow, so we applied Newton's method to obtain the local minimizer (h).
    }
    \label{fig:Flow_TorusKnot}
\end{figure}
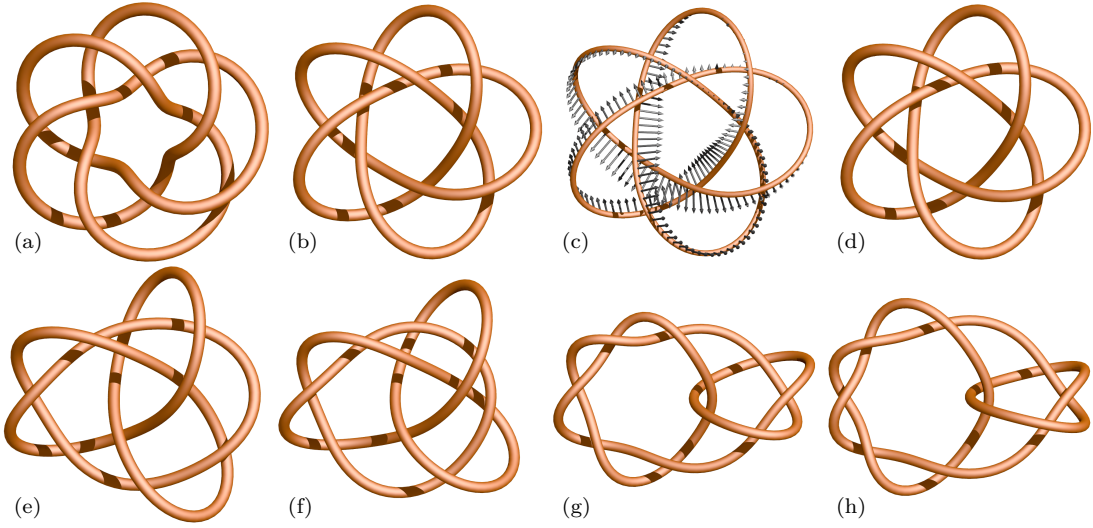

We close this introduction with a few remarks concerning notation.
The inner product in 
$\R^n$ is denoted  by $\langle\cdot,\cdot\rangle$, inner products on
other Hilbert spaces $\HL$ usually carry an index, i.e., $\langle\cdot,\cdot
\rangle_{\HL}$. 
Quite often we need to know how constants depend on certain parameters, so 
constants $C=C(a,b,c,\ldots)$ are  frequently interpreted as
functions depending monotonically on the parameters $a,b,c\ldots$.
By 
$\LIP_f$ we denote Lipschitz constants of  mappings $f$.
Closed curves are  $\ell$-periodic vector-valued
functions $\gamma$ on $\R$ for some $\ell>0$, and we simply
write $\gamma \colon \R/\ell\Z \to \R^n$. The \emph{bilipschitz constant} 
of $\g$ defined as
\begin{equation}\label{eq:bilip-def}
\BiLip(\gamma):=
\inf_{\substack{u_1,u_2 \in \R/\ell\Z \\ u_1\neq u_2}} \frac {\abs{\gamma(u_1)-\gamma(u_2)}} {\abs{u_1-u_2}_{\R/\ell\Z}} 
\end{equation}
is essentially 
the inverse of Gromov's distortion \cite[Sect. 9]{gromov_1983} and
describes the embeddedness of $\g$ in a quantitative way.  
As in \eqref{eq:minimal-velocity-intro} we write $v_\g$ for the 
minimal velocity of any absolutely continuous curve $\g$. 
Finally,   let  
$W^{1,1}_{\ir}(\R/\ell\Z)$ be the subset of closed
$W^{1,1}$-curves
$\gamma$ that are injective on $[0,\ell)$
and \newterm{regular}, i.e.\ with $v_\gamma >0$ a.e.\ on $\R/\ell\Z$.

\section{Lipschitz continuity of the energy gradient
}\label{sec:2}
We calculate and estimate the first and second variation
of the integral Menger curvature
on the space $W^{\frac32 p-2,2}_\textnormal{ir}(\R/\Z,\R^n)$ of  periodic regular 
Sobolev Slobodecki\v{i} curves that are injective on the fundamental 
domain $[0,1)$.  
\begin{theorem}
    \label{theorem:secondVariationMp2} \index{generalized integral Menger curvature}
    For $p\in(\frac 7 3,\frac 8 3)$ the first variation 
    $\delta\Mpq[p,2](\g,\cdot)$ and 
    the
    second variation $\delta^2\Mpq[p,2](\g)$ of the generalized integral Menger
    curvature $\Mpq[p,2]$ exist at every curve $\g\in
    W^{\frac 3 2 p -2,2}_{\ir}(\R/\Z,\R^n)$, satisfying  the estimates 
    \begin{align}
    \label{theorem:secondVariationMp2:eq:estFirstVariation}
   | {\delta \Mpq[p,2](\gamma,h)}|
    &\leq C_1\halfnorm[\frac32 p-3,2]{h'},
    \\
    \label{theorem:secondVariationMp2:eq:estSecondVariation}
   |{\delta^2\intM^{(p,2)}(\gamma)[h,g]}|
    &\leq C_2 \halfnorm[\frac32 p-3,2]{h'} \halfnorm[\frac32 p-3,2]{g'},
    \end{align}
    for all $h,g\in W^{\frac32 p-2,2}(\R/\Z,\R^n)$, where the constants
    $
    C_i = C_i(n, p,\BiLip(\gamma),\halfnorm[\frac32 p-3,2]{\gamma'})$ for
    $ i=1,2 $
        depend non-increasingly on the bilipschitz constant $\BiLip(\gamma)$ of $\g$ and non-decreasingly on the seminorm $\halfnorm[\frac32 p-3,2]{\gamma'}$ of the tangent
    $\g'$.
\end{theorem}
\begin{proof}
    Since $\gamma$ is  injective on $[0,1)$ and  a regular curve of class
    $W^{\frac32 p-2,2}$ which by Morrey's embedding,
    \thref{thm:morrey}, embeds into $C^1$, we
    have $0<\BiLip(\g) \leq v_\g\le\abs{\gamma'(u)}$    for all $u\in\R;$ see  
    \thref{lem:tangents-close} and \thref{cor:bilip2}.
    Using the explicit embedding inequality \eqref{eq:morrey-embedding}
    we obtain
    $
    \halfnorm[\frac32 p-3,2]{\g'}
    {\ge} C_E^{-1}\|\g'\|_{C^0(\R/\Z,\R^n)}\ge 
    C_E^{-1}v_\g>0.
    $
    Therefore, we can 
    choose $\epsilon>0$ sufficiently small such that
    \begin{equation}\label{eq:eta-cond}
    \halfnorm[\frac32 p-3,2]{\eta'}\le 2\halfnorm[\frac32 p-3,2]{\g'},
    \AND v_\eta
   {\ge}
    \BiLip(\eta)
    {\ge}
     \BiLip(\g)/2>0
    \end{equation}
    for all $\eta\in B_{\epsilon}(\g)
    \subset W^{\frac32 p -2,2}(\R/\Z,\R^n),
    $
    where we used \eqref{eq:bilip-comparison} of 
    \thref{lemma:arclengthParametrizationAndBiLipschitzConstants}
    as well as \eqref{eq:bilip2} of
    \thref{cor:bilip2} in Appendix \ref{section:Arclength}. 
    Fix such a curve
    $\eta\in B_\epsilon(\g)$.
        Similarly as in \cite{BlattReiter:2015:Menger}, we use the difference $
    \difference{v,w}f(u) := f(u+v)-f(u+w) $
    for a function $f\colon \R\to\R$ and arbitrary parameters $u,v,w\in\R$ to
    abbreviate 
    \begin{align*}
    \hatProduct{f,g} \!&:=\! \difference{w,0} f(u) \wedge \difference{v,0} g(u) + 
    \difference{w,0} g(u) \wedge \difference{v,0} f(u)
    \AND\!
    \hat f \!:=\! \hatProduct{f,f/2} \!=\! \difference{w,0}f(u) 
    \wedge \difference{v,0} f(u),
    \end{align*}
    where also $g\colon \R\to\R$ is an arbitrary function.
    This allows us to rewrite the Lagrangian of $\Mpq[p,2](\eta)$ in 
    \eqref{eq:gen-integral-menger} after substituting the variables
    $u:=u_1$, $u+v:=u_2,$ and $u+w:=u_3$ as
    \begin{align*}
    L(\eta)(u,v,w) &:=  {\abs{\hat \eta}^2} 
   ( {\abs{\difference{w,0}\eta(u)} \abs{\difference{v,0}\eta(u)} 
        \abs{\difference{v,w}\eta(u)}})^{-p} \abs{\eta'(u)} \abs{\eta'(u+v)} 
    \abs{\eta'(u+w)}.
    \end{align*}
    With the additional notation $
    A[f,g,h](x)
    :=\textstyle\big\langle \frac {g'(u+x)} {\abs{f'(u+x)}}, \frac {h'(u+x)} {\abs{f'(u+x)}}\big\rangle,$ and $
    B(f,g,h)(x,y):=\textstyle\big\langle \frac {\difference{x,y}g} {\abs{\difference{x,y}f}}, \frac {\difference{x,y}h} {\abs{\difference{x,y}f}} \big\rangle $
    for differentiable functions $f,g,h\colon \R\to\R$
    we can express the integrand of the first variation of $\Mpq[p,2]$
    at $\eta\in W^{\frac32 p-2,2}_\textnormal{ir}(\R/\Z,\R^n)$
    in the direction of $h\in W^{\frac32 p-2,2}(\R/\Z,\R^n)$ derived in 
    \cite[Section 3]{BlattReiter:2015:Menger}\footnote{Note that the last three 
        terms ought to have the factor $(R^{p,q})^{-1}$ instead of $R^{p,q}$ in the 
        expression for $\delta L$ preceding Lemma 3.1 in 
        \cite{BlattReiter:2015:Menger}.}  as
    \begin{multline*}
    \delta L(\eta,h)(u,v,w)
    :=L(\eta)(u,v,w)\cdot \big[ {2}
    {\langle \hat\eta,\hatProduct{\eta,h}\rangle/{\abs{\hat \eta}^2}}  
    + A[\eta,\eta,h](0)
    + A[\eta,\eta,h](v) \\
    + A[\eta,\eta,h](w) 
    - p\,\big\{B[\eta,\eta,h](w,0) 
    + B[\eta,\eta,h](v,0) + B[\eta,\eta,h](v,w)\big\}\big] .
    \end{multline*}
    Denoting by $\AL[\eta,h]$  the sum of all the 
    $A[\cdot,\cdot,\cdot](\cdot)$-terms, and by $\BL[\eta,h]$ the sum of all 
    the $B[\cdot,\cdot,\cdot](\cdot,\cdot)$-terms, we 
    arrive at the shorter expression
    \begin{align}
    \delta L(\eta,h)(u,v,w) &= L(\eta)(u,v,w)
     \cdot
     \big[\AL[\eta,h]
    -p\BL[\eta,h] +  {2\langle \hat\eta,\hatProduct{\eta,h}\rangle } 
    /{\abs{\hat \eta}^2} \big]
    =: L(\eta)(u,v,w) 
   \cdot
     I(h)\label{eq:delta1integrand}.
    \end{align}
    After replacing $\eta$ with the perturbed curve $\eta_\tau:=\eta
    +\tau g$ for some
    $g\in W^{\frac32 p-2,2}(\R/\Z,\R^n)$ and  $0<|\tau|\le\tau_g\ll 1$, so that
    $\eta_\tau\in B_{\epsilon}(\g)$ for all $|\tau|\le\tau_g$,
    we compute the
    derivative with respect to $\tau$ and evaluate at $\tau=0$ to 
    obtain by means of the product rule
    \begin{align}
    \delta^2 L(\eta,h,g)(u,v,w) :=& \textstyle
    \restrict{\frac d {d\tau }}{\tau=0} \left(\delta L(\eta_\tau,h)(u,v,w)\right)
    = L(\eta)(u,v,w)I(g)I(h)\notag\\ +\, L(\eta)(u,v,w)
    \cdot\textstyle \restrict{\frac d {d\tau}}{\tau=0} & \big[\AL(\eta_\tau,h) - p  
    \BL(\eta_\tau,h) +  {2\langle \widehat{\eta_\tau},
        \hatProduct{\eta_\tau,h}\rangle} 
	/{\abs{\widehat{\eta_\tau}}^2} 
	\big].
    \label{eq:second-lagrange}
    \end{align}
    Using the observation
    $
    \restrict{\frac d {d\tau}}{\tau=0} \big( 
    \frac {\langle a + \tau b,c\rangle}{\abs{a+\tau b}^2}\big)
    =  {\langle b,c \rangle}/ {\abs{a}^2} - 2 {\langle a, b\rangle} 
    {\langle a,c\rangle}/ {\abs{a}^4}
    $
    for $a,b,c \in \R^n$ in terms 
    $\AL[\eta_\tau,h]$ and $\BL[\eta_\tau,h]$ together with their very similar structure,
    we get for $\AL'
    := \restrict{\frac d {d\tau}}{\tau=0}\AL(\eta_\tau,h)$ and $\BL'
    := \restrict{\frac d {d\tau}}{\tau=0} \BL(\eta_\tau,h)$ the expressions
    $$ 
    \AL'
    = \textstyle\sum_{\sigma \in \menge{0,v,w}} A[\eta,h,g](\sigma) - 
    2 A[\eta,\eta,h](\sigma) A[\eta,\eta,g](\sigma)\quad\AND
    $$
    \[
    \BL'
    = \textstyle\sum_{(\sigma, \rho) \in \menge{(v,0),(w,0),(v,w)}} 
    B[\eta,h,g](\sigma,\rho) - 2 B[\eta,\eta,h](\sigma,\rho) 
    B[\eta,\eta,g](\sigma,\rho).
    \]
    For the last term in \eqref{eq:second-lagrange}
    we obtain from quotient and product rule
    \begin{equation}
   \textstyle \restrict{ \frac d {d\tau}}{\tau=0} \label{eq:delta2integrand}
    {\left< \widehat{\eta_\tau},\hatProduct{\eta_\tau,h}\right>} 
    \frac{1}{\abs{\widehat{\eta_\tau}}^2}
    =  \left\{\left< \hatProduct{h,\eta}, \hatProduct{g,\eta} \right> 
    + \left< \hat \eta, \hatProduct{h,g} \right>\right\} 
    \frac{1}{\abs{\hat \eta}^2} 
    -   {\left< \hat \eta, \hatProduct{h,\eta}\right> \left< \hat \eta, \hatProduct{g,\eta}\right>} \frac{2}{\abs{\hat \eta}^4}.
    \end{equation}
    In total, we have
    \begin{multline}
    \label{theorem:secondVariationMp2:eq:secondVariation}
    \delta^2 L(\eta,h,g)(u,v,w) 
    = L(\eta)(u,v,w) \big[I(g)I(h)+\AL'-p\BL'\\
    + \big\{\langle\, \hatProduct{h,\eta}, 
    \hatProduct{g,\eta}\, \rangle + \langle\,\hat \eta, 
    \hatProduct{h,g} \,\rangle\big\} {2}{\abs{\hat \eta}^{-2}} 
    -   {\langle\,\hat \eta, \hatProduct{h,\eta}\,\rangle 
        \langle\,\hat \eta, \hatProduct{g,\eta}\,\rangle} 
    {4}{\abs{\hat \eta}^{-4}}\big].
    \end{multline}
    Following the idea of Blatt and Reiter in 
    \cite[Lemma 1.1]{BlattReiter:2015:Menger} we do not consider the integral of 
    this formula over $\R/\Z\times (-1/2,1/2)^2$, 
    but instead compute six times the integral over $\R/\Z \times D$ with
   \begin{equation}\label{eq:D-def} 
    D:=\menge{(v,w) \in \left(-\tfrac 1 2,0\right) \times \left(0, 
        \tfrac 1 2\right) \colon  w \leq 1 + 2v, v\geq -1+2w}
   \end{equation} 
    which does not change the value of the functional but guarantees that $
    \abs{v}=\abs{v}_{\R/\Z},\abs{w}=\abs{w}_{\R/\Z}$, 
    and 
    \[
    \label{theorem:secondVariationMp2:eq:equivalenceDistances}
    \numberthis
     \abs{v-w}/2
    \leq \abs{v-w}_{\R/\Z}
    \leq \abs{v-w}\quad\Foa (v,w)\in D.
    \]
    For the non-obvious first inequality in 
    \eqref{theorem:secondVariationMp2:eq:equivalenceDistances} 
    notice that
    $\abs{v-w} \leq \frac 2 3$ for all $(v,w)\in D$, and therefore
   $ 
    1 - \abs{v-w} 
    = \abs{v-w} (  {\abs{v-w}}^{-1} - 1)
    \geq \abs{v-w} (  (3/ 2) -1 ),
   $
    which yields \eqref{theorem:secondVariationMp2:eq:equivalenceDistances}.
    Combining this with \eqref{eq:eta-cond} we have
    \[
    \label{theorem:secondVariationMp2:eq:PeriodicBiLip}
    \numberthis
     {\abs{\eta(x)-\eta(y)}}^{-1}
    \leq  ({\BiLip(\eta)\abs{x-y}_{\R/\Z}})^{-1}
    {\leq}  2/( {\BiLip(\eta)\abs{x-y}})
    {\le}
     4/( {\BiLip(\gamma)\abs{x-y}})
    \]
    for distinct parameters
    $x \in \menge{u+w, u+v}$, and $y \in \menge{u+v,u}$ with $(v,w) \in D$.
    Combining  
    the Morrey  embedding, \thref{thm:morrey}, in particular
    inequality \eqref{eq:morrey-embedding} with Poincar\'e's inequality
    \eqref{eq:poincare} 
    for $\eta$ (which satisfies $\int_0^1\eta'(\tau)\, \dd\tau=0$)
    leads to the estimates
    \begin{equation}\label{eq:first-order-est}
   \textstyle \frac {\abs{\eta(x)-\eta(y)}} {\abs{x-y}}\le
    \|\eta'\|_{C^0(\R/\Z,\R^n)}
    {\leq}  
    C_E \norm[W^{\frac 3 2 p -3,2}(\R/\Z,\R^n)]{\eta'}
    {\le} C_EC_P\halfnorm[\frac32 p-3,2]{\eta'}
    \end{equation}
    for all distinct $x,y\in\R.$
    Using \eqref{theorem:secondVariationMp2:eq:PeriodicBiLip},
    \eqref{eq:first-order-est}, and \eqref{eq:eta-cond}
    we can estimate the various
    terms in the integrands of the first and second variation.
    \begin{align}
    &   \textstyle \abs{L(\eta)(u,v,w)} \leq \big( \frac {2\cdot4^p C_EC_P} 
    {\BiLip(\gamma)^p} 
    \halfnorm[\frac32 p-3,2]{\gamma'} \big)^3 
    \frac {\abs{\hat \eta}^2} {\abs{v}^p \abs{w}^p \abs{v-w}^p}
    =: C_L \frac {\abs{\hat \eta}^2} {\abs{v}^p \abs{w}^p \abs{v-w}^p},\notag\\
    &   \textstyle \abs{\AL' - p \BL'} \leq \frac {144 (1+p)C_E^2C_P^2} 
    {\BiLip(\gamma)^2} 
    \halfnorm[\frac32 p-3,2]{h'} \halfnorm[\frac32 p-3,2]{g'} =: C_{\AL\BL} 
    \halfnorm[\frac32 p-3,2]{h'} \halfnorm[\frac32 p-3,2]{g'},
    \label{eq:first-nuetzlich}\\
    & \textstyle   \abs{I(h)} \leq \frac {24 (1+p)C_EC_P} {\BiLip(\gamma)} 
    \halfnorm[\frac32 p-3,2]{\gamma'} \halfnorm[\frac32 p-3,2]{h'} 
    +  {\abs{\hatProduct{h,\eta}}} \frac{2}{\abs{\hat \eta}}
    =:C_{I} \halfnorm[\frac32 p-3,2]{h'} 
    +  {\abs{\hatProduct{h,\eta}}} \frac{2}{\abs{\hat \eta}}.\notag
    \end{align}
    Here, the constants
    $C_L, C_{\AL\BL}$ and $C_{I}$ depend only on $n$ and  $p$, non-increasingly on 
    $\BiLip(\gamma)$, and non-decreasingly on $\halfnorm[\frac32 p-3,2]{\gamma'}$.
The last inequality yields $
    \abs{I(h)I(g)}
     \textstyle\leq C_I^2 \halfnorm[\frac32 p-3,2]{h'} \halfnorm[\frac32 p-3,2]{g'} 
    +  |\hatProduct{h,\eta}|   |\hatProduct{g,\eta}| 
    \frac{4}{\abs{\hat \eta}^2}
    + \frac{2C_I}{\abs{\hat \eta}}\big({\halfnorm[\frac32 p-3,2]{h'}}
    |\hatProduct{g,\eta}|  + \halfnorm[\frac32 p-3,2]{g'}  
    |\hatProduct{h,\eta}| \big).$
    Now we estimate the terms in the expression 
    \eqref{theorem:secondVariationMp2:eq:secondVariation} for $\delta^2L(\eta,h,g)$
    term by term and obtain
    \begin{align}
    \label{theorem:secondVariationMp2:eq:firstTerm}
    &    \abs{L(\eta)(u,v,w)(\AL'-p\BL')} \leq C_L C_{\AL\BL} 
    \halfnorm[\frac32 p-3,2]{h'}\halfnorm[\frac32 p-3,2]{g'} 
     {\abs{\hat \eta}^2} ({\abs{v} \abs{w} \abs{v-w}})^{-p}, \\
    & \textstyle  \big|L(\eta)(u,v,w)
        \label{theorem:secondVariationMp2:eq:secondTerm}
        \big\{\langle\, \hatProduct{\eta,h},\hatProduct{\eta,g}\,\rangle 
        +\langle\,\hat\eta,\hatProduct{h,g}\,\rangle\big\} 
        \frac{2}{\abs{\hat \eta}^2}\big| 
    \leq 2 C_L \frac {\abs{\hatProduct{\eta,h}}\abs{\hatProduct{\eta,g}} 
        + \abs{\hat \eta} \abs{\hatProduct{h,g}}} {\abs{v}^p\abs{w}^p\abs{v-w}^p}, \\
    \label{theorem:secondVariationMp2:eq:thirdTerm}
    & \textstyle \big|L(\eta)(u,v,w) {\langle\,\hat \eta,
            \hatProduct{\eta,h}\,\rangle 
            \langle\,\hat \eta,\hatProduct{\eta,g}\,\rangle} 
        {4}/{\abs{\hat \eta}^4}\big| 
    \leq 4 C_L \big|\hatProduct{h,\eta}\big|\hatProduct{g,\eta}\big| (
    {\abs{v} \abs{w} \abs{v-w}})^{-p},\\
    \label{theorem:secondVariationMp2:eq:fourthTerm}
 &   \abs{L(\eta)(u,v,w)I(h)I(g)}
    \leq C_L \big[ 
    \vphantom{\frac {\abs{\hatProduct{h,\eta}}\abs{\hatProduct{g,\eta}}} 
        {\abs{v}^p \abs{w}^p \abs{v-w}^p}}
    C_I^2 \halfnorm[\frac32 p-3,2]{h'} 
    \halfnorm[\frac32 p-3,2]{g'}  {\abs{\hat \eta}^2} 
    +4 \big|\hatProduct{h,\eta}\big|
        \big|\hatProduct{g,\eta}\big| \notag\\
     &\hspace{2.3cm} +2C_I\big\{ 
    {\halfnorm[\frac32 p-3,2]{h'}  \big|\hatProduct{g,\eta}}\abs{\hat\eta}\big|
     +  
    {	    \halfnorm[\frac32 p-3,2]{g'} 
        \big|\hatProduct{h,\eta}}\abs{\hat\eta}\big| 
    \big\} 
    \big]({\abs{v} \abs{w} \abs{v-w}})^{-p}
    .
    \end{align}
    We are going to apply these estimates to curves $\eta:=\g+s\tau g
    \in B_\epsilon(\g)$ for
    $s\in [0,1]$ and $|\tau|\le\tau_g\ll 1$. In order to sum up the
    terms on the right-hand sides of
    \eqref{theorem:secondVariationMp2:eq:firstTerm}--\eqref{theorem:secondVariationMp2:eq:fourthTerm}
    to find an integrable majorant for $\delta^2 L(\g+s\tau g,h,g)$ independent
    of $s$ and $t$, we need
    to use the bilinearity and commutativity of $\hatProduct{\cdot,\cdot}$. Indeed,
    using the bounds
    \begin{align*}
    \abs{\hat{\eta}}^2 &=\big|\,\hatProduct{\eta,\eta/2}\big|^2=\big|\hatProduct{
            \g,\g/2}+s\tau\big(\hatProduct{\g,g/2}+\hatProduct{g,\g/2}\big)+(s\tau)^2
        \hatProduct{g,g/2}\,\big|^2\\
    &=\big|\hat{\g}+s\tau\hatProduct{\g,g}+(s\tau)^2\hat{g}\big|^2
    \le 3\big(\abs{
        \hat{\g}}^2+\tau_g^2\big|\hatProduct{\g,g}\big|^2+\tau_g^4\abs{ \hat{g}}^2\big),\\
    \big|\hatProduct{\eta,h}\big| &=\big|\hatProduct{h,\eta}\big|
    =\big|\hatProduct{\g,h}+s\tau\hatProduct{g,h}\big|
    \le\big|\hatProduct{\g,h}\big|+\tau_g\big|\hatProduct{g,h}\big|
    \end{align*}
    as well as analogous bounds for $|{\hatProduct{\eta,g}}|$, all
    independent of $s\in [0,1]$ and $\tau\in [-\tau_g,\tau_g]$ 
    in \eqref{theorem:secondVariationMp2:eq:firstTerm}--\eqref{theorem:secondVariationMp2:eq:fourthTerm}
    we obtain by means of \eqref{theorem:secondVariationMp2:eq:secondVariation}
    \begin{align}\label{eq:dominant}
   & \abs{\delta^2L(\g+s\tau g)} \le \textstyle
    \frac{K}{\abs{v}^p\abs{w}^p\abs{v-w}^p}
    \big[\halfnorm[\frac32 p-3,2]{h'}
    \halfnorm[\frac32 p-3,2]{g'}\big( \abs{\hat{\g}}^2+\tau_g^2\big|
        \hatProduct{\g,g}\big|^2+\tau_g^4\abs{\hat{g}}^2\big)\notag\\
    & + \big( \abs{\hat{\g}}+\tau_g\big|
        \hatProduct{\g,g}\big|+\tau_g^2\abs{\hat{g}}\big)\big\{
    \halfnorm[\frac32 p-3,2]{h'}\big(
    \big|\hatProduct{\g,g}\big|+2\tau_g\abs{\hat{g}}\big)\\
     &+ \halfnorm[\frac32 p-3,2]{g'}\big(
    \big|\hatProduct{\g,h}\big|+\tau_g\big|\hatProduct{g,h}\big|\big)
    + \big|\hatProduct{h,g}\big|\big\}
      + \big(\big|\hatProduct{\g,h}\big|+\tau_g\big|\hatProduct{g,h}\big|\big)
    \big(\big|\hatProduct{\g,g}\big|+2\tau_g\abs{\hat{g}}\big)\big]=:G_{\tau_g},\notag
    \end{align}
    where the right-hand side defines the desired majorant $G_{\tau_g}$ 
    on $\R/\Z\times D$
    independent of $s$ and $\tau$. Notice that we have subsumed all constants
    appearing in  \eqref{theorem:secondVariationMp2:eq:firstTerm}--\eqref{theorem:secondVariationMp2:eq:fourthTerm}
    under the constant $K=K(n, p,\BiLip(\g),\halfnorm[\frac32 p-3,2]{\g'})$, 
    still depending only on $n$ and $p$,  non-increasingly
    on $\BiLip(\g)$, and non-decreasingly on $\halfnorm[\frac32 p-3,2]{\g'}$.
Applying Lemma \ref{lem:aux-est} below separately to integrated products like 
$|\hatProduct{h,\gamma}|
    |\hatProduct{g,\gamma}|$ in \eqref{eq:dominant} we obtain that
 $G_{\tau_g}$ is integrable on $\R/\Z\times D$ with the estimate 
        \begin{align}\label{eq:G-int-bound}
 \textstyle   \iiint_{\R/\Z\times D}
    G_{\tau_g}(u,v,w) \, \dd w\, \dd v\, \dd u & \le C^* K\halfnorm[\frac32 p-3,2]{h'}
    \halfnorm[\frac32 p-3,2]{g'}\big\{\halfnorm[\frac32 p-3,2]{\g'}^4+
    \halfnorm[\frac32 p-3,2]{\g'}^3\notag\\
    +\halfnorm[\frac32 p-3,2]{\g'}^2 & + \tau_g\halfnorm[\frac32 p-3,2]{g'}\big(
    \halfnorm[\frac32 p-3,2]{\g'}^2+
    \halfnorm[\frac32 p-3,2]{\g'}\big)\\
    +\tau_g^2\halfnorm[\frac32 p-3,2]{g'}^2 &  \big(
    \halfnorm[\frac32 p-3,2]{\g'}^2+\halfnorm[\frac32 p-3,2]{\g'}+1\big)+
    \tau_g^4
    \halfnorm[\frac32 p-3,2]{g'}^4\big\} ,\notag
    \end{align}
    where we have also used Young's inequality, and where we have subsumed 
    all numerical constants as
    well as the Morrey and Poincar\'e constants $C_E$ and $C_P$ under the a new
    constant $C^*\ge 1$.
    We obtain for the difference quotients  $\tau^{-1}\big( {\delta \Mpq[p,2] (\gamma+\tau g,h) - \delta \Mpq[p,2](\gamma,h)}\big)$ the expressions
       \begin{align}
       =& \textstyle 6 \int_{\R/\Z} \iint_D \tau^{-1}\big( {\delta L(\gamma+\tau g,h)(u,v,w)-
        \delta L(\gamma,h)(u,v,w)} {\tau}\big) \,\dd w \, \dd v\, \dd u
    \notag\\
    =& 6 \textstyle\int_{\R/\Z} \iint_D \int_0^1 \delta^2 L(\gamma+s\tau g,h,g)(u,v,w) 
    \, \dd s \, \dd w \, \dd v \, \dd u,\label{eq:diff-quotient}
    \end{align}
   converging to $
        6\iiint_{\R/\Z\times D}\delta^2L(\g,h,g)(u,v,w) \, \dd w\, \dd v \, \dd u=\delta^2\Mpq[p,2]
    (\g)[h,g]
    $
    as $\tau\to 0$ according to Lebesgue's dominated convergence theorem,
    since the integrands $\int_0^1\delta^2 L(\g+s\tau g,h,g)\, \dd s$ in
    \eqref{eq:diff-quotient} converge pointwise to $\delta^2L(\g,h,g)$ as
    $\tau\to 0$ (see \eqref{eq:delta2integrand}), 
    and $G_{\tau_g}$ serves as an integrable majorant
    for these integrands for all $\tau\in (-\tau_g,\tau_g)$ by virtue of
    \eqref{eq:dominant} and \eqref{eq:G-int-bound}. So, we established the
    existence of the second variation $ \delta^2\Mpq[p,2](\g)[h,g]$,
    and the bound \eqref{theorem:secondVariationMp2:eq:estSecondVariation}
    follows from \eqref{eq:dominant} and
    \eqref{eq:G-int-bound} letting $\tau_g\to 0$, if we set
    $
    C_2:=C^* K (\halfnorm[\frac 32 p-3,2]{\g'}^4+
    \halfnorm[\frac 32 p-3,2]{\g'}^3+\halfnorm[\frac32 p-3,2]{\g'}^2)
    $ 
    and recalling that $K=K(n, p,
    \BiLip(\g),\halfnorm[\frac32 p-3,2]{\g'})$ depends non-increasingly
    on the bilipschitz constant $\BiLip(\g)$ and non-decreasingly
    on the seminorm $\halfnorm[\frac32 p-3,2]{\g'}$.
    
    The terms $\abs{\delta L(\gamma,h)(u,v,w)}=\abs{L(\gamma)(u,v,w) I(h)}$ constituting the integrand of the first
    variation (see \eqref{eq:delta1integrand}) are estimated in
    \eqref{eq:first-nuetzlich}, so that we can again use H\"older's inequality
    and the prototype estimate \eqref{eq:new-est}
    of Lemma \ref{lem:aux-est} to bound this integrand by
       \begin{align*}
         \textstyle  \frac{C_L}{\abs{v}^p \abs{w}^p \abs{v-w}^p}  \big( C_I 
    \halfnorm[\frac32 p -3,2]{h'}  
    {\abs{\hat \gamma}^2}  + 2  \big|\hatProduct{h,\gamma}\big| 
        \abs{\hat \gamma}  \big)
    {\leq}
    \tilde{K} \halfnorm[\frac32 p -3,2]{h'} \big(  
    \halfnorm[\frac32 p -3,2]{\gamma'}^4 +  
    \halfnorm[\frac32 p -3,2]{\gamma'}^3 \big),
    \end{align*}
    where we have subsumed all constants under the new constant 
    $\tilde{K}$ $=$
    $\tilde{K}(n, p,\BiLip(\g),\halfnorm[\frac32 p-3,2]{\g'})$ depending
    non-increasingly on $\BiLip(\g)$ and non-decreasingly on $\halfnorm[\frac32 p-3,2]{\g'}$. This
    establishes 
    \eqref{theorem:secondVariationMp2:eq:estFirstVariation} if we set
    $C_1:=\tilde{K}(  \halfnorm[\frac32 p -3,2]{\gamma'}^4 
    +  \halfnorm[\frac32 p -3,2]{\gamma'}^3 )$.
    {}
\end{proof}

\begin{lemma}\label{lem:aux-est}
Fix $p\in (\frac73,\frac83 ) $ and the measure $\dd \mu:=(|v||w||v-w|)^{-p} \, \dd w \, \dd v \, \dd u$
on $\R/\Z\times D$, where $D$ is defined in \eqref{eq:D-def}. Then there
is a constant $C=C(C_E,C_P,p)$ such that 
\begin{align}\label{eq:new-est}
\|\widehat{\left[f_1,f_2\right]}  \|_{L_\mu^2(\R/\Z\times D,\R^n)}^2 &\le 
C\halfnorm[\frac32 p-3,2]{f_1'}^2\cdot\halfnorm[\frac32 p-3,2]{f_2'}^2
\Foa f_1,f_2\in W^{\frac32 p-2,2}(\R/\Z,\R^n).
\end{align}
$C_E$ and $C_P$ are the constants of the Morrey embedding \eqref{eq:morrey-embedding} and Poincar\'e's inequality
    \eqref{eq:poincare}.  
\end{lemma}
\begin{proof}
    Using the fact that $\abs{\xi \wedge \zeta} \leq \abs{\xi} \abs{\zeta}$ for 
    all $\xi, \zeta \in \R^n$, see 
    \cite[Chapter~I, Section~12, Equation~(13)]{Whitney:1957:GeometricIntegrationTheory}, 
    we can bound the numerator of the integrand as
    \begin{align}
   \big|\hatProduct{f_1,f_2}\big|^2 
    & = \textstyle |w|^2|v|^2\big| \int_0^1 f_1'(u+yw) \dd y \! \wedge \! \int_0^1 f_2'(u+yv) \dd y 
            + \int_0^1 f_2'(u+yw) \dd y \! \wedge \! \int_0^1 f_1'(u+yv) \dd y \big|^2\notag
    \\
    & \leq  \textstyle 2|w|^2|v|^2\, \big\{ \,\, 
    {\big|{ \int_0^1 f_1'(u+yw) \dd y}\big|^2 \big|{\int_0^1 
            f_2'(u+yv)-f_2'(u+yw) \dd y}\big|^2}\notag \\
    & \qquad\qquad\qquad + \textstyle  {\big|{\int_0^1 f_2'(u+yw) \dd y}\big|^2 
        \big|{\int_0^1 f_1'(u+yv)-f_1'(u+yw) \dd y }\big|^2} \,\,\big\} \notag\\
    &
    {\leq}\textstyle
    2  C_E^2C_P^2|w|^2|v|^2 \,\big\{\,\,
    \halfnorm[\frac32 p-3,2]{f_1'}^2  
    {\big|{\int_0^1 f_2'(u+yv)-f_2'(u+yw) \dd y}\big|^2}\notag \\
    & \qquad\qquad\qquad + \textstyle \halfnorm[\frac32 p-3,2]{f_2'}^2 
    {\big|{\int_0^1 f_1'(u+yv)-f_1'(u+yw) \dd y }\big|^2}\,\,\big\},
    \label{eq:part-integrand}
    \end{align}
    where we used Morrey's embedding \eqref{eq:morrey-embedding} and
    the Poincar\'e inequality \eqref{eq:poincare}.
    It therefore suffices  by Jensen's inequality to give an 
    upper bound for the triple
    integral 
    $$
 \textstyle   \iiint_{\R/\Z \times D}  {\int_0^1\abs{ f_i'(u+yv)-f_i'(u+yw)}^2 \dd y } \,|w|^2|v|^2
    \, \dd \mu\quad\Fo i=1,2,
    $$
    which can be rewritten by means of Tonnelli's variant of
    Fubini's theorem,
    the substitution $z:=u+yw$, periodicity of the resulting
    integrand in the 
    $z$-variable, and a second application of 
    Fubini's theorem as
    $
   \textstyle \int_0^1\iiint_{\R/\Z \times D} \frac {\abs{ f_i'(z+y(v-w))-f_i'(z)}^2} {\abs{v}^{p-2} \abs{w}^{p-2} \abs{v-w}^p} \, \dd w\, \dd v\, \dd z\, \dd y\quad\Fo i=1,2.
    $
    Following an idea in the proof of 
    \cite[Theorem 1, p. 7]{BlattReiter:2015:Menger} we
    use the transformation $\Phi:D\to\Phi(D)$ mapping a pair $(v,w)\in D$ to
    $(t,\theta):=(v/(v-w),y(v-w))\in \Phi(D)\subset (0,1)\times (-1,0)$ 
    by definition of
    the parameter range $D$. Observing that $|\det D\Phi(v,w)|=y/|v-w|$ for all
    $(v,w)\in D$ we arrive at the transformed integral
    $
  \textstyle  \int_0^1\int_{\R/\Z}\int_0^1\int_{-1}^0
    \frac{ \abs{ f_i'(z+\theta)-f_i'(z)}^2}{
        |t\theta/y|^{p-2}|(t-1)\theta/y|^{p-2} |\theta/y|^{p-1}y} \, \dd\theta\, \dd t\, \dd z\, \dd y \quad\Fo i=1,2.
    $
    Regrouping one obtains finite integrals over the $y$-variable and over the
    $t$-variable due to the parameter range $p\in (\frac73 , \frac83 )$, 
    multiplied
    by the double integral
    $$
  \textstyle  \int_{\R/\Z}\int_{-1}^0
    \frac{ \abs{ f_i'(z+\theta)-f_i'(z)}^2}{
        \abs{ \theta}^{3p-5}} \,\dd\theta \, \dd z\le
    \halfnorm[\frac32 p-3,2]{f_i'}^2
    +
    \int_{\R/\Z}\int_{-1}^{-\frac12}
    \frac{ \abs{ f_i'(z+\theta)-f_i'(z)}^2}{
        \abs{ \theta}^{3p-5}} \,\dd\theta \, \dd z,
    $$
    which is bounded by $\halfnorm[\frac32 p-3,2]{f_i'}^2+12
    \|f_i'\|^2_{C^0(\R/\Z,\R^n)}
    \le (1+12 C_E^2C_P^2)\halfnorm[\frac32 p-3,2]{f_i'}^2$ for $i=1,2$
    by virtue of
    the Morrey embedding \eqref{eq:morrey-embedding} and Poincar\'e's inequality
    \eqref{eq:poincare}.  
    In summary, we obtain
    \[
    \label{theorem:secondVariationMp2:eq:estHatProduct}
   \|\widehat{\left[f_1,f_2\right]}  \|_{L_\mu^2(\R/\Z\times D,\R^n)}^2 
    \leq 4 C_E^2C_P^2(1+12 C_E^2C_P^2)C^2(p)\halfnorm[\frac32 p-3,2]{f_1'}^2 
    \halfnorm[\frac32 p-3,2]{f_2'}^2.
    \]
\end{proof}

As mentioned in the introduction it was already shown in
\cite[Theorem 3]{BlattReiter:2015:Menger} that the generalized integral
Menger curvature is
continuously differentiable
on regular embeddings of class
$W^{\frac32 p-2,2}$. Now,
\thref{theorem:secondVariationMp2} yields local Lipschitz continuity of 
the differential $D\Mpq[p,2]$.
To quantify this we define  the radius of 
Lipschitz continuity $R_{M}(\g)$ for a regular embedded
curve $\g\in W^{\frac32 p-2,2}(\R/\Z,\R^n)$ 
as
\begin{equation}\label{eq:lip-radius}
R_{M}\equiv R_{M}(\g):=
\min\left\{({2C_EC_P)^{-1}}\BiLip(\g),
1\right\}
\le 1,
\end{equation}
where $C_E$ denotes the constant 
of Morrey's embedding 
\thref{thm:morrey}, and $C_P$ is the constant in Poincar\'e's
inequality \eqref{eq:poincare}.
Exactly as in the beginning of the previous proof we notice that $R_{M} >0$
since $\g$ is a regular embedded fractional Sobolev curve.
\thref{cor:bilip2} in Appendix \ref{section:Arclength} implies that
$\BiLip(\eta)\ge \BiLip(\g)/2$ and $\halfnorm[\frac32 -3,2]{\eta'} <
\halfnorm[\frac32 -3,2]{\g'}+1$ for every $\eta\in B_{R_M}(\g)$. Therefore,
combining \eqref{theorem:secondVariationMp2:eq:estSecondVariation} with
    a suitable mean value inequality 
    (see, e.g., \cite[Theorem 3.2.7]{drabek-milota_2013}) we 
    immediately deduce how  the
    Lipschitz constant of the differential $D\Mpq[p,2]$ depends
    on the bilipschitz constant and the fractional 
    Sobolev norm of the embedded curve $\gamma$
    within the ball $B_{R_M}(\gamma)$.
\begin{corollary}
    \label{corollary:DintMp2locallyLipschitz} 
    For $p\in(\frac 7 3,\frac 8 3)$ and  $\gamma \in 
    W^{\frac 3 2 p -2 ,2}_{\ir}(\R/\Z,\R^n)$
    the differential $D\Mpq[p,2](\cdot)$ and hence also the gradient
    $\nabla\Mpq[p,2](\cdot)$ are Lipschitz
    continuous on the ball $B_{R_{M}}(\g)\subset W^{\frac32 p-2,2}(\R/\Z,\R^n)$
    with Lipschitz constant 
    \begin{equation}\label{eq:lip-DintM}
    \LIP_{D\Mpq[p,2]}\equiv\LIP_{\nabla\Mpq[p,2]}=\LIP_{\nabla\Mpq[p,2]} 
    (n, p,\BiLip(\gamma),\halfnorm[\frac 3 2 p - 3,2]{\g'}),
    \end{equation}
    depending non-increasingly on
    $\BiLip(\gamma)$ and non-decreasingly 
    on $\halfnorm[\frac32 p-3,2]{\gamma'}$.
\end{corollary}

\section{Short Time Existence}
\label{section:ShortTimeExistence}
We first state 
the well-known Picard-Lindel\"of Theorem for ordinary differential equations
in Banach spaces; see e.g. 
\cite[Part~II, Corollary~1.7.2]{Cartan:1971:Differentialcalculus}.
\begin{theorem}[Picard-Lindel\"of]
    \label{theorem:PicardLindelöfSimplified}
    Let $\X$ be a real Banach space with norm ${\norm{\cdot}}_{\X}$, $V$ a neighbourhood of $(t_0,x_0) \in \R \times \X$ and $f\colon V \to \X$ be a continuous function with
$    
    {\norm{f(t,y)-f(t,z)}}_{\X} \leq \LIP_f {\norm{y-z}}_{\X}$ for all $
(t,y),(t,z)\in V.
    $
    Then, there is $\alpha>0$ such that the differential equation
$    \frac d {dt} x(t) = f(t,x(t))$
    has a unique solution $\varphi\in C^1(I,\X)$ for $I := [t_0 - \alpha, t_0 + \alpha]$ with $\varphi(t_0)=x_0$.
    More precisely, if $\tau>0$ and $r>0$ are chosen to be small enough such that $[t_0-\tau, t_0+\tau] \times \ball[x_0]{r}$ is contained in $V$ and 
    $
    {\norm{f(t,x)}}_{\X} \leq F$ for all $|{t-t_0}| \leq \tau $ and $
    \|{x-x_0}\|_{\X} \leq r,
    $
    then $\alpha$
    may be chosen as
    $
    \alpha := \min\menge{\tau, r/F}
    $
    and  the image $\varphi(I)$ is contained in the ball $\ball[x_0]{r}\subset\X$.
\end{theorem}
In the last section, we saw that the gradient of $\Mpq[p,2]$ is locally Lipschitz continuous, so we have short time existence of a solution of
\eqref{eq:gradflow}.
\begin{theorem}[Short time existence]
    \label{theorem:shortTimeExistence}
    Let $p \in \bigl(\frac 7 3, \frac 8 3\bigr)$ and $\gamma_0 \in W^{\frac 3 2 p - 2,2}_{\ir}(\R/\Z,\R^n)$.
    Then, there exists a time  $T>0$ and a  unique\footnote{$\g$ is unique in the sense that any other $C^1$-mapping $\eta$ solving \eqref{eq:gradflow}
        on a time interval $[0,\hat{T}]$ coincides with $\g$ on
        $[0,\min\{T,\hat{T}\}\,]$.} mapping
 $
    t\mapsto  \gamma(\cdot,t) \in C^{1,1}\big(
    [0,T],W^{\frac 3 2 p - 2,2}(\R/\Z,\R^n)\big)
    $
    such that $\gamma$ solves \eqref{eq:gradflow}, with strict
    energy decay in time, i.e., such that the mapping 
    $t\mapsto\Mpq[p,2](\g(\cdot,t))$ is strictly decreasing $[0,T]$.
    One can estimate the minimal time $T$
    of existence from below as
    \begin{equation}\label{eq:existence-time-est}
    T\ge  T_\textnormal{min}:=\frac{R_{M}(\g_0)}{C_1(n, p,\BiLip(\g_0)/2,\halfnorm[\frac32 p-3,2]{\g_0'}+1)},
    \end{equation}
    where $R_{M}$ is the Lipschitz radius of the initial curve
    $\g_0$ defined in
    \eqref{eq:lip-radius} and $C_1$ is the constant bounding the first
    variation $\delta\Mpq[p,2]{(\g_0,\cdot)}$  
    in \eqref{theorem:secondVariationMp2:eq:estFirstVariation} of \thref{theorem:secondVariationMp2}. Moreover, the image
    $\g(\cdot,[0,T_\textnormal{min}])$ is contained in
    the ball $B_{R_M}(\g_0)\subset W^{\frac32 p-2,2}(\R/\Z,\R^n)$, 
    and the
    barycenter is conserved along the flow, i.e.,
    \begin{equation}\label{eq:barycenter-conserved}
    \overline{\g(\cdot,t)}:=\textstyle\int_{\R/\Z}\g(u,t)\,\dd u=\int_{\R/\Z}\g_0(u)\, \dd u=
    \overline{\g_0}\quad\Foa t\in [0,T].
    \end{equation}
\end{theorem}
\begin{proof} 
Set $\HL:=W^{\frac32 p-2,2}(\R/\Z,\R^n)$.
    According to \thref{corollary:DintMp2locallyLipschitz} the right-hand 
    side of \eqref{eq:gradflow} is
    Lipschitz continuous on the open neighbourhood 
    $V:=\R\times B_{R_{M}}(\g_0)$
    of $(0,\g_0)\in  \R\times \HL$, so that
    we can apply \thref{theorem:PicardLindelöfSimplified} to conclude 
    short time existence of a solution 
    $\g\in C^1([0,T],
    \HL)$. 
        Differentiating the energy $\mathcal{E}:=\Mpq[p,2]$
    of $\g(\cdot,t)$ with respect to time we obtain
    \begin{align}\label{eq:energy-decay1}
  \textstyle  \frac{d}{dt} \mathcal{E}(\g(\cdot,t)) =&\langle \nabla\mathcal{E}
    (\g(\cdot,t)),
    \tfrac{\partial }{\partial t}\g(\cdot,t)\rangle_{\HL}
   {=}
    -     
    \|\nabla\mathcal{E}(\g(\cdot,t))\|_{\HL}^2
    < 0
    \end{align}
    for all $ t
        \in (0,T)$,
    which implies the strict
    energy decay in time.   Notice for the strict inequality in
    \eqref{eq:energy-decay1} that $\mathcal{E}=\Mpq[p,2]$ does not possess any
    critical point since it is not scale-invariant. Indeed, 
    every strictly positive and Fr\'echet-differentiable
    energy $\mathcal{F}$ on an open subset $\mathcal{O}$ of
    a Hilbert space  with the homogeneity
    condition $\mathcal{F}(\mu x)=\mu^\beta \mathcal{F}(x)$ for all 
    $x\in\mathcal{O}$ and
    $|\mu-1|\ll 1$, 
    for some fixed 
    $\beta\not=0$,
    satisfies $D\mathcal{F}(x)x=\beta 
    \mathcal{F}(x)\not=0$ for all $x\in\mathcal{O}\setminus\{0\}$.

    The minimal existence time  $T$
    can be bounded from below according to the explicit choice of $\alpha$ in
    \thref{theorem:PicardLindelöfSimplified}, which leads
    to  \eqref{eq:existence-time-est} , since  the upper bound \eqref{theorem:secondVariationMp2:eq:estFirstVariation} in 
    \thref{theorem:secondVariationMp2} yields  
    \begin{align}
        \|D\mathcal{E}(\eta)\|_{\HL^*}
	&=\norm[\HL^*]{
        \delta\mathcal{E}(\eta,\cdot)}
   {\le} 
    C_1(n, p,\BiLip(\eta),\halfnorm[\frac32 p-3,2]{\eta'})\notag\\
   {\le} 
    C_1(n, p,&\BiLip(\g_0)/2,  \halfnorm[\frac32 p-3,2]{\g_0'}+1)=:F\quad\Foa 
    \eta\in B_{R_{M}}(\g_0),\label{eq:bound-gradient}
    \end{align}
    by choice of the radius $R_{M}$ in \eqref{eq:lip-radius} combined with
    \thref{cor:bilip2} in the appendix.
    
    The time derivative $\partial_t\g:=\partial\g/\partial t$
    is even
    Lipschitz  continuous with respect to time since the right-hand side
    $f(\eta):=\nabla\EL(\eta)$ of \eqref{eq:gradflow}, satisfying 
    $\norm[\HL]{f(\eta)}\le  F$ for all $\eta
    \in B_{R_{M}}(\g_0)$,
    does not depend explicitly on $t$. Therefore, 
    we  may estimate by means of
    \thref{corollary:DintMp2locallyLipschitz}
    for $0\le \sigma < s$
    \begin{align*}
  &  \norm[\HL]{\partial_t\g(s,\cdot)-\partial_t
        \g
        (\sigma,\cdot)} 
   {=}
    \norm[\HL]{f(\g(s,\cdot))-f
        (\g(\sigma,\cdot))}
 {\le}
    \LIP_f\|{\g(s,\cdot)-\g(\sigma,\cdot)}\|_{\HL}\\
    &
   {=}\textstyle
    \LIP_f
\big\|
\int_\sigma^s f(\g(\tau,\cdot))\,\dd\tau\big\|_{\HL} 
    \le \LIP_f\int_\sigma^s 
    \norm[\HL]{f(\g(\tau,\cdot))}\,\dd\tau 
      \,\,\,\, {\le} \,\,
    F\cdot  \LIP_f  |s-\sigma|.
    \end{align*}
    The conservation of the barycenter follows from the
    first identity in \eqref{eq:integral-mean}
    of  
    \thref{lem:gradients-mean}  below,
    applied to the open subset $\mathcal{O}:=W_\textnormal{ir}^{\frac32 p-2,2}(\R/\Z,\R^n)$ of the Hilbert space $\HL_1:=\HL$
    and the energy $\mathcal{E}=\Mpq[p,2]$ yielding
    \begin{equation}\label{eq:integral-means-Menger}
    \overline{\nabla\mathcal{E}(\g(\cdot,t))}:=\textstyle\int_{\R/\Z}\nabla\mathcal{E}
    (\g(u,t))\,\dd u=0\quad\Foa t\in [0,T].
    \end{equation}
    By the uniform bound $F$ on the energy's gradient (and therefore on
    $\partial_t\g(\cdot,t)$ according to \eqref{eq:gradflow}) presented
    in \eqref{eq:bound-gradient} we may interchange differentiation
    and integration to obtain
   \[ 
 \textstyle   \frac d{dt}\int_{\R/\Z}\g(u,t)\,\dd u=\int_{\R/\Z}\partial_t\g(u,t)\,\dd u
   {=}-\overline{\nabla\mathcal{E}
    (\g(\cdot,t))}=0
    \Foa t\in [0,T].
    \]
\end{proof}

We finish this section with the above mentioned
technical lemma on vanishing integral means
of gradients in the following more abstract situation. Assume
that $\mu$ is a measure on a set $K$ with $\mu(K)<\infty$ so that
the constant vector-valued functions $\alpha:K\to\R^n$
are automatically contained in the space
$L^2_\mu(K,\R^n)$ of square integrable functions on $K$ with respect to
the measure $\mu$. We denote the integral mean of a function $f$
by $\overline{f}:=\mu(K)^{-1}\int_K f\,\dd\mu$.
\begin{lemma}[Gradients with zero integral mean]\label{lem:gradients-mean}
    Let $\HL_1$ be a real  Hilbert space containing the constant 
    functions $\alpha:K\to\R^n$, and let $\mathcal{O}\subset\HL_1$ be an open
    subset of $\HL_1$. Assume that
    $\mathcal{E}\colon \mathcal{O}\to\R$ is a Fr\'echet-differentiable,
    translationally invariant
    energy.
    Then
    \begin{equation}\label{eq:grad-skp}
    \langle \nabla\mathcal{E}(x),\alpha\rangle_{\HL_1}=0\quad\Foa
    x\in\mathcal{O},\alpha\in\R^n.
    \end{equation}
    If, moreover, for a second real Hilbert space $\HL_2$, 
    there is a Fr\'echet-differentiable mapping $S\colon \mathcal{O}\to\HL_2$   
    satisfying
    \begin{equation}\label{eq:DS-kernel}
    DS(x)\alpha=0\quad\Foa x\in\mathcal{O},\alpha\in\R^n,
    \end{equation}
    then one has
    \begin{equation}\label{eq:proj-grad-skp}
    \langle \nabla_S\mathcal{E}(x),\alpha\rangle_{\HL_1}=0\quad\Foa
    x\in\mathcal{O},\alpha\in\R^n,
    \end{equation}
    where $\nabla_S\mathcal{E}(y)=\Pi_{\NL(DS(y))}\nabla\mathcal{E}(y)$ for
    $y\in\mathcal{O}$
    denotes the projected gradient as defined in \eqref{eq:proj-abstract-gradient}
    in the introduction.
    
    If, finally, $\HL_1$ is contained in $L_\mu^2(K,\R^n)$ and if 
    there is a constant $C>0$ with
    \begin{equation}\label{eq:skp}
    \langle x,\alpha\rangle_{\HL_1}=C\langle x,\alpha\rangle_{L^2_\mu}\quad
    \Foa x\in\HL_1,\alpha\in\R^n,
    \end{equation}
    then \eqref{eq:grad-skp} and \eqref{eq:proj-grad-skp} imply, respectively, 
    \begin{equation}\label{eq:integral-mean}
    \overline{\nabla\mathcal{E}(x)}=0,\quad\AND\quad
    \overline{\nabla_S\mathcal{E}(x)}=0
    \Foa x\in \mathcal{O}.
    \end{equation}
\end{lemma}
\begin{proof}
    With $\mathcal{E}(x+\epsilon\alpha)=\mathcal{E}(x)$ for all $x\in\mathcal{O}$,
    $\alpha\in\R^n$, and $\epsilon\in\R$ with $|\epsilon|\ll 1$,
    we find
    $
    \langle\nabla\mathcal{E}(x),\alpha\rangle_{\HL_1}=D\mathcal{E}(x)\alpha
    =\delta\mathcal{E}(x,\alpha)=
    \lim_{\epsilon\to 0}\epsilon^{-1}\cdot (
   {\mathcal{E}(x+\epsilon\alpha)-\mathcal{E}(x)})=0,
    $
    which is \eqref{eq:grad-skp}.
    To prove
    \eqref{eq:proj-grad-skp} we use the assumption \eqref{eq:DS-kernel}
    as well as \eqref{eq:grad-skp} to infer
    $$
    \langle\nabla_S\mathcal{E}(x),\alpha\rangle_{\HL_1}  =
    \langle \Pi_{\NL(DS(x))}\nabla\mathcal{E}(x),\alpha\rangle_{\HL_1} 
    {=}\langle\nabla\mathcal{E}(x),
    \Pi_{\NL(DS(x))}\alpha\rangle_{\HL_1}
    {=}
    \langle\nabla\mathcal{E}(x),\alpha\rangle_{\HL_1}
    {=}0.
    $$
    Choosing $\alpha:=\mu(K)^{-1}\overline{\nabla\mathcal{E}(x)}\in\R^n$ 
    in \eqref{eq:grad-skp} one
    obtains from assumption \eqref{eq:skp}
    \begin{align*}
    0 &
    {=}
    \textstyle C^{-1}\langle\nabla\mathcal{E}(x),\frac1{\mu(K)}
    \overline{\nabla\mathcal{E}(x)}\rangle_{\HL_1}
    {=}
    \langle\nabla\mathcal{E}(x),\frac1{\mu(K)}
    \overline{\nabla\mathcal{E}(x)}\rangle_{L_\mu^2}\\
    & =\textstyle
    \frac1{\mu(K)}\int_K\langle\nabla\mathcal{E}(x)(w),
    \overline{\nabla\mathcal{E}(x)}\rangle_{\R^n}\,\dd\mu(w)
    =\textstyle\langle\frac1{\mu(K)}\int_K\nabla\mathcal{E}(x)(w)\,\dd\mu(w),
    \overline{\nabla\mathcal{E}(x)}\rangle_{\R^n}=|\overline{\nabla\mathcal{E}(x)}|^2,
    \end{align*}
    so that the first identity in \eqref{eq:integral-mean} is established. 
    Testing the identity 
    $
    \langle \nabla_S\mathcal{E}(x),\alpha\rangle_{L_\mu^2}
    \overset{\eqref{eq:skp}}{=}C^{-1}
    \langle \nabla_S\mathcal{E}(x),\alpha\rangle_{\HL_1}
    \overset{\eqref{eq:proj-grad-skp}}{=}0
    $
    with $\alpha:=\mu(K)^{-1}\overline{\nabla_S\mathcal{E}(x)}
    $ yields the second identity in \eqref{eq:integral-mean}.
    {}
\end{proof}

For a proof of long time existence for \eqref{eq:gradflow}
via a continuation argument replacing the
initial data $\g_0$ by the configuration $\g(\cdot,T)$ as a new starting
point, it would be helpful to control
the essential quantities in \eqref{eq:existence-time-est} that bound $T$ from below,
namely the bilipschitz constants of $\g(\cdot,t)$ and the seminorms $\halfnorm[\frac32 p-3,2]{
    \g'(\cdot,t)}$. If we knew that the curves $\g(\cdot,t)$ had unit speed for all $t\in
[0,T]$ we could use \cite[Proposition 2.1 \& Theorem 1]{BlattReiter:2015:Menger}
to bound $\BiLip(\g(\cdot,t))$  from below and $\halfnorm[\frac32 p-3,2]{
    \g'(\cdot,t)}$  from above in terms of the energies
$\Mpq[p,2](\g(\cdot,t))$ that do not increase in time as $t\to T$ along
the gradient flow \eqref{eq:gradflow}.
However, even if the initial curve $\g_0$ is assumed to be arc length parametrized,
we do not know in general how the parametrizations may change along the gradient
flow \eqref{eq:gradflow}, which is why we chose the projected gradient flow
\eqref{introtheorem:longTimeExistenceViaProjections:eq:ODE} 
to preserve the initial velocity $\abs{\g_0'}=\abs{\g'(\cdot,t)}$
for all $t$. In particular, starting with an arc\-length parametrized initial curve 
$\g_0$ we will obtain an evolution of unit speed curves.

Of course, the right-hand side of the projected flow equation \eqref{introtheorem:longTimeExistenceViaProjections:eq:ODE} 
differs from that of the usual gradient flow \eqref{eq:gradflow}, and we need
to ensure local Lipschitz continuity again to obtain short time existence first, and then lower bounds for the Lipschitz radii are required to extend the solutions to
establish long time existence for \eqref{introtheorem:longTimeExistenceViaProjections:eq:ODE}. 
What we need \emph{not} worry about is the upper bound on the right-hand side
of \eqref{introtheorem:longTimeExistenceViaProjections:eq:ODE}, because
abbreviating $\mathcal{E}:=\Mpq[p,2]$ and $\mathcal{H}:=
W^{\frac32 p-2,2}(\R/\Z,\R^n)$ we simply estimate by means of
\eqref{theorem:secondVariationMp2:eq:estFirstVariation}
$
\|\nabla_{\SIGMA}\mathcal{E}(\g)\|_\mathcal{H}=
\|\Pi_{\NL(D\SIGMA (\gamma))}\nabla\mathcal{E}(\g)\|_\mathcal{H}\le
\|\nabla\mathcal{E}(\g)\|_\mathcal{H}
{\le} C_1,
$
as $\Pi_{\NL(D\SIGMA(\gamma))}$ is an orthogonal projection.

The right hand side of \eqref{introtheorem:longTimeExistenceViaProjections:eq:ODE}  consists of the three mappings, $A \mapsto \Pi_{\NL(A)}$, $v \mapsto D\SIGMA (v)$, and $\nabla\Mpq[p,2](\cdot)$, and
we already took care of the third mapping in Section  \ref{sec:2}.
In the following Sections \ref{sec:4} and \ref{subsection:explicitBoundaryCondition}
we establish   Lipschitz properties of the first two.

\section{Projection onto the null space of a linear operator}
\label{sec:4}
Let $\HL_1$ and $\HL_2$ be two  real Hilbert spaces with inner products
$\langle \cdot, \cdot\rangle_{\HL_i}\colon \HL_i\times \HL_i\to\R$ for $i=1,2.$ 
Throughout this section
$A \colon \HL_1 \to \HL_2$ is a bounded linear operator, the set of all such 
operators is denoted by $\LL(\HL_1,\HL_2)$. The  adjoint
$A^*\colon \HL_2\to\HL_1$ of $A$ satisfies by definition
\begin{equation}\label{eq:adjoint}
\langle Ax, y\rangle_{\HL_2}=\langle x,A^*y\rangle_{\HL_1}\quad\Foa x\in\HL_1, \, y
\in\HL_2.
\end{equation}
We start with three simple observations regarding the composition 
$AA^*\colon \HL_2\to\HL_2$.
\begin{lemma}
    \label{lemma:NullspaceAA^*}
    The null spaces of $A^*$ and $AA^*$ are identical, that is,
    $\NL(A^*)=\NL(AA^*)\subset\HL_2.$
\end{lemma}
\begin{proof}
    It suffices to prove that $\NL(AA^*)\subset\NL(A^*)$, the other inclusion is
    immediate. For $y\in\NL(AA^*)$ one has  by means of \eqref{eq:adjoint}
    $
    0= \langle x,AA^* y\rangle_{\HL_2}{=}
    \langle A^*x,A^*y \rangle_{\HL_1}$ for all $ x\in\HL_2,
    $
    which can be used for 
    $x:=y$ to obtain $\norm[\HL_1]{A^*y}=0$ so that $y\in\NL(A^*)$.
    {}
\end{proof}

\begin{lemma}
    \label{lemma:invertibilityAA*}
    Let $A \in \LL(\HL_1,\HL_2)$ be surjective.
    Then $AA^*$ is invertible and the inverse $(AA^*)^{-1}$ is bounded.
\end{lemma}
\begin{proof}
    By \cite[Theorem~4.12]{Rudin:1973:FunctionalAnalysis}, we have for any bounded linear operator $T$ between two Hilbert spaces the identity
$    \NL(T^*) = \RL(T)^\perp.$
    In particular, $\NL(A^*) = \menge{0}$ because $\RL(A)=\HL_2$ by assumption. 
    \thref{lemma:NullspaceAA^*} yields
    $\NL(AA^*)=\{0\}$, i.e.,  $AA^*$ is  injective.
    To show that $AA^*$ is also surjective, and therefore invertible,
    we use the fact that $\RL(A)=\HL_2$ by assumption
    and apply \cite[Theorem~4.15]{Rudin:1973:FunctionalAnalysis}
    to the linear operator $T:=A$
    to find a constant $c>0$ such that 
$    \|A^*y\|_{\HL_1}\ge c\|y\|_{\HL_2}$ for all $y\in \HL_2.$
    This implies by means of \eqref{eq:adjoint} and the Cauchy-Schwarz inequality
    $
    c^2\|y\|_{\HL_2}^2
    {\le}
    \langle A^*y,A^*y\rangle_{\HL_1}
    \overset{\eqref{eq:adjoint}}{=}\langle AA^* y,y\rangle_{\HL_2}=
    \langle (AA^*)^* y,y\rangle_{\HL_2}\le\|(AA^*)^* y\|_{\HL_2}\|y\|_{\HL_2},
    $
    which simplifies to 
    \begin{equation}\label{eq:final-astern-lower-bound}
    \|(AA^*)^* y\|_{\HL_2} \ge c^2\|y\|_{\HL_2}\quad\Foa y\in \HL_2,
    \end{equation}
    so that again by 
    \cite[Theorem~4.15]{Rudin:1973:FunctionalAnalysis} now applied to the linear
    operator $T:=AA^*$ we obtain $\RL(AA^*)=\HL_2$. Now that we know that
    $(AA^*)^{-1}$ exists we can evaluate
    \eqref{eq:final-astern-lower-bound} at $y:=(AA^*)^{-1}x$ for an arbitrary
    $x\in\HL_2$, and use again
    that $AA^*=(AA^*)^*$ to arrive at the estimate
    $
    \|x\|_{\HL_2}\ge c^2\|(AA^*)^{-1}x\|_{\HL_2}$ for all $ x\in\HL_2,
    $
    which implies $\|(AA^*)^{-1}\|_{\LL(\HL_2)}\le 1/c^2.$
    {}
\end{proof}
The bound on the norm of the operator $(AA^*)^{-1}$ is rather implicit, since
its origin in the proof of \cite[Theorem~4.15]{Rudin:1973:FunctionalAnalysis} 
lies in the open mapping theorem. However,  we do need more explicit bounds in order
to verify Lipschitz continuity of the projection involving
the logarithmic strain constraint. Therefore,  
we add an extra assumption in the present
abstract context, namely, 
that the operator $A$ possesses a bounded right inverse $Y$.
\begin{lemma}
    \label{lemma:boundedInverseAA*}
    Suppose that $A \in \LL(\HL_1,\HL_2)$ possesses a right inverse, that is, 
    there is a linear operator
    $Y \in \LL(\HL_2,\HL_1)$ such that $AY =\Id_{\HL_2}$.
    Then,
    \begin{equation}\label{eq:R-bound} 
    \norm[\HL_2]{x} \leq \norm[\LL(\HL_2,\HL_1)]{Y}\norm[\HL_1]{A^*x}
    \quad\Foa x\in \HL_2,
    \end{equation} 
    the composition
    $AA^*$ is invertible, and  its inverse $(AA^*)^{-1}$ satisfies
    the estimate
    \begin{equation}\label{eq:R-AA-bound} 
    \norm[\LL(\HL_2,\HL_2)]{(AA^*)^{-1}} \leq \norm[\LL(\HL_2,\HL_1)]{Y}^2.
    \end{equation} 
\end{lemma}
\begin{proof}
    Because $A$ has a right-inverse, it is surjective, so that
    according to \thref{lemma:invertibilityAA*}  $(AA^*)^{-1}$
    exists and is bounded.
    For all $x \in \HL_2 $, we infer by  
    \eqref{eq:adjoint} and 
    the Cauchy-Schwarz inequality 
    that 
    $
    \norm[\HL_2]{x}^2=\langle AY 
    x,x\rangle_{\HL_2}
    {=}
    \langle Y x,A^* x\rangle_{\HL_1}\le\norm[\HL_1]{Y x}\norm[\HL_1]{A^*x}\le
    \norm[\LL(\HL_2,\HL_1)]{Y}\norm[\HL_1]{A^*x}\norm[\HL_2]{x},
    $
    which implies \eqref{eq:R-bound}.
    Furthermore, we have by the Cauchy-Schwarz inequality, 
    \eqref{eq:adjoint}, and  \eqref{eq:R-bound}
    \begin{alignat*}{3}
    \|{(AA^*)^{-1}}\|_{\LL(\HL_2,\HL_2)}
    &=
    \sup_{0\neq y \in \HL_2} 
    \frac {\|{(AA^*)^{-1}y}\|_{\HL_2}} {\|{y}\|_{\HL_2}}
    &&=
    \sup_{0\neq x \in \HL_2} 
    \frac {\|{(AA^*)^{-1}AA^*x}\|_{\HL_2}} {\|{AA^*x}\|_{\HL_2}}\\
    \leq 
    \sup_{0\neq x \in \HL_2} 
    \frac {\norm[\HL_2]{x}} {\left\langle AA^*x, 
         {x}/ {\norm[\HL_2]{x}}\right\rangle_{\HL_2}\hspace{-0.7em}}
    &
    {=} 
    \sup_{0\neq x \in \HL_2} \frac {\norm[\HL_2]{x}} 
    {\left\langle A^*x,  {A^*x}/ {\norm[\HL_2]{x}}\right\rangle_{\HL_1}\hspace{-1em}}
    &&= \sup_{0\neq x \in \HL_2}  \frac{\norm[\HL_2]{x}^2}{\norm[\HL_1]{A^*x}^2}
    {\leq} 
    \norm[\LL(\HL_2,\HL_1)]{Y }^2.
    \end{alignat*}
    {}
\end{proof}

Our interest in the invertibility of the mapping $AA^*\in\LL(\HL_2,\HL_2)$
is motivated by the following characterization of orthogonal projections
onto the null space of $A$, which can be found in
\cite[Lemma 6.2]{Neuberger:1997:SobolevGradients}. For the readers' convenience
we provide an alternative  proof here.
\begin{lemma}[Orthogonal projection onto \protecting{$\NL(A)$}]
    \label{lemma:orthogonalProjektion}
    Let $\HL_1,\HL_2$ be Hilbert 
    spaces and  $A : \HL_1 \to \HL_2$ linear and bounded.
    Suppose $(AA^*)^{-1}$ exists and is linear and bounded.
    Then, the orthogonal projection $\Pi_{\NL(A)}$ of $\HL_1$ onto $\NL(A)$ is given by
    \begin{equation}\label{eq:explicit-projection}
    \Pi_{\NL(A)} = \Id_{\HL_1} - A^*(AA^*)^{-1}A.
    \end{equation}
\end{lemma}
The operator $A^*(AA^*)^{-1}$ is one form of the \emph{Moore-Penrose pseudoinverse} of $A$.
\begin{proof}
    Denoting the right-hand side of \eqref{eq:explicit-projection} as $Z$
    we observe that  $Z\colon \HL_1\to\HL_1$ is linear and bounded, $A
    Z(x) =0$ for all $x\in\HL_1$,  and $Z$
    restricted to $\NL(A) $ coincides with the identity $\Id_{\HL_1}$
    and therefore $\RL(Z)=\NL(A)$. 
    Finally, $x-Z(x)\perp\NL(A)$
    for all $x\in\HL_1$, since by property \eqref{eq:adjoint} characterizing
    the adjoint
    $\langle x-Z(x),y\rangle_{\HL_1} =
    \langle A^*(AA^*)^{-1}Ax,y\rangle_{\HL_1}
{=}
    \langle (AA^*)^{-1}Ax,Ay\rangle_{\HL_1}=0$ for all $y \in\NL(A).$
    This last property characterizes orthogonal projections onto
    closed subspaces in Hilbert spaces
    according to \cite[4.4.2]{Alt:2016:LinearFunctionalAnalysis}, so that
    $Z=\Pi_{\NL(A)}$.
    {}
\end{proof}
We are going to prove Lipschitz continuity of the mapping
$A\mapsto \Pi_{\NL(A)}$ by bounding its differential, similarly
as in the proof of \thref{corollary:DintMp2locallyLipschitz}. In view of the
explicit formula \eqref{eq:explicit-projection} we therefore need to estimate
the differential of inverses of operators, for which we provide the preparation
in the two following elementary lemmas.

For Banach spaces $\B_1$, $\B_2$, recall that $\LL(\B_1,\B_2)$ is the space of linear and bounded operators from $\B_1$ to $\B_2$, $\epi(\B_1,\B_2)$ denotes the subset of surjective linear bounded operators, whereas $\GLL(\B_1,\B_2)$ is the subset of such operators that are bijective.

The following lemma is a quantified version of \cite[Lemma 2.5.4]{AbrahamMarsdenRatiu:1988:ManifoldsTensoAnalysis}.
\begin{lemma}
    \label{lemma:GLOffenInLLMitAbschaetzung}
    Let $\B$ be a Banach space with operator norm $\norm{\cdot}:=\norm[\LL(\B,\B)]{\cdot}.$
    Suppose that the two operators
    $f \in \GLL(\B,\B)$ and $g \in \LL(\B,\B)$ satisfy $\norm{f-g}\leq 
    \frac 1 2 \|{f^{-1}}\|^{-1}$.
    Then, $g \in \GLL(\B,\B)$ and $\|{g^{-1}}\| \leq 2 \|{f^{-1}}\|$.
\end{lemma}
\begin{proof}
    The operator $g=f\circ (\Id_\B-f^{-1}\circ (f-g))$ is invertible since $f$ is, and
    because
    $
    \|{f^{-1}\circ (f-g)}\|\le \|{f^{-1}}\|\cdot\|{f-g}\|\le 1/2,
    $
    and therefore the Neumann series $\sum_{k=0}^\infty (f^{-1}\circ (f-g))^k$
    converges and equals $(\Id_\B-f^{-1}\circ (f-g))^{-1}$. Thus, $g^{-1}$ exists
    and equals $(\Id_\B-f^{-1}\circ (f-g))^{-1}\circ f^{-1}$ satisfying the estimate
    $
    \norm{g^{-1}}\le\norm{(\Id_\B-f^{-1}\circ (f-g))^{-1}}\cdot\norm{f^{-1}}
    \le\sum_{k=0}^\infty\norm{f^{-1}\circ (f-g)}^k\norm{f^{-1}}\le 2\norm{f^{-1}}.
    $
    {}
\end{proof}
\begin{lemma}[\protecting{\cite[Lemma 2.5.5]{AbrahamMarsdenRatiu:1988:ManifoldsTensoAnalysis}}]
    \label{lemma:differentialInversion}
    Let $\B_1,\B_2$ be Banach spaces and consider the mapping 
    $
    \Jj \colon \GLL(\B_1,\B_2) \to \GLL(\B_2,\B_1)
    $ defined as $\Jj(f):=f^{-1}$ for $f\in\GLL(\B_1,\B_2)$.
    Then, $\Jj$ is of class $C^\infty$ and 
    \begin{equation}\label{eq:differentialInversion}
    D\Jj (f)g = -f^{-1} \circ g \circ f^{-1}\quad\Foa g\in T_f\GLL(\B_1,\B_2)\simeq
    \LL(\B_1,\B_2).
    \end{equation}
\end{lemma}
With this, we may differentiate $A\mapsto\Pi_{\NL(A)}$  
 if $A$
is surjective, since then we may use  the explicit representation 
\eqref{eq:explicit-projection} keeping in mind that $(AA^*)^{-1}$ exists and
is a linear and bounded operator by virtue of \thref{lemma:invertibilityAA*}.
\begin{lemma}
    \label{lemma:estimateFirstVariationP}
    Let $\HL_1,\HL_2$ be Hilbert spaces.
    Then $\Pi_{\NL(\cdot)}\colon \epi(\HL_1,\HL_2)\to\LL(\HL_1,\HL_1)$ 
    is Fréchet differentiable, and the differential $D(\Pi_{\NL(A)})$ 
    at $A\in\epi(\HL_1,\HL_2)$ may be estimated as 
    \begin{align}
    &\|{ D(\Pi_{\NL(A)})B}\|_{\LL(\HL_1,\HL_1)}
    \leq 2 \|{A}\|_{\LL(\HL_1,\HL_2)} \|{B}\|_{\LL(\HL_1,\HL_2)} 
    \|{(AA^*)^{-1}}\|_{\LL(\HL_2,\HL_2)}\,\textnormal{for $B\in\LL(\HL_1,\HL_2).$}
 \label{eq:DPi-bound}
    \end{align}
\end{lemma}
\begin{proof}
    Differentiating 
    $
    \Pi_{\NL(A+tB)}\overset{\eqref{eq:explicit-projection}}{=}\Id_{\HL_1}-
    (A+tB)^*((A+tB)(A+tB)^*)^{-1}(A+tB)
    $
    with respect to $t$ and evaluating at $t=0$ we obtain by virtue of 
    \thref{lemma:differentialInversion}
    \begin{align*}
    D(\Pi_{\NL(A)})B 
    & = -B^*(AA^*)^{-1}A - A^*(AA^*)^{-1}B + A^*(AA^*)^{-1}
    \left(AB^* + BA^*\right) (AA^*)^{-1}A\\
    & = -\Pi_{\NL(A)}B^*(AA^*)^{-1}A-A^*(AA^*)^{-1}B\Pi_{\NL(A)}.
    \end{align*}
    Now \eqref{eq:DPi-bound}  immediately follows 
    since the orthogonal projection
    $\Pi_{\NL(A)}$ has norm one, and the adjoint of a bounded linear
    operator has the same norm as the operator itself.
    {}
\end{proof}
To bound the right-hand side of \eqref{eq:DPi-bound} in a controlled way, it is helpful to
have the explicit bounds on $\norm[\LL(\HL_2,\HL_2)]{(AA^*)^{-1}}$ that are
available if $A$ has a right inverse according to 
\thref{lemma:boundedInverseAA*}.
\begin{lemma}
    \label{lemma:LipschitzContinuityProjection}
    Let  $A \in \LL(\HL_1,\HL_2)$ and  $Y
    \in \LL(\HL_2,\HL_1)$ satisfy $AY =\Id_{\HL_2}$.
    Then, $\Pi_{\NL(\cdot)}$ restricted to the ball $B_{ R_\Pi}(A)
    \subset\LL(\HL_1,\HL_2)$ 
    of radius
    \begin{equation}\label{eq:radius-pi-lip}
    R_{\Pi}\equiv R_{\Pi}(A)
    := \min\big\{\, \big[{2 \norm[\LL(\HL_2,\HL_1)]{Y}^2 
            (2\norm[\LL(\HL_1,\HL_2)]{A}+1)\big]^{-1}},1\,\big\}
    \end{equation}
    is Lipschitz continuous with
    Lipschitz constant
   $ 
    \LIP_{\Pi}=\LIP_{\Pi}(\|{A}\|_{\LL(\HL_1,\HL_2)},
    \|{Y}\|_{\LL(\HL_2,\HL_1)}),
  $ 
    which is non-decreasing in its arguments.
\end{lemma}
\begin{proof}
    Within this proof we simply denote by
    $\|{\cdot}\|$ all operator norms disregarding the respective domains and target 
    spaces of the operators involved. The convex combination of two operators
    $A_0,A_1\in B_{ R_\Pi}(A)$ will be abbreviated by $A_t:=(1-t)A_0 +tA_1 $
    for $t\in [0,1]$
    so that
    we can estimate the difference of their projections as
    $
    \|{\Pi_{\NL(A_0)}-\Pi_{\NL(A_1)}}\|
    \leq \int_0^1 \|{D(\Pi_{\NL(A_t)})(A_1 -A_0 )}\| \,\dd t.
    $
    With \eqref{eq:DPi-bound} in \thref{lemma:estimateFirstVariationP} 
    we may bound the integrand by
    $
    2 \|{A_t}\| \|{A_1 -A_0 }\| \|(A_tA_t^*)^{-1}\|.
    $
    Since $
    \|{A_t}\| \leq \|{A}\|+ R_{\Pi} \leq \|{A}\| +1
    $
    for all $t \in [0,1]$, the remaining task is to bound 
    $
    \|{(A_tA_t^*)^{-1}}\|.
    $
    According to 	\thref{lemma:boundedInverseAA*} the operator $AA^*$ is
    invertible satisfying the estimate $
    {\|(AA^*)^{-1}\|} \leq \|{Y}\|^2,
    $ and
    \thref{lemma:GLOffenInLLMitAbschaetzung} gives that if
    \begin{equation}\label{eq:AtAt*-condition}
    \|{A_tA_t^* - AA^*} \|
    \leq  ({2 \norm{Y}^2} )^{-1}
    \leq  ({2 \|(AA^*)^{-1}\|})^{-1},	
    \end{equation}
    then $A_tA_t^*$ is invertible and satisfies 
    $
    \|(A_tA_t^*)^{-1}\|
    \leq 2 \|(AA^*)^{-1}\|  
    \leq 2 \|{Y}\|^2.
    $
    In order to ensure condition \eqref{eq:AtAt*-condition} we observe 
 by virtue of the norm identity for
    operators and their adjoints, and by \eqref{eq:radius-pi-lip}
   \begin{align*} 
    \|{A_tA_t^* - AA^*} \|
    &\leq \|{A_tA_t^* - A_tA^*}\| + \|{A_tA^*-AA^*}\|
    \leq (\|{A_t}\|+\|{A}\|)\|{A_t-A}\|\\
    &
    \leq (2\|{A}\| + R_{\Pi})  R_{\Pi}
        \leq (2\|{A}\| +1)  R_{\Pi}
 {\leq} ({2\|{Y}\|^2})^{-1},
   \end{align*} 
    which is the desired prerequisite \eqref{eq:AtAt*-condition} that we needed
    to bound $\|(A_tA_t^*)^{-1}\|$.
    In summary,
    $
    \|{\Pi_{\NL(A_0)}-\Pi_{\NL(A_1)}}\|
    \leq 2(\|{A}\|+1)2\|{Y}\|^2 
    \|{A_0 -A_1}\|
    =:\LIP_{\Pi}(\norm{A},\norm{Y})\norm{A_0 -A_1}.$
    {}
\end{proof}

\section{The logarithmic strain constraint}
\label{subsection:explicitBoundaryCondition}
Recall from \eqref{eq:log-constraint} in the introduction
that we chose the \emph{logarithmic strain}
$ \SIGMA (\g):=\log (|\g'(\cdot)| )$ 
as the constraint determining the projected
flow \eqref{introtheorem:longTimeExistenceViaProjections:eq:ODE}. Here, we consider $\SIGMA $ as a mapping defined on the space
$W^{\frac32 p-2,2}_{\textnormal{ir}}(\R/\Z,\R^n)$ of regular and 
embedded fractional Sobolev curves $\g$ whose minimal
velocity $v_\g=\min_{\R}|\g'|$ is well-defined and strictly positive because of the
Morrey embedding, \thref{thm:morrey}. The logarithmic strain $\SIGMA $
takes values in the space of scalar-valued functions of class
$W^{\frac32 p-3,2}$, thus
loosing one order of differentiability in the image  space.
We need differentiability and local Lipschitz continuity of the 
differential  $D\SIGMA $, which is provided by
the following proposition for $\rho:=2$ and $s:=3p/2-3$.

\begin{proposition}
	\label{lemma:differentiabilityB}
	Let $s \in (0,1)$ and $\varrho > 1$ satisfy $s - 1/\varrho >0$.
	Then the logarithmic strain  $
	\ConstraintMap 
\colon W^{1+s,\varrho}_\textnormal{ir}(\R/\Z,\R^n) \to W^{s,\varrho}(\R/\Z)$
defined as $
		\ConstraintMap(\g) \ceq \log (|\g'|)$
	is Fréchet-differentiable and its differential is given by
    	\begin{equation}\label{eq:DB}
    		D\ConstraintMap(\gamma)
    		\colon
    		W^{1+s,\varrho}(\R/\Z,\R^n) \to W^{s,\varrho}(\R/\Z),
    		\quad
		D\ConstraintMap(\gamma) \, h 
		= 
		\big\langle{
			\tfrac {\gamma'(\cdot)} {\abs{\gamma'(\cdot)}}, 
			\tfrac{h'(\cdot)} {\abs{\gamma'(\cdot)}}
		}\big\rangle,
	\end{equation}
	for	$\g \in W^{1+s,\varrho}_\textnormal{ir}(\R/\Z,\R^n)$ applied to $h\in W^{1+s,\varrho}(\R/\Z,\R^n)$.
	Moreover $D\ConstraintMap$ is bounded and locally Lipschitz-continuous 
	with $
	\|{D\ConstraintMap(\g)} \|_{\LL(W^{1+s,\varrho},W^{s,\varrho})}
	\leq 
	C_{\ConstraintMap}$, and
	\begin{equation}
	\|D\ConstraintMap(\eta_1)-D\ConstraintMap(\eta_2)\|_{\LL(W^{1+s,\varrho},W^{s,\varrho})} 
	\leq 
	\LIP_{D\ConstraintMap} \, \halfnorm[s,\varrho]{\eta_1'-\eta_2'}
	\label{eq:DB-bound/DB-lip}
	\end{equation}
for all 
$\eta_1,\eta_2\in W^{1+s,\varrho}_\textnormal{ir}(\R/\Z,\R^n)$ with 
$\halfnorm[s,\varrho]{\g'- \eta_i'} < \BiLip(\g)/(2C_EC_P)$.
The constant $C_{\ConstraintMap}$ and the Lipschitz
constant $\LIP_{D\ConstraintMap}$ depend only on $n$, $p$, non-increasingly
on $v_\g$, and non-decreasingly on $\halfnorm[s,\varrho]{\g'}$.
\end{proposition}
\begin{proof}
Define $\varphi \colon \R^n\setminus\{0\}  \to \R$, 
$\varphi(x) \ceq \log(|x|)$
and observe that $\varphi$ is smooth. Its first three partial
derivatives are given by $\partial_i\varphi(x)=x_i/|x|^2$, 
$\partial_i\partial_j\varphi(x)=(\delta_{ij}-2x_ix_j/|x|^2)|x|^{-2}$, and
$$
\partial_i\partial_j\partial_k\varphi(x)=-\tfrac{2}{|x|^4}\Big(
\delta_{ij}x_k+\delta_{ik}x_j+\delta_{jk}x_i+\tfrac{4x_ix_jx_k}{|x|^2}\Big),
$$
so that the operator norms $\|D^k\varphi(x)\|$ are bounded
from above by $C(k)|x|^{-k}$ for\footnote{This is actually true for all
orders of differentiability $k\in\N$.} $k=1,2,3.$
With our candidate expression for the Fréchet derivative from \eqref{eq:DB} and Taylor's theorem, we have pointwise for each $u \in \R/\Z$:
\begin{align*}
	\big(
	\ConstraintMap(\g+h)
	-
	\ConstraintMap(\g)
	-
	(D\ConstraintMap(\g) \, h)
	\big)(u)
	&=
	\varphi ( (\g+h)' (u) )
	-
	\varphi  (  \g'(u) )
	-
	D \varphi (  \g'(u)) \, h'(u) \\
	&=
	\textstyle
	\int_0^1 (1-\theta)D^2\varphi \big(  (\g + \theta \, h)'(u)\big) \,(h'(u), h'(u))  \, \dd \theta.
\end{align*}
Thus, in order to show that our candidate is indeed the Fréchet derivative, 
we merely have to show that the integral on the right is dominated by 
$\|h\|_{W^{1+s,\varrho}}^2$ for sufficiently small $\|h\|_{W^{1+s,\varrho}}^2$. Hence we may assume that
 $\halfnorm[s,\varrho]{h'} \leq \BiLip(\g)/(2C_EC_P)$, so that by means
of Corollary \ref{cor:bilip2} 
$\nabs{\g_\theta'(u)}  \geq v_\g/2$ for all $\theta\in [0,1]$ and $u\in\R/\Z$
where we have abbreviated $\g_\theta \ceq \g + \theta \, h$. 
Therefore, by definition of the seminorm, boundedness and Lipschitz continuity of 
$D^2 \varphi$, we have for any $\theta\in [0,1]$
\begin{align}
    \label{eq:estimateD2phi}
        \|D^2\varphi (  &\g_\theta'( \cdot))\|_{W^{s,\varrho}}
        \le \|{D^2\varphi (  \g_\theta'( \cdot))}\|_{L^\varrho} + 
	\halfnorm[s, \varrho]{D^2\varphi (\g_\theta'( \cdot))}\notag\\
        &\le 
	v_\g^{-2}\big(4C(2)+C(3)v_\g^{-1}\halfnorm[s,\varrho]{\g_\theta'} \big)
        \le 
	v_\g^{-2}\big(4C(2)+C(3)v_\g^{-1}
	(\halfnorm[s,\varrho]{\g'} + \halfnorm[s, \varrho]{h'})\big).
\end{align}
By the product rule from 
\thref{proposition:WsrhoClosedunderPointwiseMultiplication} and by 
\eqref{eq:estimateD2phi}, we have
\begin{align*}
\|
	\textstyle
	\int_0^1 (1-\theta)D^2\varphi \big(\g_\theta'(\cdot)\big) 
	&\,(h'(\cdot), h'(\cdot))  \, \dd \theta
	\|_{W^{s,\varrho}}
	\leq 
	\textstyle	
	\int_0^1 	
\|{
D^2\varphi \big(  (\g_\theta'(\cdot)\big) \,(h'(\cdot), h'(\cdot))
	}\|_{W^{s,\varrho}} \, \dd \theta
	\\
	&\overset{\eqref{eq:estimateD2phi}}{\leq }
	C(n,s,\varrho)^2 v_\g^{-2}\big(4C(2)+C(3)v_\g^{-1}
	(\halfnorm[s,\varrho]{\g'} + \halfnorm[s, \varrho]{h'})\big)
	\|{h'}\|_{W^{s,\varrho}}^2.
\end{align*}
Using the bounds on $D\varphi$ and $D^2\varphi$ on $\R^n\setminus\{0\}$
we obtain  
similarly to \eqref{eq:estimateD2phi}
\[
    \nnorm{D\ConstraintMap(\gamma)} = \sup\big\{ 
    \norm[W^{\sigma,\rho}]{D\varphi(\g'(\cdot))\,h'(\cdot)}\colon
    \|{h}\|_{W^{1+s,\varrho}}\le 1\big\}\leq C_\ConstraintMap
\]
where  $C_\ConstraintMap$ depends only on $s, \varrho$, non-increasingly on 
$v_\g$ and non-decreasingly on $[{\gamma'}]_{s,\varrho}$.
Repeating the whole proof with $\varphi$ replaced by $D \varphi$, 
one shows that 
$\ConstraintMap$ is actually twice Fréchet-differentiable with uniformly 
bounded operator norm $\|D^2\ConstraintMap(\eta)\|$ for all $\eta$ in
 the ball of 
radius $r:=\BiLip(\g)/(2C_EC_P)$ around $\g$.
Thus $D\ConstraintMap$  is locally Lipschitz-continuous with Lipschitz constant 
$\LIP_{D\ConstraintMap} =  \sup_{\eta \in B_{r}(\g)} \nnorm{D^2 \ConstraintMap(\eta)} < \infty$ 
depending non-increasingly on $v_\g$ and non-decreasingly
on $\halfnorm[s,\varrho]{\g'}$.
\end{proof}

In view of the results of Section       
\ref{sec:4} we need 
to show that the differential $D\SIGMA (\gamma)$ possesses 
a uniformly bounded right inverse.  
The following lemma and its proof are simpler variants of 
\cite[Lemma 2.4]{ScholtesSchumacherWardetzky:2019:DiscreteElasticae}.
\begin{lemma}
\label{lemma:uniformBoundSurjectivityCharacterization}
Let $s\in (0,1)$ and  $\varrho >1$ satisfy $s-1/\varrho >0$. 
For all $\g \in W^{1+s,\varrho}_\textnormal{ir}(\R/\Z,\R^n)$ the  differential $D\ConstraintMap(\g)$ 
admits a right inverse 
$Y_{\g}\in \LL\big(W^{s,\varrho}(\R/\Z), W^{1+s,\varrho}(\R/\Z,\R^n)\big)$
satisfying
$
	\|{ Y_\gamma}\|_{\LL(W^{s,\varrho},W^{1+s,\varrho})} \leq K_{Y}
$,
where $K_{Y}= K_{Y}(n,s,\varrho,v_\gamma,[{\gamma'}]_{s,\varrho})>0$ 
depends non-increasingly on the minimal velocity
$v_\g$, and non-decreasingly on $[{\g'}]_{s,\varrho}$.
\end{lemma}
\begin{proof}
Fix  a  curve $\g\in W^{1+s,\varrho}_\textnormal{ir}(\R/\Z,\R^n)$. Our goal
is to construct a bounded linear map
$
 Y_\g\colon W^{s,\varrho}(\R/\Z) \to W^{1+s,\varrho}(\R/\Z,\R^n)
$
by assigning to a given function $\lambda\in W^{s,\varrho}(\R/\Z)$
the vector-valued function
\begin{equation}\label{eq:right-inverse}
		(Y_\g \,\lambda)(x)
 		\ceq
 		\textstyle
		\int_0^x
			\big({
				\lambda(y)\, \gamma'(y) 
				+ 
				P_{\gamma'(y)^\perp} \, V_\lambda \big)
			}
		\, \dd y, \quad x\in\R.
\end{equation}
Here, 
$V_\lambda \in \R^n$ is a vector (depending on the given $\lambda$) to be determined later
to secure that $Y_\g\lambda$ is $1$-periodic,
and
$
	P_{\gamma'(y)^\perp}
 	\ceq 
 	\Id_{\R^n}  -  \g'(y)  \g'(y)\transp /|\g'(y)|^2 $
 	is the orthogonal projection onto the orthogonal complement of $\g'(y)$.
From the product rule \thref{proposition:WsrhoClosedunderPointwiseMultiplication} we learn that
$Y_\g \, \lambda$ is indeed a member of $W^{1+s,\varrho}(\R/\Z,\R^n)$:
The mappings $\lambda$, $\g'$ and $\g'/|\g'|$ are all of class $W^{s,\varrho}$, so multilinear combinations 
of them are again of class $W^{s,\varrho}$; after integration, we end up with 
$Y_\g(\lambda) \in W^{1+s,\varrho}(\R/\Z,\R^n)$ where we assumed the $1$-periodicity
of $Y_\g\lambda$ by a suitable choice of the vector $V_\lambda$  below. 
More precisely, we may use 
\thref{lemma:multiplicativeInverseSobolevFunction} in addition 
to obtain
\begin{align}\label{eq:Ygl-est}
	\halfnorm[s,\varrho]{(Y_\g\lambda)'}
	&\leq
	C \, \halfnorm[{s,\varrho}]{\g'} \, \halfnorm[s,\varrho]{\lambda}
	+
	C \, (1+v_\g^{-4}\halfnorm[s,\varrho]{\g'}^{4}) \, |{V_\lambda}|.
\end{align}
We conclude from \eqref{eq:DB} that
$
		D\ConstraintMap(\gamma) \, Y_\g \, \lambda
	=
	\langle{ \g'/|\g'|^2 , \lambda \, \g' + P_{\gamma'^\perp} \, V_\lambda }\rangle_{\R^n} = \lambda
	,
$
no matter what choice we make for $V_\lambda\in\R^n$. Regarding that choice notice first that 
by the Morrey embedding \thref{thm:morrey}, $ (Y_\g\lambda)'$  is continuous and $1$-periodic.
Hence $Y_\g\lambda $ is $1$-periodic if and only if $
	0 = (Y_\g \,\lambda)(1)-  (Y_\g\,\lambda)(0)
	= \textstyle \int_0^1 \lambda(y) \,\g'(y) \, \dd y + \Theta_\g \, V_\lambda,$
where $\Theta_\g\colon \R^n\to\R^n$ is the linear map defined by
$$
	\textstyle
	\Theta_\gamma
	\ceq
	\int_0^1 P_{\gamma'(u)^\perp}\,\dd u
	=
	\int_0^1 ({\Id_{\R^n} - \g'(u) \g'(u)\transp /|\g'(u)|^2})  \, \dd u.
$$
Hence, assuming $\Theta_\g$ is boundedly invertible, we put $V_\lambda  \ceq  - \Theta_\g^{-1} \int_0^1 
\lambda(y) \,\g'(y) \, \dd y$ and observe that
$|{V_\lambda} |
\leq C \, \|{\Theta_{\gamma}^{-1}} \| \, \|{\lambda}\|_{W^{s,\varrho}}$.
By the Morrey embedding, we have $\g'/|\g'| \in C^{0,\alpha}(\R/\Z,\R^n)$ 
with $\alpha = s- 1/\varrho >0$.
So finally, \thref{lem:AquavitLemma} below applied to $\varTheta:=
\Theta_{\gamma}$ 
with $\tau:=\g'/|\g'|$ and $\mathcal{H}:=\R^n$
provides us with an explicit bound on the operator
norm $\|{\Theta_{\gamma}^{-1}}\| =1/ 
\lambda_{\min}(\Theta_\gamma)$ and in doing so yields the lacking invertibility.
Note that the unit tangent
$\tau$ is indeed \emph{not} contained in any hemisphere of 
$\S^{n-1}$,
since otherwise there existed 
a vector $v \in \S^{n-1}$ such that 
$\langle v,\tau(t)\rangle>0$ for all $t \in \R/\Z$.
Then, also $0<\langle v, \tau(t)\rangle \abs{\g'(t)} = \langle v, \g'(t)\rangle$ and so $0<\int_0^1 \langle v, \g'(t)\rangle \, \dd t = \langle v , \g(1)-\g(0)\rangle=0$ which is a contradiction.
The H\"older constant in \thref{lem:AquavitLemma} is compatible to the one defined via the Euclidean distance 
as $|{x-y}| = 2 \sin (\frac 1 2 \sphericalangle(x,y))= 2 \sin (\frac 1 2 d_{\mathbb{S}}(x,y))$ for 
$|{x}|=|{y}|=1$ and as such we have $|{x-y} |\le d_{\mathbb{S}}(x,y) \le 
\tfrac \pi 2 |{x-y}|$. 
As $\lambda\mapsto
V_\lambda$ is linear, one has linearity of $\lambda\mapsto Y_\g\lambda$ as well. 
Finally, \thref{lem:AquavitLemma} yields the estimate $|V_\lambda|\le C(s,\varrho,v_\g,[\g']_{s,\varrho})
\|\lambda\|_{W^{s,\varrho}}$ with a constant $C$ depending non-increasingly on $v_\g$ and non-decreasingly
on $[\g']_{s,\varrho}$,
and this inequality 
inserted in \eqref{eq:Ygl-est} together with a straightforward
estimate of the $L^\varrho$-norm of $Y_\g\lambda$ 
implies the desired bound on the operator norm
$\|Y_\g\|_{\LL(W^{s,\varrho},W^{1+s,\varrho})}$.
\end{proof}

\begin{lemma}\label{lem:AquavitLemma}
Let $\mathcal{H}$ be a real 
Hilbert space of dimension at least $2$, let $0<\alpha \leq 1$ and denote by 
$\mathbb{S} \subset \mathcal{H}$ the unit sphere.
Let $\tau \colon \R/\Z \to \mathbb{S}$ be a closed $\alpha$-H\"older-continuous curve 
which is not contained in any hemisphere and has H\"older constant $C = \sup_{s\neq t} 
d_{\mathbb{S}} (\tau(t),\tau(s)) \, \abs{s-t}_{\R/\Z}^{-\alpha}$.
Here $d_{\mathbb{S}}$ denotes the angular distance on the sphere.
Then the smallest eigenvalue of the self-adjoint operator
$
	\varTheta \ceq 
	\textstyle \int_0^1 \big({ \Id_{\mathcal{H}}  -   \tau(t) \tau(t)^\top }\big) 
	\, \dd t
$
is bounded from below by
\begin{align*}
	\lambda_{\min}(\varTheta) \geq  \tfrac{\pi^2}{(2+1/\alpha)^2} \big(
	{\tfrac{\pi}{(2 + 4 \, \alpha) \,C}}\big)^{1/\alpha}.
\end{align*}
\end{lemma}
\begin{proof}
Let $v \in \mathbb{S}$ be an eigenvector of $\varTheta$ corresponding 
to $\lambda_{\min}(\varTheta)$.
For $\delta \in (0,\pi/2]$ define the polar caps 
$
	\mathbb{B}_{\pm \delta} \ceq \{ e \in \mathbb{S} \mid  
\pm \ninnerprod{v,e}_\mathcal{H} >
 \cos(\delta)\}
$
and the set
$
	I_\delta \ceq \{ t \in  \R/\Z \mid \tau(t) \not \in \mathbb{B}_{\delta} 
\cup \mathbb{B}_{-\delta} \}
$.
Observe that
\begin{align*}
	\lambda_{\min}(\varTheta)
	=
	\textstyle
	\int_0^1 (1 - \ninnerprod{v, \tau(t)}_{\mathcal{H}}^2) \, \dd t
	\geq
	\textstyle
	\int_{I_\delta} (1 - \ninnerprod{v, \tau(t)}_{\mathcal{H}}^2) \, \dd t
	\geq \nabs{I_\delta}  \, \sin^2(\delta) \geq 	\tfrac{1}{4} \nabs{I_\delta}  \, \delta^2.
\end{align*}
By assumption, the curve cannot be contained in only one of the polar caps $\mathbb{B}_{\pm \delta}$.
So there must be at least one $t_0 \in \R/\Z$ so that $\tau(t_0)$ lies on the equator, i.e., 
$\ninnerprod{v,\tau(t_0)}_\mathcal{H} = 0$.
From there, the curve has to travel a distance at least 
$\pi/2 -\delta$ to reach $\mathbb{B}_{\delta}$; since 
$\tau$ is $\alpha$-H\"older continuous, it requires at least $\big( |\pi/2 -\delta|
 \, C^{-1} \big)^{1/\alpha}$ distance in the pre-image for that. The same applies for reaching
 $\mathbb{B}_{-\delta}$. Since $\tau$ is closed, it has to traverse the equator twice. 
This provides us with
$
	\nabs{I_\delta} 
	\geq 4 \, \big( |\pi/2 -\delta| \, C^{-1} \big)^{1/\alpha}
$ 
and thus with $
	\lambda_{\min}(\varTheta) \geq \big( |\pi/2 -\delta| \, C^{-1} \big)^{1/\alpha} \, \delta^2.$
Maximizing the right-hand side for $0<\delta<\pi/2$ leads to 
$\delta = \frac{\pi}{2 + 1/ \alpha}$ 
and substitution leads to the stated lower bound for $\lambda_{\min}(\varTheta)$.
\end{proof}

\section{Long Time Existence}\label{sec:6}
For a given curve $\g\in W_\textnormal{ir}^{\frac32 p-2,2}(\R/\Z,\R^n)$
define the radius
\begin{equation}\label{eq:radius-lip-proj}
R_{M,\SIGMA}=R_{M,\SIGMA}(\g):=
\frac{\min\left\{\BiLip(\g),1\right\}}{(2C_EC_P+1)(\LIP_{D\SIGMA}+1)
    (2K_Y^2+1)(2C_\SIGMA+1)},
\end{equation}
where $C_E$ and $C_P$ are the constants in the Morrey embedding
\eqref{eq:morrey-embedding} and in Poincar\'e's inequality
\eqref{eq:poincare}, and  $C_\SIGMA$ and $\LIP_{D\SIGMA}$ 
are the bound  on and the Lipschitz constant of
the differential $D\SIGMA(\g)$ in 
\thref{lemma:differentiabilityB}, and $K_Y$ is the bound on its 
right inverse $Y_\g$ in 
\thref{lemma:uniformBoundSurjectivityCharacterization}.   
\begin{lemma}
    \label{corollary:LipschitzContinuityRHSProjection}
    For $\gamma \in W^{\frac 3 2 p-2,2}_{\ir}(\R/\Z,\R^n)$ with
    $p \in (\frac 7 3, \frac 8 3)$ 
 the projected gradient $\nabla_\SIGMA\Mpq[p,2]$ is bounded and
 Lipschitz continuous on the ball $B_{R_{M,\SIGMA}}(\g)
    \subset W^{\frac32 p-2,2}(\R/\Z,\R^n)$ with bound $\widetilde{C}_1$
    and Lipschitz constant
$    \LIP_{\nabla_\SIGMA\Mpq[p,2]}$
    depending on $n,p$,  non-increasingly 
    on $\BiLip(\g)$ and $v_\g$, and non-decreasingly
    on $\halfnorm[\frac32 p-3,2]{\g'}$\,.
\end{lemma}
\begin{proof}
    Comparing the definition of $R_{M,\SIGMA}$ with that of the Lipschitz
    radius $R_M$ in \eqref{eq:lip-radius} we find $R_{M,\SIGMA}\le R_M   $,  and
    therefore all previous results proved on $B_{R_M}(\g)\subset
    W^{\frac32 p-2,2}(\R/\Z,\R^n)$ hold on the smaller concentric
    ball $B_{R_{M,\SIGMA}}(\g)$
    as well. 
    In particular, by
    \thref{cor:bilip2} 
    \begin{align}
    \BiLip(\eta) & \ge \BiLip(\g)/2\AND
    \halfnorm[\frac32 p-3,2]{\eta'}\le 
    \halfnorm[\frac32 p-3,2]{\g'}+R_{M,\Sigma}   \le 
    \halfnorm[\frac32 p-3,2]{\g'}+1
    \label{eq:bilip-comparison4}
    \end{align}
    for all $\eta\in B_{R_{M,\SIGMA}}(\g)$.
Therefore, we can use \eqref{theorem:secondVariationMp2:eq:estFirstVariation}
    in \thref{theorem:secondVariationMp2} 
    to estimate for $\mathcal{E}:=\Mpq[p,2]$, $\mathcal{H}:=W^{\frac32 p-2,2}(\R/\Z,\R^n)$ and each
    $\eta\in B_{R_{M,\SIGMA}}(\g)\subset\mathcal{H}$ 
    \begin{align}\label{eq:nabla-bound1}
    \norm[\mathcal{H}]{\nabla\mathcal{E}(\eta)}
    & \overset{\eqref{theorem:secondVariationMp2:eq:estFirstVariation}}{\le}
    C_1(\BiLip(\eta),\halfnorm[\frac32 p-2,2]{\eta'})
    \overset{
        \eqref{eq:bilip-comparison4}}{\le}
    C_1(\BiLip(\g)/2,\halfnorm[\frac32 p-2,2]{\g'}+1)=:\widetilde{C}_1,
    \end{align}
    where $\widetilde{C}_1 $ still depends non-increasingly on $\BiLip(\g)$ and 
    non-decreasingly on $\halfnorm[\frac32 p-2,2]{\g'}$ by the monotonicity
    properties of $C_1$ stated in \thref{theorem:secondVariationMp2} (suppressing its dependence on $n,p$).
    Moreover, recall from \thref{corollary:DintMp2locallyLipschitz}
    that
    \begin{equation}\label{eq:lip-gradient3}
    \norm[\mathcal{H}]{\nabla\mathcal{E}(\eta_1)-
        \nabla\mathcal{E}(\eta_2)}\le\LIP_{\nabla\mathcal{E}}
    \norm[\mathcal{H}]{\eta_1-\eta_2}
    \end{equation}
    for all $\eta_1,\eta_2\in B_{R_{M,\SIGMA}}(\g)$,
    where the Lipschitz constant 
    $
    \LIP_{\nabla\mathcal{E}}
    $ 
    depends non-increasingly on $\BiLip(\g)$,
    and non-decreasingly on $\halfnorm[\frac32 p-2,2]{\g'}$.
    In view of the logarithmic constraint we infer from 
    \eqref{eq:DB-bound/DB-lip} in \thref{lemma:differentiabilityB}
    for $\eta_1,\eta_2\in B_{R_{M,\SIGMA}}(\g)$  the inequality
    \begin{equation}\label{eq:DB-lip3}
    \norm[\LL(\mathcal{H},W^{\frac32 p-3,2})]{
        D\SIGMA(\eta_i)-D\SIGMA(\g)}\le
    \LIP_{D\SIGMA}\cdot R_{M,\SIGMA}\quad\Fo i=1,2,
    \end{equation}
    where the Lipschitz constant $\LIP_{D\SIGMA}$ depends
    non-increasingly on $v_\g$, and non-decreasingly on 
    $\halfnorm[\frac32 p-3,2]{\g'}$.
    With $\norm[\LL(\mathcal{H},W^{\frac32 p-3,2})]{D\SIGMA (\g)}\le
    C_\SIGMA$ by means of  
    \thref{lemma:differentiabilityB}, and with 
    $\norm[\LL(W^{\frac32 p-3,2},\mathcal{H})]{Y_\g}\!\le\!^ K_Y$
    according to  
    \thref{lemma:uniformBoundSurjectivityCharacterization}
    we find for the radius $R_\Pi(D\SIGMA(\g))$ of 
    \thref{lemma:LipschitzContinuityProjection} for $A:=D\SIGMA(\g)$ the
    lower estimate  (see \eqref{eq:radius-pi-lip}) 
    \begin{align}\label{eq:lower-rpi-bound}
    R_\Pi(D\SIGMA(\g)) \overset{\eqref{eq:radius-pi-lip}}{\ge}
    \min\left\{[{2K_Y^2(2C_\SIGMA +1)}]^{-1},1\right\}
    \overset{\eqref{eq:radius-lip-proj}}{\ge}\LIP_{D\Sigma} R_{M,\SIGMA}.
    \end{align}
    Thus, \eqref{eq:DB-lip3} implies that for $i=1,2$, we
    have $D\SIGMA(\eta_i)\in
    B_{R_\Pi(D\SIGMA(\g))}(D\SIGMA(\g))$  in the space $\LL(\mathcal{H},
    W^{\frac32 p-3,2})$ of linear mappings,
    as long as $\eta_1,\eta_2\in B_{R_{M,\SIGMA}}(\g)$.
    This in turn allows us to combine the abstract 
    \thref{lemma:LipschitzContinuityProjection} for $A:=D\SIGMA(\g)$ 
    with the Lipschitz estimate \eqref{eq:DB-bound/DB-lip} for 
    the differential $D\SIGMA(\cdot)$ in 
    \thref{lemma:differentiabilityB}     to
    conclude
      \begin{align}\label{eq:projection-lipschitz3}
    \| \Pi_{\NL(D\SIGMA(\eta_1))}&  -
    \Pi_{\NL(D\SIGMA(\eta_2))}\|_{\LL(\mathcal{H},\mathcal{H})}\\
    \le
   & \LIP_\Pi  \norm[\LL(\mathcal{H},W^{\frac32 p-3,2})]{
        D\SIGMA(\eta_1)-D\SIGMA(\eta_2)}
  {\le} \LIP_\Pi \LIP_{D\SIGMA}
    \halfnorm[\frac32 p-3,2]{\eta_1'-\eta_2'},\notag
    \end{align}
    where we took for simplicity $\LIP_\Pi=\LIP_\Pi(C_\SIGMA,K_Y)$ because those two
    arguments bound the norms of $D\SIGMA$ and $Y_\g$, and $\LIP_\Pi$
    is non-decreasing in both of its arguments; see 
    \thref{lemma:LipschitzContinuityProjection}.
    Finally, we infer from \eqref{eq:bilip-comparison4},
    \eqref{eq:lip-gradient3}, and \eqref{eq:projection-lipschitz3}
    the estimate
    \begin{align}
    \label{eq:final-lip-estimate-proj}
   & \|\nabla_\SIGMA  \mathcal{E}(\eta_1)  -
    \nabla_\SIGMA\mathcal{E}(\eta_2)\|_{\mathcal{H}}
    = 
    \norm[\mathcal{H}]{\Pi_{\NL(D\SIGMA(\eta_1))}\big(\nabla\mathcal{E}
        (\eta_1)\big)-\Pi_{\NL(D\SIGMA(\eta_2))}\big(\nabla\mathcal{E}(\eta_2)\big)}\notag\\
    \le  &
    \norm[\LL(\mathcal{H},\mathcal{H})]{\Pi_{\NL(D\SIGMA(\eta_1))}-
        \Pi_{\NL(D\SIGMA(\eta_2))}}
    \norm[\mathcal{H}]{\nabla\mathcal{E}
        (\eta_1)}
    +  \norm[\LL(\mathcal{H},\mathcal{H})]{\Pi_{\NL(D\SIGMA(\eta_2))}}\cdot 
    \norm[\mathcal{H}]{\nabla\mathcal{E}
        (\eta_1)-\nabla\mathcal{E}(\eta_2)}\notag\\
  & {\le}
    (\widetilde{C}_1\LIP_{\Pi}\LIP_{D\SIGMA}
    +\LIP_{\nabla\mathcal{E}})\norm[\mathcal{H}]{\eta_1-\eta_2}
        \end{align}
    for all $\eta_1,\eta_2\in B_{R_{M,\SIGMA}}(\g)$ as $\Pi$ has operator norm bounded by $1$, which proves the claim with the Lipschitz constant
    $
    \LIP_{\nabla_\SIGMA \mathcal{E}}:=\widetilde{C}_1\cdot\LIP_{\Pi}\LIP_{D\SIGMA} 
    +\LIP_{\nabla\mathcal{E}}
    $
    depending non-increasingly on $\BiLip(\g)$ and $v_\g$, and
    non-decreasingly on $\halfnorm[\frac32 p-3,2]{\g'}$. Indeed, as mentioned above,
    $\widetilde{C}_1 $, $\LIP_{\nabla\mathcal{E}}$, and
    $\LIP_{D\SIGMA}$  depend non-decreasingly on $\halfnorm[\frac32 p-3,2]{\g'}$, and non-increasingly on $\BiLip(\g)$ or $v_\g$, respectively.  Both arguments $C_\SIGMA$ and $K_Y$ 
    of  $\LIP_{\Pi}(\cdot,\cdot)$
    are  non-increasing in $v_\g$ and non-decreasing in
    $\halfnorm[\frac32 p-3,2]{\g'}$, which concludes the proof.
    {}
\end{proof}
With the exact same proof as for \thref{theorem:shortTimeExistence}
we can show the following short time existence for the projected
gradient flow \eqref{introtheorem:longTimeExistenceViaProjections:eq:ODE}
for integral Menger curvature $\Mpq[p,2]$.
Notice in that context
that the right-hand side of \eqref{introtheorem:longTimeExistenceViaProjections:eq:ODE} has the same upper bound
as the gradient flow \eqref{eq:gradflow} since it differs from that only
by an orthogonal projection with operator norm $1$. 
\begin{corollary}[Short time existence]\label{cor:short-time-proj-flow}
    For $p\in (\frac73 ,\frac83 )$ and $\g_0\in W_\textnormal{ir}^{\frac32 p-2,2}(
    \R/\Z,\R^n)$ there exists a time $T>0$ and a unique mapping
    $
    t\mapsto\g(\cdot,t)\in C^{1,1}\big([0,T],W^{\frac32 p-2,2}(\R/\Z,\R^n)\big)
    $
    such that $\g$ solves \eqref{introtheorem:longTimeExistenceViaProjections:eq:ODE}
    and satisfies
    \begin{equation}\label{eq:equal-velocities}
    |\g'(\cdot,t)|=|\g_0'(\cdot)| \quad\textnormal{on $\R/\Z$\,\, for all $t\in [0,T]$.}
    \end{equation}
    The mapping $t\mapsto\Mpq[p,2](\g(t))$ is non-increasing on $[0,T]$, and the
    minimal time $T$ of existence is estimated from below as
    \begin{equation}\label{eq:min-T-proj}
    T\ge T_\textnormal{min}(\gamma_0):=\frac{R_{M,\SIGMA}(\g_0)}{
        C_1(n,p,\BiLip(\g_0)/2,\halfnorm[\frac32 p-3,2]{\g_0'}+1)},
    \end{equation}
    where $R_{M,\SIGMA}(\g_0)$ is the radius defined in \eqref{eq:radius-lip-proj}
    for $\g_0$ instead of $\g$, and $C_1$ is the bound for the differential of
    $\Mpq[p,2]$ in \eqref{theorem:secondVariationMp2:eq:estFirstVariation} of 
    \thref{theorem:secondVariationMp2}. Moreover, the image
    $\g(\cdot,[0,T_\textnormal{min}])$ is contained in
    the ball $B_{R_{M,\SIGMA}}(\g_0)\subset W_\textnormal{ir}^{\frac32 p-2,2}(\R/\Z,\R^n)$.
\end{corollary}
\begin{proof}
    According to our remark above referring to the proof of
    \thref{theorem:shortTimeExistence}  it suffices to verify 
    \eqref{eq:equal-velocities} and the energy decay in time. The conservation of velocities 
    directly follows from the projected gradient flow \eqref{introtheorem:longTimeExistenceViaProjections:eq:ODE}, since this implies --
    as pointed out in the introduction in \eqref{eq:invarianceBoundaryCondition} --
    that the logarithmic constraint $\SIGMA$ is a conserved quantity in time, i.e., 
    $
    \SIGMA(\g(\cdot,t))=\SIGMA(\g_0(\cdot))$ for all $ t\in [0,T]$,
    thus leading to \eqref{eq:equal-velocities} since the logarithm is injective
    on $(0,\infty)$. For the energy decay we abbreviate $\mathcal{E}:=\Mpq[p,2]$, $\HL:=W^{\frac32 p-2,2}(\R/\Z,\R^n)$, and differentiate the energy of $
    \g(\cdot,t)$ with respect
    to time and obtain similarly as in \eqref{eq:energy-decay1}  by means of
    \eqref{introtheorem:longTimeExistenceViaProjections:eq:ODE}
    $
  \textstyle  \frac{d}{dt} \mathcal{E}(\g(\cdot,t)) =
  \textstyle\langle \nabla\mathcal{E}(\g(\cdot,t)),
    \frac{\partial}{\partial t}\g(\cdot,t)\rangle_{\HL}
  {=}
    -\langle \nabla\mathcal{E}(\g(\cdot,t)), \nabla_\SIGMA\mathcal{E}(\g(\cdot,t))
    \rangle_{\HL} =
    - \norm[\HL]{\nabla_\SIGMA\mathcal{E}(\g(\cdot,t))}^2
    \le 0$ for all $t
    \in (0,T),$
    which implies the energy decay in time. 
    {}
\end{proof}
Our goal is to extend the short time solution by a fixed time step depending only
on the initial configuration. To that extent
we first establish a priori estimates for
the bilipschitz constant and
the fractional seminorm along the flow 
\eqref{introtheorem:longTimeExistenceViaProjections:eq:ODE}.
\begin{proposition}[A priori estimates]\label{prop:bilip-seminorm-evolution}
    Let 
    $\g_0\in 
    W_\textnormal{ir}^{\frac32 p-2,2}(\R/\Z,\R^n)$ for $p\in
    (\frac73, \frac83 )$. Then
    the energy $\Mpq[p,2](\g(\cdot,t))$ of
    any solution 
    $
    t\mapsto\g(\cdot,t)\in C^1([0,T^*],
    W^{\frac32 p-2,2}(\R/\Z,\R^n)) 
    $ 
    of 
    \eqref{introtheorem:longTimeExistenceViaProjections:eq:ODE} on the time
    interval $[0,T^*]$ for some $T^*>0$
    is non-increasing in time,  and the
    velocities  of $\g(\cdot,t)$ satisfy
    \begin{equation}\label{eq:equal-velocities2}
    \abs{\g'(\cdot,t)}=\abs{\g_0'}\quad\Foa t\in [0,T^*].
    \end{equation}
    Moreover, 
    the bilipschitz constant $\BiLip(\g(\cdot,t))$ and
    the seminorm $\halfnorm[\frac32 p-3,2]{\g'(\cdot,t)}$ of $\g(\cdot,t)$
    satisfy
    \begin{align}\label{eq:bilip-evolution}
    \BiLip(\g(\cdot,t))& \ge v_{\g_0} \mathfrak{b}(\Mpq[p,2](\g_0)),\\
    \label{eq:seminorm-evolution}
    \halfnorm[\frac32 p-3,2]{\g'(\cdot,t)}& \le \mathfrak{s}(\Mpq[p,2](\g_0))
    \halfnorm[\frac32 p-3,2]{\g_0'}
    \big(1+ v_{\g_0}^{-1}
    \halfnorm[\frac32 p-3,2]{\g_0'}^{3p-6}
    \big)^{1/2}
    \end{align}
    for all $t\in [0,T^*]$, where the functions
    $\mathfrak{b}(\cdot)$ and $\mathfrak{s}(\cdot)$ are 
    strictly positive, and $\mathfrak{b}(\cdot)$ is
    non-increasing  whereas $\mathfrak{s}(\cdot)$ is non-decreasing
    on $(0,\infty)$.  Finally,
    the time derivative $t\mapsto\partial\g(\cdot,t)/\partial t$ is Lipschitz
    continuous with a uniform Lipschitz constant $\LIP_{\partial_t\g}$ depending
    non-decreasingly  on  the initial energy $\Mpq[p,2](\g_0)$,  on the
    seminorm $
    \halfnorm[\frac32 p-3,2]{\g_0'}$, and non-increasingly on $v_{\g_0}$.
\end{proposition}
\begin{proof}
Abbreviate  $\EL:=\Mpq[p,2]$.
  The energy decay in time as well as the conservation of velocity
    \eqref{eq:equal-velocities2}
    can be shown by the exact same observations as in the previous proof.
        The energy dependent lower bound for the bilipschitz constant 
    established in \cite[Proposition 2.1]{BlattReiter:2015:Menger} implies that
    the non-increasing function $\mathfrak{b}:(0,\infty)\to\R$ defined as
    \begin{equation}\label{eq:bilipschitz-function}
    \mathfrak{b}(E)
    :=\inf\{\BiLip(\Gamma) : \Gamma\in W_{ir}^{\frac32 p-2,2}(\R/\Z,\R^n),\,
    |\Gamma'|\equiv 1,\,\EL(\Gamma)\le E\}
    \end{equation}
    is strictly positive for each $E\in (0,\infty)$. Introducing the 
    arc length parametrization $\Gamma(\cdot,t)$ of the solution $\g(\cdot,t)$
    of 
    \eqref{introtheorem:longTimeExistenceViaProjections:eq:ODE} as in
    Appendix \ref{section:Arclength}, we infer from 
    \thref{lemma:arclengthParametrizationAndBiLipschitzConstants}
    that 
    \begin{equation}\label{eq:bilip-evolution3}
    \BiLip(\g(\cdot,t))\ge v_{\g(\cdot,t)}\BiLip(\Gamma(\cdot,t))\overset{
        \eqref{eq:equal-velocities2}}{=}v_{\g_0}\BiLip(\Gamma(\cdot,t))
    \end{equation}
    for all $t\in [0,T^*]$, where we also used the conservation of velocities
    \eqref{eq:equal-velocities2} implying identical minimal velocities
    $v_{\g(\cdot,t)}$ at all times $t\in [0,T^*]$.
    Integral Menger curvature is a parameter-invariant energy
    so that the monotonicity of the mapping
    $t\mapsto\EL(\g(\cdot,t))=\EL(\Gamma(\cdot,t))$
    implies for the
    energy values $E(t):=\EL(\g(\cdot,t))$
    that
    $E(t)\le E(0)$ for all $t\in [0,T]$, 
    and therefore by means of \eqref{eq:bilip-evolution3}
    and \eqref{eq:bilipschitz-function},   
    $
    \BiLip(\g(\cdot,t))
    {\ge}
    v_{\g_0}\BiLip(\Gamma(\cdot,t))
      {\ge}
    v_{\g_0}\mathfrak{b}(E(t))\ge v_{\g_0}\mathfrak{b}(E(0)),
    $
    which proves \eqref{eq:bilip-evolution}. 
    Setting $\tilde{\mathfrak{s}}(E):=C_{\mathfrak{s}}(p)(E+E^{\beta_{\mathfrak{s}}})^{1/2}$
    such that the constants $C_{\mathfrak{s}}$ and $\beta_{\mathfrak{s}}$
    coincide with the corresponding constants $C$ and $\beta$ in 
    \cite[(0.8) in Theorem 1]{BlattReiter:2015:Menger} for $q=2$,
    we find by the energy decay in time that
    \begin{equation}\label{eq:seminorm-evolution3} 
    \halfnorm[\frac32 p-3,2]{\Gamma'(\cdot,t)}\le 
    \tilde{\mathfrak{s}}(E(t))\le
    \tilde{\mathfrak{s}}(E(0))\quad\Foa t\in [0,T].
    \end{equation}
    Now we use the arc length function $\LL(\g(\cdot,t),\cdot)$
    of $\g(\cdot,t)$ defined in \eqref{eq:arclength-function1} of Appendix
    \ref{section:Arclength}, which for every $t\in [0,T]$
    coincides with the arc length
    function $\LL(\g_0,\cdot)$ of the initial curve $\g_0$
    according to
    the conservation of velocities \eqref{eq:equal-velocities2},
    to rewrite $\g(\cdot,t)$ as
$    \g(\cdot,t)=\g(\cdot,t)\circ \LL(\g(\cdot,t),\cdot)^{-1}
    \circ \LL(\g(\cdot,t),\cdot)
    {=}
    \Gamma(\cdot,t)\circ \LL(\g_0,\cdot)$ for all $
     t\in [0,T].
    $
    This can be combined with \eqref{eq:seminorm-evolution3} and the
    chain rule inequality \eqref{eq:composition2} in 
    \thref{lemma:compositionSobolevSlobodeckijFunctions} in Appendix 
    \ref{section:Sobolev} to obtain with $\LL(\g_0,\cdot)'=\abs{\g_0'}$ and
    $\halfnorm[\frac32 p-3,2]{\,\abs{\g_0'}\,}\le
    \halfnorm[\frac32 p-3,2]{\g_0'}$ the estimate
    \begin{eqnarray*}
    \halfnorm[\frac32 p-3,2]{\g'(\cdot,t)}
    & {\le} & C\tilde{\mathfrak{s}}
    (E(0))   \big( \halfnorm[\frac32 p-3,2]{ \LL(\g_0,\cdot)'}^2
    + v_{\LL(\g_0,\cdot)}^{-1}\norm[C^0(\R/\Z)]{\LL(\g_0,\cdot)'}^{3p-4}
    \big)^{1/2}\\
    & \le  &C\tilde{\mathfrak{s}}(E(0))   \big( \halfnorm[\frac32 p-3,2]{\g_0'}^2
    + v_{\g_0}^{-1}\norm[C^0(\R/\Z,\R^n)]{\g_0'}^{3p-4}\big)^{1/2},\\
    &
    {\le} &
    C\tilde{\mathfrak{s}}(E(0))\halfnorm[\frac32 p-3,2]{\g_0'}
    \big( 1+(C_EC_P)^{3p-4}v_{\g_0}^{-1}\halfnorm[\frac32 p-3,2]{\g_0'}^{3p-6}
    \big)^{1/2}
    \end{eqnarray*}
    by means of the Morrey embedding \eqref{eq:morrey-embedding}
    and the Poincar\'e inequality \eqref{eq:poincare}. This  
    proves \eqref{eq:seminorm-evolution} 
    if we define the non-decreasing function 
    $
    \mathfrak{s}(\cdot):=
    (1+C_EC_P)^{(3p-4)/2}\cdot
    C\tilde{\mathfrak{s}}(\cdot),
    $ 
    where $C=C(n,p)$ is the
    constant from \thref{lemma:compositionSobolevSlobodeckijFunctions}.
    Now let $\HL:=W^{\frac32 p-2,2}(\R/\Z,\R^n)$. 
    To
    prove Lipschitz continuity of $t \mapsto \partial\g(\cdot,t)/\partial t$
    note first that the radius
$R_{M,\SIGMA}(\gamma(\cdot,t))$ as defined in \eqref{eq:radius-lip-proj}
depends non-decreasingly on $\BiLip(\gamma(\cdot,t))$ and the minimal
velocity $v_{\gamma(\cdot,t)}\ge\BiLip(\gamma(\cdot,t))$, and
non-increasingly on the seminorm $\halfnorm[\frac 32 p-3,2]{\gamma'(\cdot,t)}$.
Combining this with the already established a priori estimates 
\eqref{eq:bilip-evolution} and \eqref{eq:seminorm-evolution} we find
a uniform radius $R_0$ depending only on the initial curve $\gamma_0$ such
that 
    $R_{M,\SIGMA}(\gamma(\cdot,t)) \ge R_0 >0$ for all $t \in [0,T^*]$.
    Similarly, we find a uniform energy bound 
    \begin{equation}
        \label{eq:uniform-tildeC1} 
	\norm[\HL]{\nabla_\SIGMA \EL(\eta)} \le \norm[\HL]{\nabla \EL(\eta)} 
	\le \overline C_1 
    \end{equation}
    for all $\eta\in B_{R_0}(\gamma(\cdot,t))$ by
   \thref{corollary:LipschitzContinuityRHSProjection}, in particular
     \eqref{eq:nabla-bound1}, in combination with the a priori estimates
     \eqref{eq:bilip-evolution} and \eqref{eq:seminorm-evolution}.
    The constant $\overline{C_1}$ depends non-increasingly on $v_{\g_0}$ and 
    non-decreasingly on both $\EL(\g_0)$ and 
    $\halfnorm[\frac 3 2 p -3,2]{\g_0'}$.
    Using \eqref{eq:bilip-evolution} and \eqref{eq:seminorm-evolution} a 
    third time, again in connection with \thref{corollary:LipschitzContinuityRHSProjection}, we obtain a uniform Lipschitz constant $\overline{\LIP}_{\nabla_\SIGMA \EL}$ for $\nabla_{\SIGMA} \EL$ restricted to the ball
    $B_{R_0}(\gamma(\cdot,t))$ for any time $t\in [0,T^*]$.
    This uniform Lipschitz constant has the same dependencies and 
    monotonicities as $\overline C_1$. 
    For any $0\le \sigma_1<\sigma_2\le T^*$ one finds
    \begin{equation}
        \label{eq:lipschitz-gamma}
        \textstyle \norm[\HL]{\gamma(\cdot,\sigma_2)-\gamma(\cdot,\sigma_1)}
	\le
        \int_{\sigma_1}^{\sigma_2} 
	\norm[\HL]{\frac {\partial} {\partial t} \g(\cdot, t)} \, \dd t
        = \int_{\sigma_1}^{\sigma_2} 
	\norm[\HL]{\nabla_{\SIGMA} \EL(\g(\cdot, t))} \, \dd t
        \le \overline C_1 \abs{\sigma_2-\sigma_1}.
    \end{equation}
    Now for any $0\le\sigma <\tau\le T^*$
    choose a partition
 $\sigma=t_0<\ldots<t_N=\tau$ such that 
 $\norm[\HL]{\g(\cdot,t_k)-\g(\cdot,t_{k-1})}<R_0$ for all $k=1,\ldots,N$,
to estimate by means of \eqref{eq:lipschitz-gamma}
    \begin{equation}
        \label{eq:lipschitz-est-gamma_t}
        \begin{aligned}
           &
	\textstyle \norm[\HL]{\frac {\partial} {\partial t} \g(\cdot, \tau) - \frac {\partial} {\partial t} \g(\cdot, \sigma)}
            \le \sum_{k=1}^N  \norm[\HL]{\frac {\partial} {\partial t} \g(\cdot, t_k) - \frac {\partial} {\partial t} \g(\cdot, t_{k-1})}\\
            = &\textstyle\sum_{k=1}^N  \norm[\HL]{\nabla_{\SIGMA}\EL(\g(\cdot, t_k)) - \nabla_{\SIGMA}\EL(\g(\cdot, t_{k-1}))}
            \le \overline{\LIP}_{\nabla_\SIGMA \EL} \sum_{k=1}^N \norm[\HL]{\g(\cdot,t_k)-\g(\cdot,t_{k-1})}\\
            \le &\textstyle\overline{\LIP}_{\nabla_\SIGMA \EL} \overline C_1 \sum_{k=1}^N \abs{t_k-t_{k-1}}
            = \overline{\LIP}_{\nabla_\SIGMA \EL} \overline C_1 \abs{\tau-\sigma}
            =:\LIP_{\partial_t\g}|\tau-\sigma|.
        \end{aligned}
    \end{equation}
    Notice that the newly defined Lipschitz constant $\LIP_{\partial_t\g}=\LIP_{\partial_t\g}
    (E(0),\smash{\halfnorm[\frac32 p-3,2]{\g_0'}},v_{\g_0})$ is non-decreasing in the first two
    arguments and non-increasing in its third argument  as claimed, since 
    $\smash{\overline C_1}$ and $\smash{\overline{\LIP}_{\nabla_\SIGMA \EL}}$ have these monotonicity
    properties.
    {}
\end{proof}

\begin{corollary}
    \label{cor:minimalExtensionTime}
    Let $\g_0\in W_\textnormal{ir}^{\frac32 p-2,2}(\R/\Z,\R^n)$ for 
    $p \in (\frac73, \frac83 )$ and suppose that the mapping
   $
    t\mapsto\g(\cdot,t)\in C^1([0,T^*],
    W^{\frac32 p-2,2}(\R/\Z,\R^n)) 
   $
   is a  
    solution of 
    \eqref{introtheorem:longTimeExistenceViaProjections:eq:ODE} on the time
    interval $[0,T^*]$ for some $T^*>0$.
    Then there is a constant $T_0>0$ depending only on $\gamma_0$ such
    that for any $t_0 \in [0,T^*]$, a unique solution 
    $\eta$ of \eqref{introtheorem:longTimeExistenceViaProjections:eq:ODE} with initial point $\gamma(t_0)$ exists and its minimal existence time $T_\eta$ from \thref{cor:short-time-proj-flow} is uniformly bounded below by $T_0>0$.
\end{corollary}
\begin{proof}
    Set $\HL:=W^{\frac32 p-2,2}(\R/\Z,\R^n)$.
    According to \eqref{eq:bilip-evolution} the evolving curves
    $\g(\cdot,t)$ are regular and embedded for all 
    $t\in [0,T^*]$ so that
    each $\g(\cdot,t)$ can serve as a new initial curve
    for the evolution 
    \eqref{introtheorem:longTimeExistenceViaProjections:eq:ODE}.
    Let $\eta\in C^1([0,T_\eta],\HL)$ be the unique
    solution of 
    \eqref{introtheorem:longTimeExistenceViaProjections:eq:ODE} with initial
    curve  $\g(\cdot,t_0)$, which exists by virtue of 
    the short time existence result, \thref{cor:short-time-proj-flow},
    up to time 
    \begin{equation}\label{eq:eta-existence-time}
    T_\eta
    \ge T_\textnormal{min}(\g(\cdot,t_0))=
    \frac{R_{M,\SIGMA}(\g(\cdot,t_0))}{
        C_1(n,p,\frac12\BiLip(\g(\cdot,t_0)),
	\halfnorm[\frac32 p-3,2]{\g'(\cdot,t_0)}
        +1)};
    \end{equation}
    see \eqref{eq:min-T-proj}. To obtain a uniform lower bound $T_0$ 
    independent of $t_0$ for the
    right-hand side of \eqref{eq:eta-existence-time} notice that the
    denominator $C_1$ is non-increasing in its third argument, and non-decreasing
    in its last argument; see \thref{theorem:secondVariationMp2}. A thorough
    inspection of the definition \eqref{eq:radius-lip-proj} of the radius
    $R_{M,\SIGMA}$ yields a non-decreasing dependence on $\BiLip(\g(\cdot,t_0))$
    and on the minimal velocity $v_{\g(\cdot,t_0)}$ of $\g(\cdot,t_0)$, 
    as well as
    a non-increasing dependence on $\halfnorm[\frac32 p-3,2]{\g'(\cdot,t_0)}$.
    We can neglect the dependence on the minimal velocity, since 
    the conservation of velocities \eqref{eq:equal-velocities2} in
    \thref{prop:bilip-seminorm-evolution} implies
    $v_{\g(\cdot,t)}=v_{\g_0}$ for all $t\in [0,T_\g]$.
    To summarize,  the right-hand side of \eqref{eq:eta-existence-time} may be written as a
    strictly positive  function $\Upsilon=\Upsilon(\BiLip(\g(\cdot,t_0)),
    \halfnorm[\frac32 p-3,2]{\g'(\cdot,t_0)})$ that depends non-decreasingly
    on $\BiLip(\g(\cdot,t_0))$, and non-increasingly on 
    $\halfnorm[\frac32 p-3,2]{
        \g'(\cdot,t_0)}$. 
    We use \eqref{eq:bilip-evolution} in \thref{prop:bilip-seminorm-evolution}
    to estimate the bilipschitz constant from below as
    $\BiLip(\g(\cdot,t_0))\ge v_{\g_0}\mathfrak{b}(E(0))$,
    where we denoted the energy by $E(t):=\Mpq[p,2](\g(\cdot,t))$ for 
    $t\in [0,T^*]$
    as in the previous proof. Moreover, \eqref{eq:seminorm-evolution} in
    \thref{prop:bilip-seminorm-evolution} implies that
    $
    F_0:=\mathfrak{s}(E(0))
    \halfnorm[\frac32 p-3,2]{\g_0'}
    \big(1+ v_{\g_0}^{-1}
    \halfnorm[\frac32 p-3,2]{\g_0'}^{3p-6}
    \big)^{1/2}
    $
    bounds the seminorm $\halfnorm[\frac32 p-3,2]{\g'(\cdot,t_0)}$ from above.
    Consequently, by \eqref{eq:eta-existence-time}
$    T_\eta
    \ge
    \Upsilon(\BiLip(\g(\cdot,t_0)),\halfnorm[\frac32 p-3,2]{\g(\cdot,t_0)})
    \ge \Upsilon(v_{\g_0}\mathfrak{b}(E(0)),F_0)=:T_0 >0.$
\end{proof}

{\it Proof of \thref{introtheorem:longTimeExistenceViaProjections}.}\,
By \thref{corollary:LipschitzContinuityRHSProjection}, 
$-\nabla_\SIGMA E$ is locally Lipschitz continuous, so standard ODE 
theory (cf. for example 
\cite[Part~II, Theorem~1.8.3 ]{Cartan:1971:Differentialcalculus}) 
yields existence of a unique solution
$\g \in C^1(J,W^{\frac 3 2 p-2,2}(\R/\Z,\R^n))$ such that
$\gamma(\cdot,t)$ is embedded and regular for all $t\in J$
for some largest
interval $J$ containing $0$.
Let $T_{\max} := \sup J$ which is bounded below by 
$T_{\min}(\gamma_0)>0$, see \eqref{eq:min-T-proj}.
Assume that $T_{\max} < \infty$, then we can use 
\thref{cor:minimalExtensionTime} to find a solution $\eta$ of 
\eqref{introtheorem:longTimeExistenceViaProjections:eq:ODE} with initial 
value $\g(\cdot, T_{\max} - \frac 1 2 T_0)$  
which exists at least 
up to time $T_0$.
By the global uniqueness theorem for ODEs, see e.g. 
\cite[Part~II, Theorem~1.8.2]{Cartan:1971:Differentialcalculus}, 
we can extend $\gamma$ with $t \mapsto \eta(\cdot,t-(T_{\max}-
\frac 1 2 T_0))$ for $t\in (T_\textnormal{max}-\tfrac12 T_0,
T_\textnormal{max}+\tfrac12 T_0]$ to a solution
on the larger time interval $[0,T_{\max} + \frac 1 2 T_0]$ which is a contradiction to the maximality of $J$ and $T_{\max}$.

The energy decay in time as well as
the conservation of velocities 
\eqref{eq:velocity-preserving} was proven in 
\thref{prop:bilip-seminorm-evolution}; see \eqref{eq:equal-velocities2}.
That the barycenter is preserved along the flow can be proven exactly as
in the proof of the short time existence result for the gradient flow
\eqref{eq:gradflow} in  \thref{theorem:shortTimeExistence} using
the second identity  of \eqref{eq:integral-mean} instead of the first
in \thref{lem:gradients-mean} for $S:=\SIGMA$
and $\HL_2:=W^{\frac32 p-3,2}(\R/\Z)$.
The necessary uniform bound
on $\nabla_\SIGMA\Mpq[p,2](\g(\cdot,t))$ independent of $t$ to interchange
integration and differentiation with respect to $t$ was established
in \eqref{eq:uniform-tildeC1} of 
\thref{prop:bilip-seminorm-evolution} for an arbitrary solution of
the projected flow \eqref{introtheorem:longTimeExistenceViaProjections:eq:ODE}.
Now we use the conservation of the
barycenter and of the
velocity to estimate for arbitrary $x\in\R/\Z$ and $t\in (0,\infty)$  the absolute value as $
|\g(x,t)| 
 \le\textstyle |{\g(x,t)-\int_{\R/\Z}\g(u,t)\,\dd u}| + |{\int_{\R/\Z}\g_0(u)\,\dd u}| \le\norm[C^0(\R/\Z,\R^n)]{\g'(\cdot,t)} + |{\int_{\R/\Z}\g_0(u)\,\dd u}|
=\textstyle
\norm[C^0(\R/\Z,\R^n)]{\g_0'} + |{\int_{\R/\Z}\g_0(u)\,\dd u}|,$
so that we have a uniform $L^\infty$-bound on the solution
$\g(\cdot,t)$ independent of $t$.
The  seminorms
$\halfnorm[\frac32 p-3,2]{\g'(\cdot,t)}$ are uniformly 
bounded by means of the a priori estimate \eqref{eq:seminorm-evolution}
of \thref{prop:bilip-seminorm-evolution}.  The Poincar\'e inequality
\eqref{eq:poincare}
implies that the  fractional Sobolev norms $\norm[W^{\frac32 p-3,2}(\R/\Z,\R^n)]{\g'(\cdot,t)}$ are uniformly bounded, which together with the $L^\infty$-bound
on $\g(\cdot,t)$
leads to a uniform bound on $\norm[W^{\frac32 p-2,2}(\R/\Z,\R^n)]{\g(\cdot,t)}$
independent of $t$ .
\hfill $\Box$

\medskip

{\it Proof of \thref{introtheorem:regularity}.}\,
The Lipschitz continuity of $t\mapsto\g(\cdot,t)$ and
of $t\mapsto\partial \g(\cdot,t)/\partial t$
was established in \eqref{eq:lipschitz-gamma} and 
\eqref{eq:lipschitz-est-gamma_t} of
\thref{prop:bilip-seminorm-evolution} for an arbitrary
solution of the projected flow 
\eqref{introtheorem:longTimeExistenceViaProjections:eq:ODE}. 
In particular the Lipschitz
constants in \eqref{eq:lipschitz-gamma} and
\eqref{eq:lipschitz-est-gamma_t}
depend non-increasingly on their respective third argument
$v_{\g_0}$. 
\hfill $\Box$

\medskip

{\it Proof of \thref{cor:subconvergence}.}\,
Abbreviate $\HL:=W^{\frac32 p-2,2}(\R/\Z,\R^n)$, $\mathcal{E}:=\Mpq[p,2]$, and  $E(t):=\mathcal{E}(\g(\cdot,t))$ for $t\in [0,\infty)$, where  $\g(\cdot,t)$
is the solution
of \eqref{introtheorem:longTimeExistenceViaProjections:eq:ODE}.
For the proof of \eqref{eq:limit-proj-gradient} we take a 
monotonically increasing sequence $(T_i)_i\subset (0,\infty)$ with
$T_i\to\infty$ as $i\to\infty$ and notice that the non-negative
sequence of
energy values $E(T_i)$ 
is non-increasing according to 
\thref{introtheorem:longTimeExistenceViaProjections}
and hence convergent with limit $E_\infty:=
\lim_{i\to\infty}E(T_i)\in [0,E(0)]$. This limit does not depend on the
choice of the $T_i$ by the subsequence principle.
We can take the limit $i\to\infty$
in the identity $
E(T_i)-E(0)
= \textstyle\int_0^{T_i}\langle\nabla
\mathcal{E}(\g(\cdot,t)),\partial_t\g(\cdot,t)\rangle_{\HL}\,\dd t
{=}
-\int_0^{T_i}\langle\nabla
\mathcal{E}(\g(\cdot,t)),\nabla_\SIGMA\mathcal{E}(\g(\cdot,t))
\rangle_{\HL}\,\dd t
= -\int_0^{T_i}\norm[\HL]{
    \nabla_\SIGMA\mathcal{E}(\g(\cdot,t))}^2\,\dd t   $
to obtain 
$
\int_0^\infty\norm[\HL]{
    \nabla_\SIGMA\mathcal{E}(\g(\cdot,t))}^2\,\dd t=E(0)-E_\infty<\infty,
$
which implies \eqref{eq:limit-proj-gradient}.
Along any sequence of times $t_k\to\infty$ the fractional Sobolev norms
$\norm[\HL]{\g(\cdot,t_k)}$ are uniformly 
bounded by \thref{introtheorem:longTimeExistenceViaProjections}.
This
guarantees the existence
of a subsequence $(t_{k_l})_l\subset (t_k)_k$ such that the
$\g(\cdot,t_{k_l})$
converge weakly in $\HL$ and strongly in $C^1$ (by Morrey's embedding,
\thref{thm:morrey}) to a limiting curve
$\g^*\in \HL$ as $l\to\infty$.
By Fatou's Lemma $\EL$ is sequentially
lower semicontinuous with respect to $C^1$-convergence, so that 
$\EL(\g^*)\le E_\infty$.
The a priori estimate \eqref{eq:bilip-evolution} 
of \thref{prop:bilip-seminorm-evolution}
is conserved in the $C^1$-limit, which implies that $\g^*$ is embedded and
hence of the same knot class as $\g(\cdot,t_{k_l})$ for $l\gg 1$, which
proves $[\g^*]=[\g_0]$.
\hfill $\Box$

We conclude with a simple
stability estimate for  solutions  of 
\eqref{introtheorem:longTimeExistenceViaProjections:eq:ODE}.
\begin{corollary}\label{cor:stability}
    Let $\gamma(\cdot,t)$ and $\eta(\cdot,t)$ be solutions to
    \eqref{introtheorem:longTimeExistenceViaProjections:eq:ODE}
    with initial
    curves $\gamma_0,\eta_0 \in \HL:=W_{\ir}^{\frac 3 2 p -2,2}(\R/\Z,\R^n)$.
    Then there are constants $K_{\g_0}$ and $K_{\eta_0}$, depending 
    only on $\g_0$ and $\eta_0$, respectively, such that
    $
    \norm[\HL]{\gamma(\cdot,t)-\eta(\cdot,t)}
    \leq \norm[\HL]{\gamma_0 -\eta_0} +
    (K_{\g_0}+K_{\eta_0})t $ for all $ t\in [0,\infty).$
   \end{corollary}
\begin{proof}
Set $\EL:=\Mpq[p,2]$ and use
    the Lipschitz constant $K_{\g_0}:=\overline C_1=
    \overline C_1(\EL(\g_0),
    \halfnorm[\frac32 p-3,2]{\g_0'},v_{\g_0})$ for $t\mapsto\g(\cdot,t)$
    introduced  in \eqref{eq:uniform-tildeC1} and \eqref{eq:lipschitz-gamma}, and 
    analogously for $\eta(\cdot,t)$ setting $K_{\eta_0}:=C_0(\EL(\eta_0),
    \halfnorm[\frac32 p-3,2]{\eta_0'},v_{\eta_0})$, to estimate $
    \norm[\HL]{\gamma(\cdot,t)-\eta(\cdot,t)}
    \leq \norm[\HL]{\gamma(\cdot,t)-\gamma_0} + \norm[\HL]{\gamma_0-\eta_0} +
    \norm[\HL]{\eta_0 - \eta(\cdot,t)}
    {\leq} \norm[\HL]{\gamma_0 -\eta_0} +
    (K_{\g_0}+K_{\eta_0})t. $
\end{proof}

\section{Numerical Experiments}
\label{sec:Numerics}

\subsection{Discretization}

For a numerical treatment of the gradient flow, not only the energy $\Energy \ceq \intM^{(p,2)}$ and its derivative $D\Energy$ have to be discretized, but also the Riesz isomorphism $\RieszOp$ and the constraints $\ConstraintMap \colon \mathcal{O} \to  \mathcal{H}_2$, where $\HL_2 := W^{\frac 3 2 p -3,2}(\R/\Z)$ and $\mathcal{O}:=W^{\frac 3 2 p -2,2}_{\ir}(\R/\Z,\R^n)$ is an open subset of $\HL_1 := W^{\frac 3 2 p -2,2}(\R/\Z,\R^n)$.
In a nutshell, we do this by discretizing a curve $\Curve \colon \Circle \to \AmbSpace$ by a closed polygonal line $P$ and by replacing all occurring integrals by appropriate sums.
To this end, we fix a partition $\Triangulation$ of $\Circle$ into a set of disjoint intervals or \emph{edges} $\Edges(\Triangulation)$.
Denote the set of interval end points (\emph{vertices}) by $\Vertices(\Triangulation) \subset \Circle$.
By a polygonal line we mean a mapping $\Polygon \colon \Vertices(\Triangulation) \to \AmbSpace$ and identify it with periodic, piecewise-linear interpolation. 
Hence, our discrete configuration space is
$
	\HilbertSpace_{\Triangulation,1} \ceq \{\varPhi \colon \Vertices(\Triangulation) \to \AmbSpace \}
$.
This is a vector space of dimension $\AmbDim \cdot \VertexCount$, where $\VertexCount$ is the number of vertices (which is also equal to the number of edges).
In the following, we will employ the natural basis $\varPhi_{\Vertex,i}$, $\Vertex \in \Vertices(\Triangulation)$, $i \in \{1,\dotsc,\AmbDim\}$ where $\ninnerprod{\varPhi_{\Vertex,i}(\Vertex'), e_j}_\AmbSpace = \delta_{\Vertex,\Vertex'} \, \delta_{i,j}$.

Denoting the left and right endpoints of an edge $\Edge \in \Edges(\Triangulation)$ by $\down{\Edge}$ and $\up{\Edge}$, respectively, we may write down the \emph{edge length} $\EdgeLengths_\Polygon(\Edge)$ and the \emph{unit edge vector} $\tau_\Polygon(I)$ of edge $\Edge$ with respect to $\Polygon$ as follows:
\begin{align*}
    \textstyle
	\EdgeLengths_\Polygon(\Edge) \ceq \nabs{ \Polygon(\up{\Edge}) - \Polygon(\down{\Edge})}
	\quad
	\text{and}
	\quad
	\tau_\Polygon(I) \ceq \frac{P(\up{\Edge})-P(\down{\Edge})}{\nabs{P(\up{\Edge})-P(\down{\Edge})}}
	.
\end{align*}

\subsection{Constraints}
With
$\HilbertSpace_{\Triangulation,2} \ceq \{\varPhi \colon \Edges(\Triangulation) \to \R \} \cong \R^{N}$,
we may discretize the constraint $\ConstraintMap(\Curve) = \log(\nabs{\Curve'})$ by
\begin{align*}
    \textstyle
	\ConstraintMap_\Triangulation \colon \HilbertSpace_{\Triangulation,1} \to \HilbertSpace_{\Triangulation,2},
	\qquad
	\ConstraintMap_\Triangulation(\Polygon) 
	\ceq 
	\big( \log \big(  \frac{\EdgeLengths_\Polygon(\Edge)}{\nabs{\Edge}} \big) \big)_{\Edge \in \Edges(\Triangulation)}.
\end{align*}
In terms of the basis $\varPhi_{\Vertex,i}$ on $\HilbertSpace_{\Triangulation,1}$ and the standard basis on $\HilbertSpace_{\Triangulation,2}$,
the derivative $D \ConstraintMap_\Triangulation(\Polygon)$ can be represented by a sparse matrix of size $\VertexCount \times (\AmbDim \,  \VertexCount)$ with only 
$2 \, \AmbDim \,\VertexCount$ nonzero entries. The assembly is straight-forward once
the derivatives of $\log (  \EdgeLengths_\Polygon(\Edge)/ \nabs{\Edge} )$ with respect to 
the coordinates of $\Polygon(\down{\Edge})$ and $\Polygon(\up{\Edge})$ have been computed.
These derivatives are given by
\begin{align*}
    \textstyle
	\frac{\partial }{\partial P(\down{\Edge})}  
	\log \big(  \frac{\EdgeLengths_\Polygon(\Edge)}{\nabs{\Edge}} \big)
	=
	- \frac{\tau_\Polygon(I)\transp}{\EdgeLengths_\Polygon(\Edge)}
	\quad
	\text{and}
	\quad
	\frac{\partial }{\partial P(\up{\Edge})}  
	\log \big(  \frac{\EdgeLengths_\Polygon(\Edge)}{\nabs{\Edge}} \big)
	=
	+\frac{\tau_\Polygon(I)\transp}{\EdgeLengths_\Polygon(\Edge)}	
	.
\end{align*}

\subsection{Energy}
The energy of a smooth curve $\Curve \colon \Circle \to \AmbSpace$ can be decomposed as follows:
\begin{align*}
    \textstyle
	\Energy(\Curve)
	=
	\sum_{\Edge_1, \Edge_2,\Edge_3 \in \Edges(\Triangulation)}
	\int_{\Edge_1}
	\int_{\Edge_2}
	\int_{\Edge_3}
	W_\Curve(u_1,u_2,u_3)
	\, \dd u_1
	\, \dd u_2
	\, \dd u_3,
\end{align*}
where
\begin{align*}
    \textstyle
	W_\Curve(u_1,u_2,u_3)	
	\ceq
		\frac {
			\nabs{\Curve'(u_1)} \,\nabs{\Curve'(u_2)} \, \nabs{\Curve'(u_3)}
		}{
			R^{(p,2)}(\Curve(u_1),\Curve(u_2),\Curve(u_3))
		} 
	.
\end{align*}
By employing the midpoint and by skipping problematic summands, i.e., those for which $\Edge_1$, $\Edge_2$, and $\Edge_3$ are not pairwise distinct, we obtain the \emph{discrete integral Menger energy}
\begin{align*}
    \textstyle
	\Energy_\Triangulation(\Polygon)
	\ceq
	\sum
	W_\Polygon(\Edge_1,\Edge_2,\Edge_3),
\end{align*}
where the sum is taken over mutually distinct $\Edge_1,\Edge_2,\Edge_3 \in \Edges(\Triangulation)$, and
\begin{align*}
    \textstyle
	W_\Polygon(\Edge_1,\Edge_2,\Edge_3)
	\ceq
	\frac {
		\EdgeLengths_\Polygon(\Edge_1) \, \EdgeLengths_\Polygon(\Edge_2) \, \EdgeLengths_\Polygon(\Edge_3)
	}{
		R^{(p,2)}( \Midpoint_\Polygon(\Edge_1), \Midpoint_\Polygon(\Edge_2), \Midpoint_\Polygon(\Edge_3))
	}
		\quad
	\text{with}
	\quad
	\Midpoint_\Polygon(\Edge) \ceq \frac{1}{2} \big( \Polygon(\up{\Edge}) + \Polygon(\down{\Edge}) \big)
	.
\end{align*}
The local contribution $W_\Polygon(\Edge_1,\Edge_2,\Edge_3)$ is symmetric in $\Edge_1$, $\Edge_2$, $\Edge_3$,
and exploiting this, one can reduce the number of summands by a factor of $\frac{1}{6}$. Nonetheless, the computational complexity of evaluating $\Energy_\Triangulation(\Polygon)$ is $\LandO(\VertexCount^3)$.

The term $W_\Polygon(\Edge_1,\Edge_2,\Edge_3)$ is an analytical expression in the six points
$\Polygon(\down{\Edge_1})$, $\Polygon(\up{\Edge_1})$,
$\Polygon(\down{\Edge_2})$, $\Polygon(\up{\Edge_2})$,
$\Polygon(\down{\Edge_3})$, and $\Polygon(\up{\Edge_3})$.
Hence the derivative $W_\Polygon(\Edge_1,\Edge_2,\Edge_3)$ with respect to these points can be computed  symbolically (we strongly suggest to use a CAS for that) and compiled into a library for runtime performance.
The derivative $D\Energy_\Triangulation(\Polygon) \in \HilbertSpace_{\Triangulation,2}^*$ can be represented by a vector of size
$n \, \VertexCount$; it can then be assembled from the output of this library and from the vertex indices of the interval end points $\down{\Edge_1}$, $\up{\Edge_1}$,
$\down{\Edge_2}$, $\up{\Edge_2}$, $\down{\Edge_3}$, and $\up{\Edge_3}$.

\subsection{Metric}

We define the Gagliardo product of $\varphi$, $\psi \in \HilbertSpace_1$ by
\begin{align*}
    \textstyle
	\ninnerprod{ \RieszOp \,  \varphi,\psi}
	\ceq
	\iint_{(\Circle)^2}
		\frac{
			\ninnerprod{ \varphi'(u_1) -  \varphi'(u_2), \psi'(u_1) -  \psi'(u_2)}
		}{
			\nabs{u_1-u_2}^{2 s - 1}			
		}
	\, \dd u_1 \, \dd u_2
	.
\end{align*}
By \thref{proposition:PoincareInequality}, this is indeed an inner product when restricted to the subspace
\begin{align*}
    \textstyle
	\dot \HilbertSpace_1 
	\ceq 
	\{
		\varphi \in \HilbertSpace_1
		\mid
		\textstyle \int_\Circle \varphi(u) \, \dd u =0
	\}.
\end{align*}
Analogously to the energy, we may discretize this metric by first decomposing the integral with respect to the partition $\Triangulation$:
\begin{align*}
    \textstyle
	\ninnerprod{ \RieszOp \,  \varphi,\psi}
	=
	\sum_{\Edge_1, \Edge_2 \in \Edges(\Triangulation)}
	\int_{\Edge_1}
	\int_{\Edge_2}	
		\frac{
			\ninnerprod{ \varphi'(u_1) -  \varphi'(u_2), \psi'(u_1) -  \psi'(u_2)}
		}{
			\nabs{u_1-u_2}^{2 s - 1}			
		}
	\, \dd u_1 \, \dd u_2
	.
\end{align*}
We interpret $\varPhi$, $\varPsi \in \HilbertSpace_{\Triangulation,1}$
as a piecewise linear functions.
Thus $\varPhi'$  and $\varPsi'$ are constant on (the interior of) each edge.
This is why we may put
$\varPhi'(\Edge) \ceq  \tfrac{\varPhi(\up{\Edge}) - \varPhi(\down{\Edge})}{\nabs{\Edge}}$,
$\varPsi'(\Edge) \ceq  \tfrac{\varPsi(\up{\Edge}) - \varPsi(\down{\Edge})}{\nabs{\Edge}}$
and define the \emph{discrete Riesz isomorphism} $\RieszOp \colon \HilbertSpace_{\Triangulation,1} \to \HilbertSpace_{\Triangulation,1}^*$ via the midpoint rule by
\begin{align*}
    \textstyle
	\ninnerprod{\RieszOp_\Triangulation \, \varPhi , \varPsi}
	\ceq
	\sum_{\substack{
		\Edge_1,\Edge_2 \in \Edges(\Triangulation)
		\\
		\Edge_1 \neq \Edge_2
	}}	
		\frac{
			\ninnerprod{ 
				\varPhi'(\Edge_1)
				-
				\varPhi'(\Edge_2)						
				,
				\varPsi'(\Edge_1)
				-
				\varPsi'(\Edge_2)			
			}		
		}{
			\nabs{\Midpoint(\Edge_1)-\Midpoint(\Edge_2)}^{2 s - 1}			
		}
	\, \nabs{\Edge_1} \, \nabs{\Edge_2}
	,
\end{align*}
where $\Midpoint(\Edge)$ denotes the midpoint of the interval $I$.
In the following, we identify $\RieszOp_\Triangulation$ with the
Gram matrix with respect to the basis $\varPhi_{\Vertex,i}$ of the bilinear form $\Metric$ defined by
$
	\Metric( \varPhi , \varPsi) \ceq \ninnerprod{\RieszOp_\Triangulation \, \varPhi , \varPsi}
$.
This is a matrix of size $(\AmbDim \cdot \VertexCount) \times (\AmbDim \cdot \VertexCount)$.
Since different spatial components of $\varPhi$ of $\varPsi$ do not interact, 
one may reorder the degrees of freedom such that
the matrix $\RieszOp_\Triangulation$ is block diagonal with $n$ identical dense blocks of size $\VertexCount \times \VertexCount$ on its diagonal; each of the blocks is a copy of $\RieszOp_\Triangulation$ for $n=1$.

We abbreviate $\mathcal{B}_\Triangulation \ceq D \ConstraintMap_\Triangulation(\Polygon)$
and denote the discrete barycenter constraint by $\BarycenterContraint_\Triangulation \colon \HilbertSpace_{\Triangulation,1} \to \AmbSpace$, $\BarycenterContraint_\Triangulation \,  \varPhi = \frac{1}{2}  \sum_{\Edge \in \Edges(\Triangulation)} (\varPhi(\down{\Edge}) + \varPhi(\up{\Edge})) \, \nabs{I}$.
With the symmetric saddle point matrix
\begin{align}
    \textstyle
	\SaddlePointMatrix_\Triangulation(P)
	\ceq
	\begin{pmatrix}
		\RieszOp_\Triangulation & \mathcal{B}_\Triangulation\transp & \BarycenterContraint_\Triangulation\transp
		\\
		\mathcal{B}_\Triangulation&0 &0\\
		\BarycenterContraint_\Triangulation &0 &0
	\end{pmatrix},
	\label{eq:SaddlePointMatrix}
\end{align}
the discrete projected gradient $\nabla_\SIGMA E_T(\Polygon)$ of $\Energy_\Triangulation$ can be computed by solving the linear system
\begin{align}
    \textstyle
	\SaddlePointMatrix_\Triangulation(P)
	\,	
	\begin{pmatrix}
		\nabla_\SIGMA E_T(\Polygon)
		\\
		\lambda
		\\
		\mu
	\end{pmatrix}
	=
	\begin{pmatrix}
		D \Energy_\Triangulation(\Polygon)
		\\
		0
		\\
		0
	\end{pmatrix},
	\label{eq:DiscreteGradientEquation}
\end{align}
where $\lambda \in  \HilbertSpace_{\Triangulation,2}^*$ and $\mu \in (\AmbSpace)^*$ act as Lagrange 
multipliers.
Note that we included $\BarycenterContraint_\Triangulation$ and $\BarycenterContraint_\Triangulation^\top$ into the saddle point matrix \eqref{eq:SaddlePointMatrix} in order to obtain a symmetric matrix of full rank because otherwise all constant functions would be in the kernel of the saddle point matrix.

Unfortunately, the matrix $\RieszOp_\Triangulation$ is fully populated so that sparse matrix methods do not help in solving \eqref{eq:DiscreteGradientEquation} numerically. 	
In principle, one may enforce that all edges have the same length and then exploit that $\RieszOp_\Triangulation$ is a Toeplitz matrix (e.g., via using the fast Fourier transform).
However, in our experiments, we simply used the dense $LU$-factorization provided by \emph{LAPACK}. Although its costs are of order $((\AmbDim+1) \, \VertexCount + \AmbDim)^3 = \LandO( (\AmbDim+1)^3 \, \VertexCount^3)$, the factorization took up only a relatively small portion of the overall computation time for $N \leq 600$.
This is because the timings of the naive $\LandO( (\AmbDim \, \VertexCount)^3)$-implementations we employed for
computing $\Energy_\Triangulation(\Polygon)$ and $D \Energy_\Triangulation(\Polygon)$ were dominant.

\subsection{Discrete Gradient Flow}

The gradient flow equation for $P(t)$ reads as follows:
\begin{align*}
    \textstyle
	\SaddlePointMatrix_\Triangulation(P(t))
	\,	
	\begin{pmatrix}
		\dot{P}(t)
		\\
		\lambda(t)
		\\
		\mu(t)
	\end{pmatrix}
	=
	\begin{pmatrix}
		- D \Energy_\Triangulation(\Polygon(t))
		\\
		0
		\\
		0
	\end{pmatrix}.
\end{align*}
In our experiments, we simply used the local model
$
	P(t + \tau)
	=
	P(t) 
	+
	\dot{P}(t) \, \tau
	+ \LandO(\tau^2)
$
for the update step, which effectively leads to the explicit Euler method.
Thus with 
\begin{align*}
    \textstyle
	Q_0(\tau) \ceq P(t) + \dot{P}(t) \, \tau, 
\end{align*}
we have $Q_0(\tau) = P(t + \tau) 	+ \LandO(\tau^2)$.
In particular, this guarantees that the constraint violation grows only slowly:
\begin{align*}
    \textstyle
	\ConstraintMap_\Triangulation(Q_0(\tau))
	=
	\ConstraintMap_\Triangulation(P(t)) + \LandO(\tau^2).
\end{align*}

The remaining constraint violation can be reduced by 
applying a few iterations of the following \emph{modified Newton method} involving the saddle point matrix $\SaddlePointMatrix_\Triangulation$ from \eqref{eq:SaddlePointMatrix}:
\begin{align}
    \textstyle
	Q_{k+1}(\tau) = Q_k(\tau) - v_k,
	\quad
	\text{where}
	\quad
	\SaddlePointMatrix_\Triangulation(P(t)) 
	\begin{pmatrix}
		v_k \\ \lambda_k \\ \mu_k
	\end{pmatrix}
	=
	\begin{pmatrix}
		0 \\ \ConstraintMap_\Triangulation( Q_k(\tau))  \\ 0
	\end{pmatrix}	
	.
	\label{eq:RestoringFeasibility}
\end{align}
Then $Q_\infty(\tau) \ceq \lim_{k \to \infty} Q_{k}(\tau)$ may serve as next iterate of the discrete flow.
In contrast to the classical Newton method, the matrix in the linear system is \emph{not} updated; this allows us to reuse the already computed matrix $\SaddlePointMatrix_\Triangulation(P(t))$ and its factorization. The update vectors $v_k$ are always perpendicular to the tangent space of the constraint manifold at $P(t)$ (see \autoref{fig:Helpers} (a)).
The iteration in \eqref{eq:RestoringFeasibility} converges because of the implicit function theorem, see \cite[Theorem 4.B (b)]{Zeidler:1993:NonlinearfunctionalanalysisanditsapplicationsFixedpointtheorems}.\footnote{In Zeidler's notation, set $F=\SIGMA_\Triangulation$, $X = \ker(D\SIGMA_\Triangulation(P(t)))$ and $Y = X^\perp$ as subspaces of $\ker(\BarycenterContraint_\Triangulation)\subseteq \HilbertSpace_{\Triangulation,1}$. Then, for $Z=\HilbertSpace_{\Triangulation,1}$, one can show that $F_y(P(t))=D\SIGMA_\Triangulation(P(t))\Pi_Y$ is bijective by first constructing a right-inverse as in Section \ref{subsection:explicitBoundaryCondition} and then projecting this to $Y$. The equation $v_k = F_y^{-1}(P(t))F(Q_k(\tau))$ is then equivalent to \eqref{eq:RestoringFeasibility}, because $\BarycenterContraint_\Triangulation^\top(\HilbertSpace_{\Triangulation,1})=\ker(\BarycenterContraint_\Triangulation)^\perp$ (see e.g.\ \cite{KunzeHoffman:1971:LinearAlgebra}) and the same holds for $\mathcal B_\Triangulation$.}

\begin{figure}
\begin{center}
        \newcommand{\myincludegraphics}[2]{\begin{tikzpicture}
            \node[inner sep=0pt] (fig) at (0,0) {\input{#1}};
            \node[above right= 0.25ex] at (fig.south west) {\footnotesize #2};    
            \end{tikzpicture}
        }%
\begin{minipage}{0.5\textwidth}
	\def\svgwidth{0.9\textwidth}
	{\footnotesize \myincludegraphics{RestoringFeasibility_pdf.tex}{(a)}}
\end{minipage}%
\begin{minipage}{0.5\textwidth}
	\def\svgwidth{0.9\textwidth}
	{\footnotesize \myincludegraphics{LineSearch_pdf.tex}{(b)}}
\end{minipage}%
\label{fig:Helpers}
\caption{(a) The modified Newton method starts at $Q_0(\tau)$ and updates it along $\ker (D \SIGMA_\Triangulation(P(t)))^\perp$ until it converges to $Q_\infty(\tau)$. The orthogonality can be seen by computing $(u,0,0)\SaddlePointMatrix_\Triangulation(P(t))(v,\lambda,\mu)^\top$ for $u \in \ker(D \SIGMA_\Triangulation(P(t)) \cap \ker(\BarycenterContraint_\Triangulation)$ and keeping in mind that $\RieszOp_\Triangulation$ is the Gram matrix of the discrete Gagliardo product on the admissible space $\ker(\BarycenterContraint_\Triangulation)$.
(b)
Singularities are the defining properties of self-avoiding energies, so they are frequently encountered in the computational treatment of their flows.
Thus, adaptive choice of step size is crucial for the stability. The interval of step sizes that satisfy the Armijo condition is highlighted in green.}
\end{center}
\end{figure}
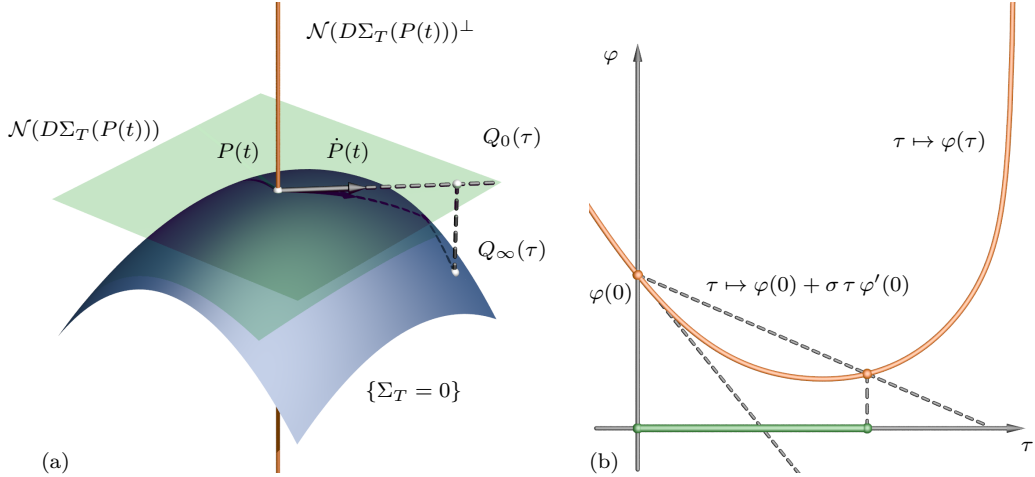

By backtracking, i.e., by systematically decreasing $\tau$ if necessary, we may enforce the following three conditions.
\begin{enumerate}
	\item The point $Q_\infty(\tau)$ is well-defined. In practice, we require that the constraint violation of $Q_k(\tau)$ is below a prescribed threshold after a prescribed number $k$ of modified Newton iterations \eqref{eq:RestoringFeasibility} and we use $Q_k(\tau)$ as approximation to $Q_\infty(\tau)$; otherwise, we shrink $\tau$ and restart \eqref{eq:RestoringFeasibility}.
	\item  For some parameter $\sigma \in (0,1)$ and for the function $\varphi(\tau) \ceq \Energy_\Triangulation(Q_\infty(\tau))$, the \emph{Armijo condition} 
	\begin{align*}
        \textstyle
		\varphi(\tau) \leq  \varphi(0) + \sigma \, \tau \, \varphi'(0) <  \varphi(0) = \Energy_\Triangulation(\Polygon(t)) 
	\end{align*}
	is fulfilled (see \autoref{fig:Helpers}). This guarantees that the step size $\tau$ is adapted to the ``stiffness'' of the ODE.
	Satisfying the Armijo condition is possible because $\varphi'(0)$ is negative (unless $P(t)$ is a critical point) and because we have
	\begin{align*}
        \textstyle
		\varphi(\tau) 
		\leq  
		\varphi(0) + \tau \, \varphi'(0)
		+ \tfrac{1}{2} \, \Lip(\varphi') \, \tau^2
	\end{align*}
	due to local Lipschitz continuity of $\varphi'$.
	For more details on the Armijo condition and on line search methods that guarantee it, we refer the interested reader to \cite[Chapter~3]{MR2244940}.
	\item The curve $Q_\infty(\tau)$ (or its proxy $Q_k(\tau)$) is ambient isotopic to $P(t)$. 
	Because $P(t)$ is an embedded polygonal line, such a step size $\tau >0$ must exist.
\end{enumerate}

The last condition is a bit involved. In principle, there are various strategies to ensure it.
What we do is the following: 
We consider the homotopy $F(u,\lambda) \ceq  (1-\lambda) \, P(t)(u) + \lambda \, Q_\infty(\tau)(u)$ (see \autoref{fig:IsotopyCheck}) and check the following:
\begin{itemize}
	\item Are the edge lengths of all the polygonal lines defined by $F(\cdot,\lambda)$, $\lambda \in [0,1]$ bounded away from $0$?
	\item Are all the turning angles of the polygonal lines $F(\cdot,\lambda)$, $\lambda \in [0,1]$ bounded away from $\pi$?
	\item Is the image surface of $F$ in $\R^3 \times [0, 1]$ free of self-intersections? (See \cite[Section~4.2]{doi:10.1111/1467-8659.t01-1-00587} for computational techniques to answer this question.)
\end{itemize}
If the answer to all these questions is affirmative, then $F$ is a level preserving $C^0$-isotopy.
In general, the existence of a level preserving $C^0$-isotopy $F$ does \emph{not} imply that $F(\cdot,0)$ and $F(\cdot,1)$ are ambient isotopic.
But the special structure of $F$  implies that it is a \emph{locally trivial isotopy} and thus a \emph{locally unknotted isotopy} in the sense of \cite[p.~72]{MR163318}. Thus \cite[Theorem~2]{MR163318} guarantees that $F(\cdot,0) = P(t)$ and $F(\cdot,1) = Q_\infty(\tau)$ are indeed ambient isotopic.

\begin{figure}
\begin{center}
        \newcommand{\myincludegraphics}[2]{\begin{tikzpicture}
            \node[inner sep=0pt] (fig) at (0,0) {\includegraphics[width=0.99\textwidth]{#1}};
            \node[below right= 0.125ex] at (fig.north west) {\footnotesize #2};    
            \end{tikzpicture}
        }%
\begin{minipage}{0.25\textwidth}
\begin{center}
	\myincludegraphics{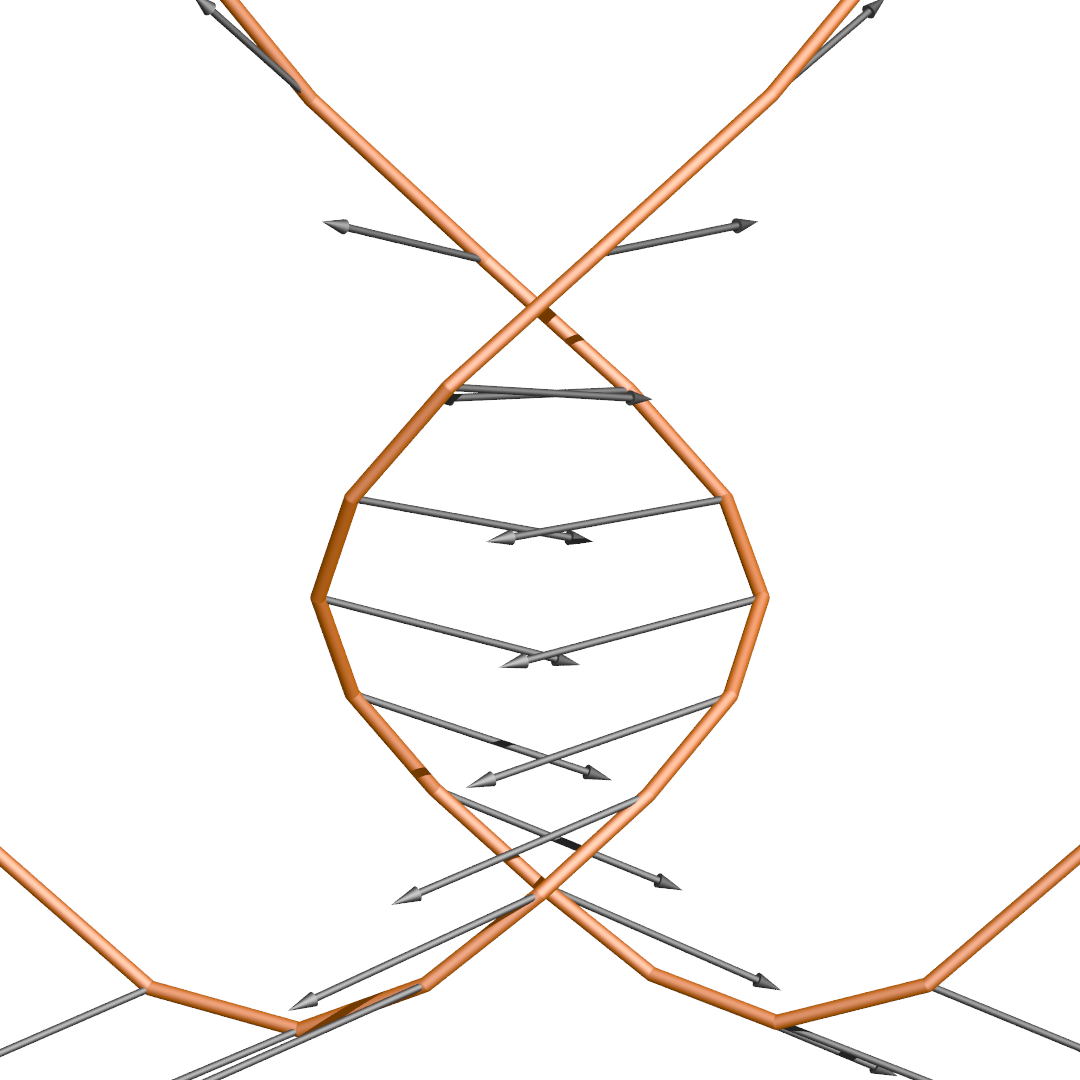}{(a)}
\end{center}
\end{minipage}%
\begin{minipage}{0.25\textwidth}
\begin{center}
	\myincludegraphics{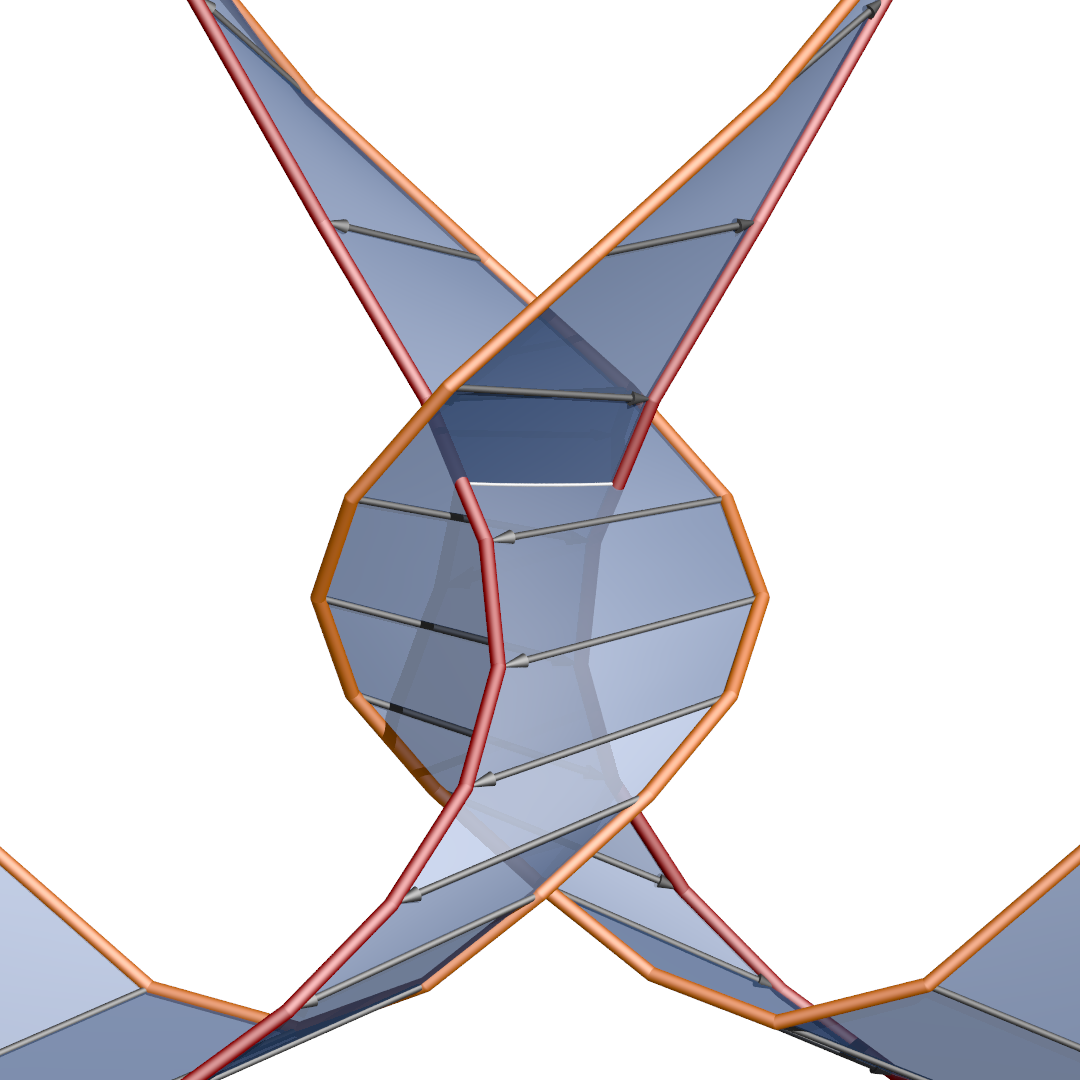}{(b)}
\end{center}
\end{minipage}%
\begin{minipage}{0.25\textwidth}
\begin{center}
	\myincludegraphics{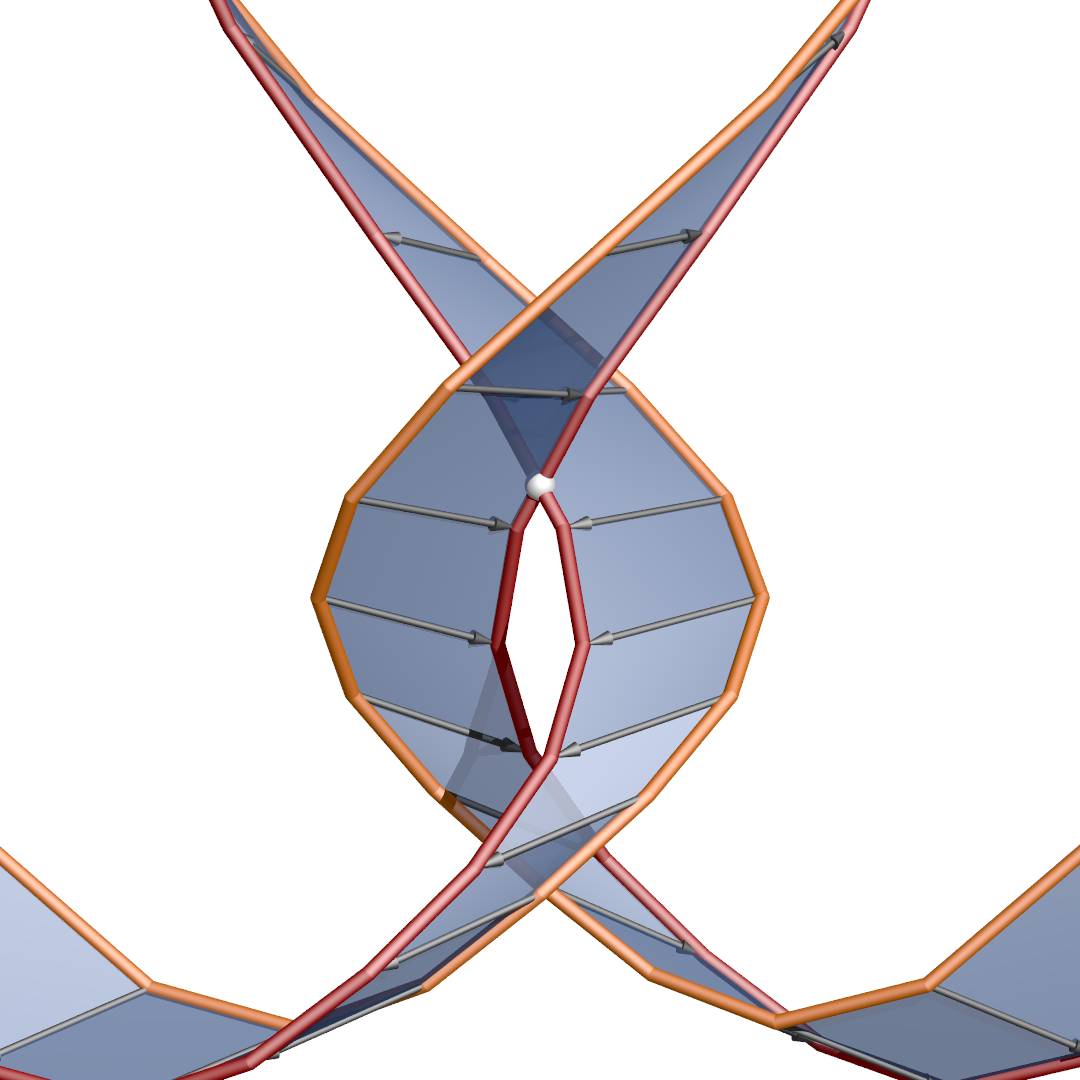}{(c)}
\end{center}
\end{minipage}%
\begin{minipage}{0.25\textwidth}
\begin{center}
	\myincludegraphics{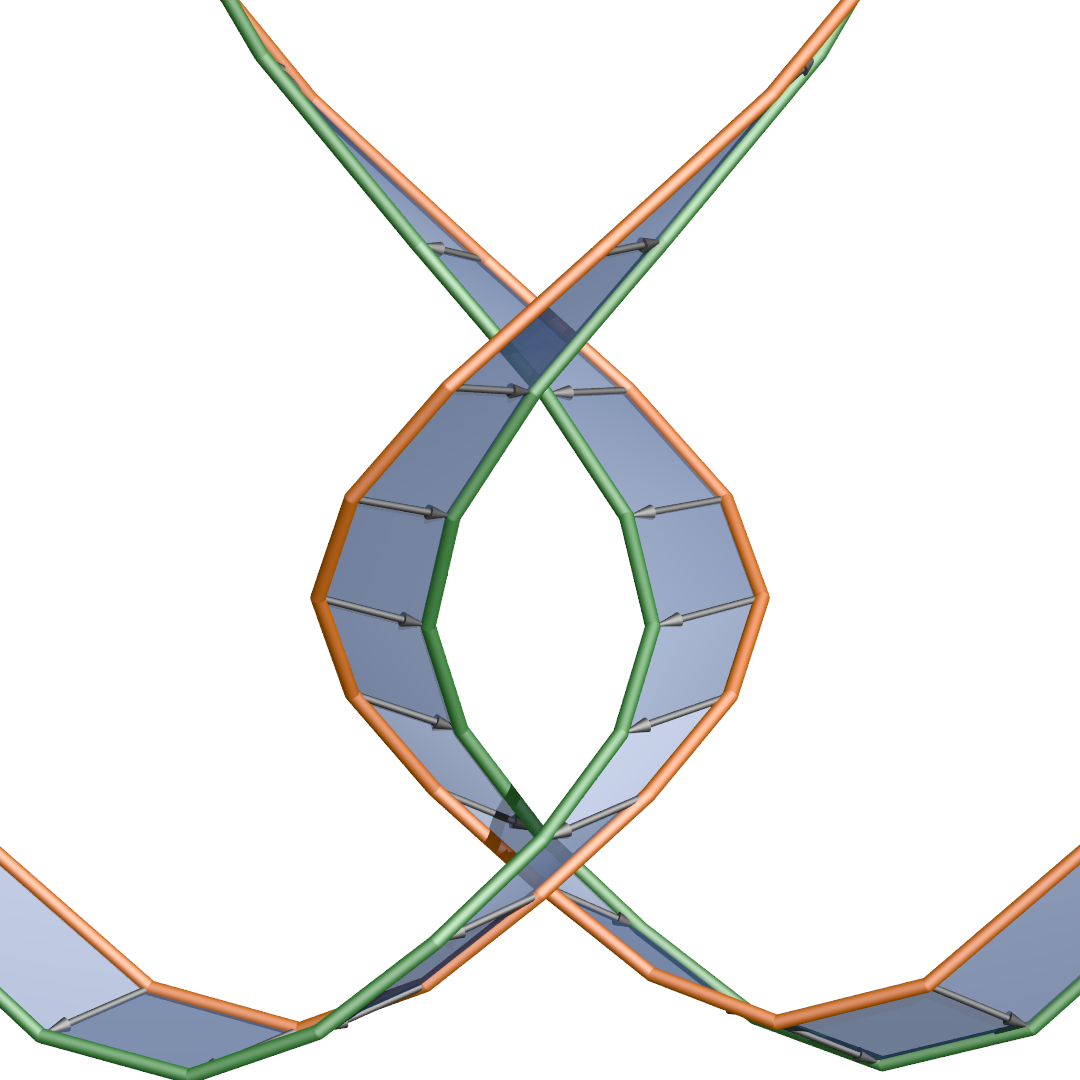}{(d)}
\end{center}
\end{minipage}%
\label{fig:IsotopyCheck}
\caption{(a) 
The current polygonal line $\Polygon(t)$ (orange) and an update direction, represented by the vectors on the vertices.
(b)
For too large a step size $\tau$, the updated polygonal line $Q_\infty(\tau)$ (red) may lie outside the isotopy class of $P(t)$. The projection of the image of the homotopy $F$ to $\R^3$ (in blue) indicates that $F$ might not be a level preserving isotopy.
(c)
The updated polygonal line $Q_\infty(\tau)$ (red) for the smallest $\tau >0$ such that $F$ is not a level preserving isotopy.
(d) 
After further decreasing $\tau$, the updated curve $Q_\infty(\tau)$ (green) is guaranteed to be in the same isotopy class as $P(t)$.}
\end{center}
\end{figure}
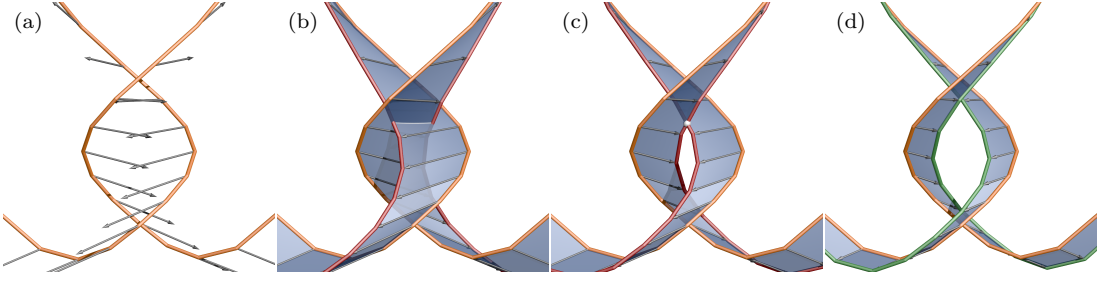

\appendix

\section{Periodic  Sobolev-Slobodecki\texorpdfstring{\v{\i}}{j} spaces}
\label{section:Sobolev}
For fixed $\ell>0$, $s \in (0,1)$ and $\rho \in [1,\infty)$ define the
seminorm
\[
\label{eq:defSeminorm}
\numberthis
\halfnorm[{s,\rho}]{f}
:= \Big( \textstyle\int_{\R/\ell\Z} \int_{-{\ell}/2}^{{\ell}/2} 
\frac {\abs{f(u+w) - f(u)}^\rho} {\abs{w}^{1+s\rho}} \,\dd w \, \dd u
\Big)^{1/\rho}
\]
for $\ell$-periodic  functions $f\colon \R\to\R^n$ that are locally 
$\rho$-integrable on $\R$, i.e., for $f\in L^\rho(\R/\ell\Z,\R^n)$.
\begin{definition}\label{def:sob-slobo}
    For $k\in\N\cup\{0\}$ the function space
    \[
    \textstyle
    W^{k+s,\rho}(\R/\ell\Z,\R^n):=
    \{f\in W^{k,\rho}(\R/\ell\Z,\R^n):\|f\|_{W^{k+s,\rho}}<\infty\},
    \]
    where  $
        \norm[W^{k+s,\rho}(\R/\ell\Z,\R^n)]{f} := 
    \big( \norm[W^{k,\rho}((0,\ell),\R^n)]{f}^\rho + [{f^{(k)}}]^\rho 
    \big)^{1 /\rho}, $
        is called the \emph{periodic Sobolev-Slobodecki\v{\i} space}
    with
    fractional differentiability of order $k+s$ and integrability $\rho$.
    Here, $W^{k,\rho}(\R/\ell\Z,\R^n)\subset W^{k,\rho}_\textnormal{loc}
    (\R,\R^n)$
    denotes the usual Sobolev space of $\ell$-periodic functions whose weak
    derivatives $f^{(i)}$ of order $i\in\{0,\ldots,k\}$ 
    are locally $\rho$-integrable, where we have set $f^{(0)}:=f$.
\end{definition}
Since the restriction of any $\ell$-periodic Sobolev
function $f\in W^{k,\rho}
(\R/\ell\Z,\R^n)$ to its fundamental domain $(0,\ell)\subset\R$
is contained in the Sobolev space $W^{k,\rho}((0,\ell),\R^n)$ it suffices
to compare the seminorm \eqref{eq:defSeminorm} to the standard
Gagliardo seminorm (which we denote by $\llbracket\cdot\rrbracket $)
used in the literature on (non-periodic)
Sobolev-Slobodecki\v{\i} spaces, in order to transfer known results
to the periodic setting.
Indeed, for $w\in (-\ell/2,\ell/2)$ the absolute value $|w|$ in the integrand of \eqref{eq:defSeminorm}
equals the \emph{$\ell$-periodic distance}  $
|w|_{\R/\ell\Z}=\min_{k\in\Z}|w+k\ell |,$
so that we rewrite the seminorm 
\eqref{eq:defSeminorm} by substituting
$v(w):=u+w$, 
and estimate for 
$f \in L^\rho(\R/\ell\Z,\R^n)$ 
\begin{eqnarray}
\label{eq:classicVsPeriodicSeminorm}
\halfnorm[{s,\rho}]{f}^\rho
&= & \textstyle\int_{\R/\ell\Z} \int_{-{\ell}/2}^{{\ell}/2} 
\frac {\abs{f(u+w)-f(u)}^\rho} 
{\abs{w}^{1+s\rho}} \, \dd w \, \dd u
=\int_{\R/\ell\Z} \int_{-{\ell}/2}^{{\ell}/2} 
\frac {\abs{f(u+w)-f(u)}^\rho} 
{\abs{w}^{1+s\rho}_{\R/\ell\Z}} \, \dd w \, \dd u
\notag\\
&= &\textstyle\int_{\R/\ell\Z} \int_{u-{\ell}/2}^{u+{\ell}/2} 
\frac {\abs{f(v)-f(u)}^\rho} 
{\abs{v-u}^{1+s\rho}_{\R/\ell\Z}} \, \dd v \, \dd u
=\textstyle
\int_0^\ell \int_0^\ell 
\frac {\abs{f(v)-f(u)}^\rho} 
{\abs{v-u}_{\R/\ell\Z}^{1+s\rho}} \, \dd v \, \dd u \notag\\ 
&\geq  &\textstyle\int_0^\ell \int_0^\ell 
\frac {\abs{f(v)-f(u)}^\rho} 
{\abs{v-u}^{1+s\rho}} \, \dd v \, \dd u 
= \llbracket f \rrbracket_{s,\rho}^\rho, 
\end{eqnarray} 
where we used the $\ell$-periodicity of the integrand
to shift the domain  
of the inner integration by $-u+\ell/2$.

As a first application of inequality \eqref{eq:classicVsPeriodicSeminorm}
we show that the well-known Morrey embedding into classical H\"older
spaces transfers to the periodic setting. To be more precise,
the \emph{periodic H\"older seminorm} for an $\ell$-periodic function
$f\colon \R\to\R^n$ and $\alpha\in (0,1]$ is defined as
(cf. \cite[Def. 3.2.5]{grafakos_2008})
\begin{equation}\label{eq:hoelder-seminorm}
\hol_{\R/\ell\Z,\alpha}(f)
:=
\sup_{x,y\in [0,\ell]\atop x\not=y}\frac{
    \abs{f(x)-f(y)}}{\abs{x-y}^\alpha_{\R/\ell\Z}}=
\sup_{x\in [0,\ell]}\sup_{0<\abs{w}\le\ell/2}\frac{
    \abs{f(x+w)-f(x)}}{\abs{w}^\alpha},
\end{equation}
and the periodic H\"older space $C^{k,\alpha}(\R/\ell\Z,\R^n)$ 
for $k\in\N\cup
\{0\}$ consists of
those functions $f\in C^k(\R/\ell\Z,\R^n)$ whose  H\"older norm satisfies
$
\norm[C^{0,\alpha}(\R/\ell\Z,\R^n)]{f}:=\norm[C^0(\R/\ell\Z,\R^n)]{f}+
\hol_{\R/\ell\Z,\alpha}(f)<\infty.$
Notice that the second equality in \eqref{eq:hoelder-seminorm}
follows from the $\ell$-periodicity of $f$, since for any distinct $x,y\in
[0,\ell]$ with, say, $x<y$, we find $0\not= w\in [-\ell/2,\ell/2]$ such that $x+w=y$, or
$x+\ell=y+w$, namely $w:=y-x$ if $|y-x|\le \ell/2$, or $w:=\ell-(y-x)$ 
if
$\ell/2< |y-x|\le \ell$.
\begin{theorem}[Morrey embedding]\label{thm:morrey}
    Let $k\in\N\cup\{0\}$ and $\ell\in (0,\infty)$. If $\rho\in (1,\infty)$ and
    $s\in (1/\rho,1)$, then
    there is a positive constant $C_E=C_E(n, k,s,\rho,\ell)$ such that
    \begin{equation}\label{eq:morrey-embedding}
    \|f\|_{C^{k,s-1/{\rho}}(\R/\ell\Z,\R^n)}\le 
    C_E\|f\|_{W^{k+s,\rho}(\R/\ell\Z,\R^n)}\quad\textnormal{for all
        $f\in W^{k+s,\rho}(\R/\ell\Z,\R^n)$.}
    \end{equation}
\end{theorem}
\begin{proof}
    It suffices to treat the case $k=0$, the full statement follows by induction.
    Moreover, $\norm[C^0(\R/\ell\Z,\R^n)]{f}=\norm[{C^0([0,\ell],\R^n)}]{f}$
    for any $\ell$-periodic function $f\colon \R\to\R^n$,
    so that we can focus on the periodic
    H\"older seminorm.
    Rewrite the numerator in \eqref{eq:hoelder-seminorm} as 
    $|{f\!\circ\!\tau_{x-\frac{\ell}2}(w+\frac{\ell}2)\!-\!f\!\circ\!\tau_{x-\frac{\ell}2}
        (\frac{\ell}2)}|$ where we used the 
    shift $\tau_a(x):=x+a$ for some fixed $a\in\R$.
    Hence  the right-hand side of \eqref{eq:hoelder-seminorm}
    for $\alpha:=s-(1/\rho)\in (0,1)$ may be rephrased by setting $z:=w+\ell/2$,
    and estimated from above as
    \begin{equation}\label{eq:hoelder-shift}
    \sup_{x\in [0,\ell]}\sup_{z\in [0,\ell]\atop z\not=\ell/2}
    \frac{ |{f\circ\tau_{x-{\ell}/2}(z)-f\circ\tau_{x-{\ell}/2}
            ({\ell}/2)}|  }{\abs{z-(\ell/2)}^{s-1/\rho} }
    \le\sup_{x\in [0,\ell]}\hol_{[0,\ell],s-1/\rho}
    (f\circ\tau_{x-{\ell}/2}).
    \end{equation}
    By virtue of the non-periodic Morrey embedding
    \cite[Theorem 8.2]{DiNezzaPalatucciValdinoci:2012:HitchhikersGuide}
    and \eqref{eq:classicVsPeriodicSeminorm}
    the right-hand side may be bounded from above by  the expression
    $
    \sup_{x\in [0,\ell]}C\|{f\circ
        \tau_{x-{\ell}/2}}\|_{W^{s,\rho}([0,\ell],\R^n)}$ $
   \le$ $
    \sup_{x\in [0,\ell]}C\|{f\circ
        \tau_{x-{\ell}/2}}\|_{W^{s,\rho}(\R/\ell\Z,\R^n)},
    $
    for some constant $C=C(n,s,\rho,\ell)$. Changing variables via translations
    by
    $x-(\ell/2)$ and using the $\ell$-periodicity of the 
    respective integrands in the $W^{s,\rho}$-norm one can easily
    check that $\|{f\circ
        \tau_{x-{\ell}/2}}\|_{W^{s,\rho}(\R/\ell\Z,\R^n)}=
	\|{f
    }\|_{W^{s,\rho}(\R/\ell\Z,\R^n)} $ for each $x\in \R$, which concludes the proof.
    {}
\end{proof}

We frequently make use of the following Poincar\'e inequality.
\begin{proposition}[Poincaré inequality]
    \label{proposition:PoincareInequality}
    For every $\ell\in (0,\infty)$,  $s\in (0,1)$, 
    and $\rho \in [1, \infty)$
    there is a constant $C_P=C_P(s,\rho,\ell)$ such that
    \begin{equation}\label{eq:poincare} 
    \|{f}\|_{L^\varrho(\R/\ell\Z,\R^n)}\le  
     C_P [f]_{s,\varrho}
    \quad\textnormal{for all $f\in W^{1+s,\rho}(\R/\ell\Z,\R^n)$ with
    $\textstyle\int_{\R/\ell\Z}f(t)\,\dd t=0.$}
    \end{equation} 
\end{proposition}
\begin{proof}
By Jensen's inequality, we have
\begin{align*}
\textstyle
	\int_{\R/\ell\Z} \nabs{f(u)}^\varrho \, \dd u
	&=
	\textstyle	
	\int_{\R/\ell\Z} \nabs{f(u) - \fint_{\R/\ell\Z} f(t) \, \dd t}^\varrho 
	\, \dd u
	\leq
	\textstyle	
	\int_{\R/\ell\Z} \fint_{\R/\ell\Z}  \nabs{f(u) - f(t) }^\varrho 
	\, \dd t \, \dd u\\
	&\leq
	\textstyle	
	\frac{1}{\nabs{{\R/\ell\Z}}}
	(\nabs{{\R/\ell\Z}}/2)^{1+s \varrho }
	\int_{\R/\ell\Z} \int_{\R/\ell\Z}  
	\frac{\nabs{f(u) - f(t) }^\varrho}{|u-t|_{\R/\ell\Z}^{1+s 
	\varrho }} 
	\, \dd t \, \dd u.
\end{align*}
{}
\end{proof}

As a first application we deal with the composition of Sobolev-Slobodecki\v{\i}
functions.
\begin{lemma}
    \label{lemma:compositionSobolevSlobodeckijFunctions}
    Let  $L\in (0,\infty)$, $\rho\in (1,\infty)$, 
    $s \in (\frac 1 \rho,1)$, 
    and $g \in C^1([0,L])$ with  $|g'|>0$ on $[0,L]$, and $\ell:=
    g(L)-g(0)$, and $f \in W^{1+s,\rho}(\R/\ell\Z,\R^n)$.
    Then   $g$ can be extended to a  function $ G\in C^1(\R)$ 
    such that $f\circ  G\in C^1(\R/L\Z,\R^n)$ and $G'\in C^0(\R/L\Z),$ satisfying
    \begin{equation}
    \label{eq:composition2}	
    \left[(f \circ G)'\,\right]_{s,\rho}
    \leq C
    \big(\halfnorm[s,\rho]{G'}^\rho+v_g^{-1}\norm[C^0(\R/L\Z)]{G'}^{\rho(1+s)}\big)^{1/\rho}
    \halfnorm[s,\rho]{f'} 
    \end{equation}	
    for some constant $C=C(n,s,\rho,\ell)$,
    where $v_g:=\min_{[0,L]}|g'(\cdot)|>0.$ 
\end{lemma}
Notice that both sides in  
\eqref{eq:composition2} may be infinite as we do not assume that $g\in
W^{1+s,\rho}((0,L))$. We should mention also that we have suppressed 
the respective domains $[0,L]$, or $\R/L\Z$, or $\R/\ell\Z$, in our
notation for the seminorms of  $G$ and $f\circ G$,
or of  $f$ in these inequalities.
\begin{proof}
    In order to prove \eqref{eq:composition2} we 
    extend $g$ onto all of $\R$, first by
    setting $G(x):=g(x)$ for $x\in (0,L]$ and
    $G(x):=g(x-L)+\ell$ consecutively
    for $x\in (L,2L],
    (2L,3L],\ldots,$ and then by $G(x):=g(x+L)-\ell$
    consecutively for $x\in (-L,0], (-2L,-L],
    (-3L,-2L],\ldots,$  to  find
    $
    G(x+L)=G(x)+\ell$ $G'(x+L)=G'(x)$ for all $ x\in\R.
    $
    We  then bound the numerator in
    the integrand
    of the seminorm $[(f\circ G)'\,]_{s,\rho}$ for $x\in\R$ and 
    $|w|\le \frac{L}2$ from above as
    $
    \abs{(f\circ G)'(x+w)-(f\circ G)'(x)}^\rho 
    \leq 2^{\rho-1} \bigl(\abs{G'(x+w)}^\rho
    \abs{f'(G(x+w))-f'(G(x))}^\rho
    \abs{f'(G(x))}^\rho \abs{G'(x+w)-G'(x)}^\rho \bigr),$
    which  leads to 
    \begin{align}
    [(f\circ G)'\,]_{s,\rho}^\rho  
    &\leq 2^{\rho-1}\textstyle \int_{\R/L\Z} 
    \int_{- {L}/ 2}^{ {L}/ 2} 
    \abs{w}^{-\alpha} 
    \abs{G'(x+w)}^\rho\abs{f'(G(x+w))-f'(G(x))}^\rho \, \dd w \, \dd x\notag\\
    &\quad + 2^{\rho-1} 
    \norm[C^0(\R/\ell\Z,\R^n)]{f'}^\rho 
    \halfnorm[s,\rho]{G'\,}^\rho,\label{eq:compo-est1}
    \end{align}
    where we have set $\alpha:=1+\rho s$. Note that $f$ is of class 
    $C^1(\R/\ell\Z,\R^n)$ due to \thref{thm:morrey}.
    From now on we may assume without loss of generality that the seminorm
    $\halfnorm[s,\rho]{G'}$ is finite, otherwise there is nothing to prove
    for \eqref{eq:composition2}.
    Substituting $u(w):=x+w$ and using periodicity 
    we can rewrite the double integral as
    \begin{equation}\label{eq:transform1}
    \textstyle\int_{\R/L\Z} \int_{- {L}/ 2}^{ {L}/ 2} 
    \abs{u-x}_{\R/L\Z}^{-\alpha} 
    \abs{G'(u)}^\rho \abs{f'(G(u))-f'(G(x))}^\rho \, \dd u \, \dd x.
    \end{equation}	
    With
    $
    \abs{G(u)-G(x)}_{\R/\ell\Z} \leq \abs{u-x}_{\R/L\Z}
    \norm[C^0(\R/L\Z)]{G'}
    $
    we obtain for the inverse function $G^{-1}:\R\to\R$  for all $x,u\in\R$
    \begin{equation}\label{eq:inverse-est}
    \abs{G^{-1}(G(u)) - G^{-1}(G(x))}_{\R/L\Z} 
    = \abs{u-x}_{\R/L\Z}
    \geq  \norm[C^0
        (\R/L\Z)]{G'}^{-1} \abs{G(u)-G(x)}_{\R/\ell\Z}.
    \end{equation}	
    Now we use the transformation
   $ 
    \Phi\colon (\R/L\Z)^2 \to (\R/\ell\Z)^2, 
    (u,x) \mapsto (\tilde u,\tilde x):=(G(u),G(x))
   $
    with $\abs{\det(D\Phi(u,x))}=G'(u)G'(x)\ge v_g^2>0$ for all $x,v\in\R$
    due to periodicity of $G'$ and $G'|_{[0,L]}=g', $
    to rewrite and estimate \eqref{eq:transform1} by means of
    \eqref{eq:inverse-est} as
    \begin{align}
    &\textstyle\iint_{(\R/\ell\Z)^2} 
    \abs{G^{-1}(\tilde u)-G^{-1}(\tilde x)}^{-\alpha}_{\R/L\Z} 
     {\abs{G'(G^{-1}(\tilde u))}^{\rho-1}} 
    \abs{G'(G^{-1}(\tilde x))}^{-1} 
    \abs{f'(\tilde u)-f'(\tilde x)}^\rho \,\dd\tilde u \, \dd\tilde  x \notag\\
    & \phantom{xxxxxxxxxxxxxx}\overset{\eqref{eq:inverse-est}}{\leq}
    v_g^{-1} {\norm[C^0(\R/L\Z)]{G'}^{\rho -1 +\alpha}}  
    \halfnorm[s,\rho]{f'}^\rho
    = v_g^{-1} {\norm[C^0(\R/L\Z)]{G'}^{\rho(1+ s)}} 
    \halfnorm[s,\rho]{f'}^\rho.\label{eq:second-int-est}
    \end{align}
    The $C^0$-norm of $f'$ in \eqref{eq:compo-est1} may be bounded
    by $C_E C_P[f']_{s,\rho} $ 
    according to the Morrey embedding, \thref{thm:morrey}, and
    the Poincar\'e inequality, 
    \thref{proposition:PoincareInequality}, and combining this with
    \eqref{eq:second-int-est} leads to \eqref{eq:composition2} with 
    the constant $C:= (2^{\rho-1}((C_EC_P)^\rho +1))^{1/\rho}$.
    {}
\end{proof}

The following result provides an estimate 
for the Sobolev-Slobodecki\v{\i} norms
of products of functions of fractional Sobolev regularity.
\begin{proposition}[Bounds on inner product]
    \label{proposition:WsrhoClosedunderPointwiseMultiplication}
    Let $k\in\N\cup\{0\}$ and $\ell\in (0,\infty)$, $\rho \in (1,\infty)$, 
    and $s\in (1/\rho, 1)$. 
    Then, there is  a constant $C=C(n, k ,s,\rho,\ell) \in (0,\infty)$ 
    such that 
    \[
    \norm[W^{k+s,\rho}(\R/\ell\Z)]{\langle f(\cdot),g(\cdot)\rangle_{\R^n}} 
    \!\leq\! C \norm[W^{k+s,\rho}(\R/\ell\Z,\R^n)]{f}\!
    \norm[W^{k+s,\rho}(\R/\ell\Z,\R^n)]{g}
    \forall f,g \in W^{k+s,\rho}(\R/\ell\Z,\R^n).
    \] 
\end{proposition}
\begin{proof}
    For brevity, we write $\langle \xi,\zeta \rangle := 
    \langle \xi,\zeta \rangle_{\R^n}$ for $\xi,\zeta \in \R^n$, and then
    $\langle F,G\rangle := \langle F(\cdot),G(\cdot)\rangle$ for any two
    functions
    $F,G\colon \R\to\R^n$.
    By virtue of the  Morrey embedding, \thref{thm:morrey}, the
    functions $f,g $ are both of class $C^{k}$.
    Then, according to the Leibniz rule, we have for every $0\le i\le k$
    \begin{equation}\label{eq:leibniz}	
    \textstyle\langle	 f,g\rangle^{(i)} = 
    \sum_{j=0}^i \binom{i}{j} \langle f^{(j)} ,g^{(i-j)}\rangle,
    \end{equation}	
    of which the $L^\rho$-norm can be estimated as 
    \begin{align*}
    \|{\langle f,g\rangle^{(i)}}\|^\rho_{L^\rho(\R/\ell\Z)}
    &=\textstyle \int_{\R/\ell\Z} \big|{\sum_{j=0}^i \binom{i}{j} 
        \langle f^{(j)}(x), g^{(i-j)}(x)\rangle}\big|^\rho \,dd x\\
    & \leq 2^{(2\rho-1)i }
    \norm[W^{i,\rho}(\R/\ell\Z,\R^n)]{f}^\rho \norm[C^i(\R/\ell\Z,
    \R^n)]{g}^\rho
    \Foa 0\le i\le k,
    \end{align*}
    which implies by means of the Morrey embedding \eqref{eq:morrey-embedding}
    that the Sobolev norm
    $\|\langle f,g\rangle\|_{W^{k,\rho}}^\rho$ may be bounded from above by
    \begin{align}
    & C_E^{\rho} (k+1)2^{(2\rho-1)k}\|f\|^\rho_{W^{k,\rho}(\R/\ell\Z,\R^n)}
    \|g\|_{W^{k+s,\rho}(\R/\ell\Z,\R^n)}^{\rho}. \label{eq:skp-sobolev-bound}
    \end{align}
    
    The numerator of the integrand in the
    seminorm $[{\langle f,g \rangle^{(k)}}]_{s,\varrho}$
    can 
    be estimated from above by means of \eqref{eq:leibniz} similarly as before
    by
    \begin{align*}
    &
    \textstyle
    2^{(2\rho-1)k}\sum_{j=0}^k\left|\langle f^{(j)}, g^{(k-j)}\rangle(x+w)-
    \langle f^{(j)}, g^{(k-j)}\rangle(x)\right|^\rho
    \\
    & \le \textstyle 2^{(2\rho-1)(k+1)}\sum_{j=0}^k\big\{ \,\left|\langle f^{(j)}(x+w),g^{(k-j)}(x+w)
    -g^{(k-j)}(x)\rangle\right|^\rho\\
    & \qquad +\textstyle
    \left|\langle f^{(j)}(x+w)-f^{(j)}(x),g^{(k-j)}(x)\rangle\right|^\rho
    \big\}\quad\Fo x\in\R, |w|\le{\ell}/2,
    \end{align*}
    which may be bounded by 
    \begin{align*}
    &\textstyle 
     2^{(2\rho-1)(k+1)}\sum_{j=0}^k
    \big\{\|f\|^\rho_{C^k(\R/\ell\Z,\R^n)}\left|
    g^{(j)}(x+w)
    -g^{(j)}(x)\right|^\rho 
    \\
    & \qquad +\textstyle 
    \|g\|^\rho_{C^k(\R/\ell\Z,\R^n)}\left|f^{(j)}(x+w)
    -f^{(j)}(x)\right|^\rho\big\} \Fo x\in\R,  |w|\le{\ell}/2.
    \end{align*}
    Integrating this  against $|w|^{-(1+s\rho)}$
    and using the Morrey embedding inequality
    \eqref{eq:morrey-embedding} we find 
    \begin{align*}
    &[{\langle f,g\rangle^{(k)}}]_{s,\varrho}^\rho \le 
    \textstyle 2^{(2\rho-1)(k+1)}\big\{
    \|f\|^\rho_{C^k(\R/\ell\Z,\R^n)}\sum_{j=0}^k\halfnorm[s,\rho]{g^{(j)}}^\rho
    +
    \|g\|^\rho_{C^k(\R/\ell\Z,\R^n)}\sum_{j=0}^k\halfnorm[s,\rho]{f^{(j)}}^\rho
    \big\}\\
    &\overset{\eqref{eq:morrey-embedding}}{\le} 
    \textstyle 2^{(2\rho-1)(k+1)} \textstyle C_E^\rho 
    \big\{
    \|f\|^\rho_{W^{k+s,\rho}(\R/\ell\Z,\R^n)}
    \sum_{j=0}^k\halfnorm[s,\rho]{g^{(j)}}^\rho
    +
    \|g\|^\rho_{W^{k+s,\rho}(\R/\ell\Z,\R^n)}
    \sum_{j=0}^k\halfnorm[s,\rho]{f^{(j)}}^\rho
    \big\}.
    \end{align*}
    For $i <k$ we observe that 
    \begin{align*}
    [{g^{(i)}}]_{s,\varrho}^\rho
    &= \textstyle\int_{\R/\ell\Z} \int_{- \ell/ 2} ^{ \ell/ 2} 
    {\abs{g^{(i)}(u+w) - g^{(i)}(u)}^\rho} {\abs{w}^{-(1+s\rho)}} \, \dd w
    \, \dd u
    \\
    &\leq \|{g^{(i+1)}}\|_{L^\infty(\R/\ell\Z,\R^n)}^\rho
    \textstyle\int_{- \ell/
        2}^{ \ell/ 2} \abs{w}^{\rho - 1 - s\rho} \, \dd w
    =: \|{g^{(i+1)}}\|_{L^\infty(\R/\ell\Z,\R^n)}^\rho c(s,\rho,\ell)\\
    &
    \overset{\eqref{eq:morrey-embedding}}{\leq} C_E^\rho 
    \norm[W^{k+s,\rho}(\R/\ell \Z,\R^n)]{g}^\rho c(s,\rho,\ell),
    \end{align*}
    since $\rho - 1 - s\rho > \rho -2 > -1$, and by Morrey's embedding
    \eqref{eq:morrey-embedding}.
    Analogous estimates hold for the seminorms of $f^{(j)}$, which combine with
    \eqref{eq:skp-sobolev-bound} to the inequality as claimed
    with constant  
    $
    C(n, k,s,\rho,\ell)
    :=2^{(2\rho-1)(k+1)/\rho} C_E(k+1)^{1/\rho}\left[1+2c(s,\rho,\ell)C_E^\rho
    \right]^{1/\rho}.
    $
    Notice that
    target dimension $n$
    enters this constant 
    through the Morrey embedding constant
    $C_E$; cf. \thref{thm:morrey}.
    {}
\end{proof}

We conclude this section with an estimate on the fractional Sobolev norm of the 
multiplicative inverse of a non-vanishing 
Sobolev-Slobodecki\v{\i}  function.
\begin{lemma}
    \label{lemma:multiplicativeInverseSobolevFunction}
    Let $\ell\in (0,\infty)$, $\rho\in (1,\infty)$, $ s\in (\frac 1 \rho,1)$, $c>0$ and $h \in W^{s,\rho}(\R/\ell\Z)$ with $\abs{h(x)}\geq c$ for almost all $x \in \R$.
    Then the multiplicative inverse, $1/h\colon \R \to \R, x \mapsto  1/ {h(x)}$ is in 
    $W^{s,\rho}(\R/\ell\Z)$, and we have 
    $
    \norm[W^{s,\rho}(\R/\ell\Z)]{1/ h} \leq  \norm[W^{s,\rho}(\R/\ell\Z)]{h}/c^2.
    $
\end{lemma}
\begin{proof}
    The estimate
    $
    \norm[L^\rho(\R/\ell\Z)]{1/ h}^\rho \leq \ell
    \norm[L^\infty(\R/\ell\Z)]{1/ h}^\rho
    \leq \ell/ c^\rho$
    follows immediately from the pointwise lower bound on $h$.
    The integrand of the seminorm $\halfnorm[s,\rho]{1/h}$ may be estimated as
    $
    \abs{w}^{-1-s\rho}\abs{{h(x+w)}}^{-\rho} \abs{{h(x)}}^{-1\rho} \abs{h(x+w)-h(x)}^\rho 
    \!\leq\!  {c^{-2\rho}}{|w|^{-1-s\rho}}{\abs{h(x+w)-h(x)}^\rho},
    $
    which implies $\halfnorm[s,\rho]{1/h}^\rho\le
    c^{-2\rho}\halfnorm[s,\rho]{h}^\rho,$
    and together with $\norm[L^\rho(\R/\ell\Z)]{h}^\rho \geq c^\rho
    \ell $ we obtain
    the desired estimate for $\norm[W^{s,\rho}(\R/\ell\Z)]{1/h}$.
    {}
\end{proof}
\section{Arc length and bilipschitz constants}
\label{section:Arclength}
Whenever we speak about  \emph{the arc length parametrization} of a curve 
$\g\in W^{1,1}(\R/\ell\Z,\R^n)$ with $|\g'|>0$ a.e.\ on $\R$ and with length
$L:=\LL(\g)\in (0,\infty)$ we refer to the mapping $\Gamma:=\g\circ \LL(\g,\cdot)^{-1}
\in C^{0,1}(\R/L\Z,\R^n)$ obtained from the inverse $\LL(\g,\cdot)^{-1}$
of the strictly increasing \emph{arc length function $\LL(\g,\cdot)$} of class $W^{1,1}_\textnormal{loc}(\R)$ defined as
\begin{equation}\label{eq:arclength-function1}
\LL(\g,x):= \textstyle\int_0^{x} \abs{\gamma'(y)} \, \dd y \quad\Fo x\in\R.
\end{equation}
Notice that $\LL(\g,x+\ell)=\LL(x)+L$ and  $(\LL(\g,x+\ell))'=
(\LL(\g,x))'$ for all $x\in\R$, as well as 
$\LL(\g,z+L)^{-1}=
\LL(\g,z)^{-1}+\ell$ for all $z\in\R$, and therefore the composition $\Gamma=\g\circ
\LL(\g,\cdot)^{-1}$ is $ L$-periodic and $1$-Lipschitz because of 
\begin{equation}\label{eq:unit-speed}
\textstyle\left|\frac{d}{ds}\Gamma(s)\right|=\big|\g'\left(\LL(\g,s)^{-1}\right)
\frac{1}{|\g'(x)|}_{x=\LL(\g,s)^{-1}}\big|=1\quad\textnormal{for a.e. $s\in [0,L].$}
\end{equation}

Recall from the introduction the definition of the bilipschitz constant
\eqref{eq:bilip-def} and
the notation 
$W^{1,1}_{\ir}(\R/\ell\Z,\R^n)$ for the subset of all regular $W^{1,1}$-curves such
that their restriction to the fundamental interval $[0,\ell) $ is injective. 
\begin{lemma}
    \label{lemma:arclengthParametrizationAndBiLipschitzConstants}
    Let $\ell\in (0,\infty)$ and  $\gamma \in W_{\ir}^{1,1}(\R/\ell\Z,\R^n)$ with length
    $L:=\LL(\g)\in (0,\infty)$ and with arc length parametrization 
    $ \Gamma \in W^{1,\infty}(\R/L\Z,\R^n)$. Then the bilipschitz
    constants of $\g$ and $\Gamma$ are related by
    \begin{equation}\label{eq:bilip-comparison}
    \BiLip(\g)\in [\,v_\g\cdot\BiLip(\Gamma),\,v_\g\,],
    \end{equation}
    where $
    v_\g =\essinf_{[0,\ell]}|\g'(\cdot)|\ge 0.$
\end{lemma}
\begin{proof}
    The upper bound on $\BiLip(\g)$ in \eqref{eq:bilip-comparison}
    follows immediately from the estimate
    $
    |\g'(x)|=\lim_{h\to 0}{\abs{h}_{\R/\ell\Z}}^{-1}{\abs{\g(x+h)-\g(x)}}
    \ge\BiLip(\g)
    $
    at points $x\in\R/\ell\Z$, where $\g'(u)$ exists. Taking the infimum over all
    such $x\in\R/\ell\Z$ establishes the upper bound in
    \eqref{eq:bilip-comparison}.
    For the proof of the lower bound in \eqref{eq:bilip-comparison}
    assume w.l.o.g. that $0\le y\le  x<\ell$, 
    so that by monotonicity $0\le\LL(\g,y)\le
    \LL(\g,x)<L$.  Then $|\g(x)-\g(y)|$ 
    can be written as and estimated from below by
    \begin{equation}\label{eq:bilip-comp1}
    |\Gamma\circ \LL(\g,x)-\Gamma\circ\LL(\g,y)|\ge \BiLip(\Gamma)
    |\LL(\g,x)-\LL(\g,y)|_{\R/L\Z}.
    \end{equation}
    If $|\LL(\g,x)-\LL(\g,y)|_{\R/L\Z}=|\LL(\g,x)-\LL(\g,y)|
    =(\LL(\g,x)-\LL(\g,y))$, then  the
    distance on 
    the right-hand side of \eqref{eq:bilip-comp1} equals
$ \int_{y}^{x}|\g'(\tau)|\,d\tau  \ge v_\g(x-y)=
    v_\g |x-y|
    \ge v_\g|x-y|_{\R/\ell\Z},$
    which proves the claim in this case.
    If, on the other hand, $|\LL(\g,x)-\LL(\g,y)|_{\R/L\Z}=
    L-|\LL(\g,x)-\LL(\g,y)|$ the latter can be written as
    \begin{align*}
    L-
    (\LL(\g,x)-\LL(\g,y))
    & =L-\textstyle\int_{y}^{x}|\g'(\tau)|\,\dd\tau=\int_0^{y}|\g'(\tau)|\,\dd\tau
    +\int_{x}^\ell|\g'(\tau)|\,\dd\tau,
    \end{align*}
    which can be bounded from below by $
     v_\g(y+\ell-x)=v_\g(\ell-|x-y|)
    \ge v_\g|x-y|_{\R/\ell\Z},$ and this
    again can be inserted into \eqref{eq:bilip-comp1} to yield the
    desired inequality \eqref{eq:bilip-comparison}.
    {}
\end{proof}

\begin{lemma}[Seminorm of arc length]
    \label{lemma:estimateHalfnormLL}
    Let $s \! \in \! (0,1)$, $\rho \!\in\! (1,\infty)$, and  $\gamma \in W^{1,\rho}(\R/\ell\Z,\R^n)$.
    Then, 
   $ 
    [ (\LL(\gamma,\cdot))']_{s,\rho}\le \halfnorm[s,\rho]{\g'}.
   $
\end{lemma}
Note that the right-hand side of this estimate might be infinite, since we 
did not assume that $\g$ is of class $W^{1+s,\rho}(\R/\ell\Z,\R^n)$.
\begin{proof}
If the seminorm on the right-hand side is infinite there is nothing to prove. 
    If, on the other hand $\g\in W^{1+s,\rho}(\R/\ell\Z,\R^n)$ then we
    estimate
    \begin{align*}
    \halfnorm[s,\rho]{(\LL(\gamma,\cdot))'}^\rho
    &=\textstyle\int_{\R/\ell\Z}\int_{-{\ell}/2}^{{\ell}/2} 
    \abs{w}^{-1-s\rho}{\big|{ \abs{\gamma'(x+w)} - \abs{\gamma'(x)} }\big|^\rho} \,\dd w \, \dd x\\
    &\leq\textstyle \int_{\R/\ell\Z}\int_{-{\ell}/2}^{{\ell}/2} 
    \abs{w}^{-1-s\rho}{ \abs{\gamma'(x+w) - \gamma'(x) }^\rho} \dd w \, \dd x
    =\halfnorm[s,\rho]{\gamma'}^\rho.
    \end{align*}
    \vspace{-1cm}
    {}
\end{proof}

The bilipschitz constant $\BiLip(\g)$, defined in
\eqref{eq:bilip-def} of the introduction,
of a closed embedded, i.e., immersed and injective   
$C^1$-curve is strictly positive, since under this regularity 
assumption the quotient in
\eqref{eq:bilip-def} converges to $|\g'(u_2)|\ge v_\g =\min_{[0,\ell]}|\g'|>0$
as $u_1\to u_2$. This means that this
positive quotient possesses a continuous extension onto all of $\R\times\R$,
which is bounded by a 
positive constant from below by periodicity and continuity.
We now quantify a neighbourhood of such a curve $\g$, in which the bilipschitz
constant remains under control.
\begin{lemma}\label{lem:tangents-close}
    Let $\ell\in (0,\infty)$ and let $\g\in C^1(\R/\ell\Z,\R^n)$ be injective on $[0,\ell)$
    with $|\g'|>0$ on $\R.$ Then
    \begin{equation}\label{eq:bilip1}
    \BiLip(\eta)\ge  \BiLip(\g)/2>0
    \end{equation}
    for all $\eta\in C^1(\R/\ell\Z,\R^n)$ with $
    \|\eta'-\g'\|_{C^0(\R/\ell\Z,\R^n)}\le  \BiLip(\g)/2.$
\end{lemma}
\begin{proof}
    It suffices to prove the first inequality in \eqref{eq:bilip1}, and for that assume 
    w.l.o.g. that
    $0\le y\le x<\ell.$ We have
    \begin{equation}\label{eq:eins}
    |\eta(x)-\eta(y)|  \ge |\g(x)-\g(y)|-
    |(\eta-\g)(x)-(\eta-\g)(y)|.
    \end{equation}
    If $|x-y|_{\R/\ell\Z}=|x-y|=(x-y)$, 
    then we can estimate the right-hand side from below
    by 
    $
    \BiLip(\g)|x-y|_{\R/\ell\Z}-\textstyle\int_{y}^{x}
    \left|(\eta-\g)'(\tau)\right|\,\dd\tau
    \ge\BiLip(\g)|x-y|_{\R/\ell\Z}-\|\eta'-\g'\|_{C^0(\R/\ell\Z,\R^n)}
    |x-y|_{\R/\ell\Z}$ which is bounded from below by $
\BiLip(\g)|x-y|_{\R/\ell\Z}/2.$
    If, on the other hand $|x-y|_{\R/\ell\Z}=\ell-|x-y|
    =\ell-(x-y)$ then we use the
    $\ell$-periodicity to replace the term $(\eta-\g)(y)$ in \eqref{eq:eins}
    by $(\eta-\g)(\ell+y)$
    to estimate the right-hand side of \eqref{eq:eins} from below by
    $\BiLip(\g)|x-y|_{\R/\ell\Z}-\textstyle\int_{x}^{\ell+y}|
    (\eta-\g)'(\tau)|\,\dd\tau $, which can be bounded from below by $
 \BiLip(\g)|x-y|_{\R/\ell\Z}-
    \|\eta'-\g'\|_{C^0(\R/\ell\Z,\R^n)}|x-y|_{\R/\ell\Z}
    \ge  \BiLip(\g)|x-y|_{\R/\ell\Z}/2,$
    which concludes the proof.
    {}
\end{proof}
The following corollary quantifies the fact that the set $W^{1+s,\rho}_\textnormal{ir}(\R/\ell\Z,\R^n)$
of all regular embedded Sobolev-Slobodecki\v{i} loops is an open subset of
that fractional Sobolev space if that spaces embeds into
$C^1(\R/\ell\Z,\R^n)$.
\begin{corollary}\label{cor:bilip2}
    Let $\ell\!\in\!(0,\infty)$, $\rho\!\in\!(1,\infty)$, and $s\!\in\!(1/\rho,1)$, and suppose
    $\g\in W^{1+s,\rho}_\textnormal{ir}(\R/\ell\Z,\R^n)$. Then $\BiLip(\g)>0$, and
    one has  for  the bilipschitz constant $\BiLip(\eta)$ and the
    minimal velocity $v_\eta:=\min_{[0,\ell]}|\eta'|$
    \begin{equation}\label{eq:bilip2}
    \BiLip(\eta)\ge \BiLip(\g)/2  \quad\AND\quad v_\eta \ge  v_\gamma/2
    \end{equation}
    for all $\eta\in W^{1+s,\rho}(\R/\ell\Z,\R^n)$ satisfying  $
    \halfnorm[s,\rho]{\eta'-\g'}
    \le ({2C_EC_P})^{-1}\BiLip(\g),$
    where $C_E=C_E(n, s,\rho,\ell)$ denotes the constant 
    in \eqref{eq:morrey-embedding} of
    the Morrey embedding, \thref{thm:morrey}, 
    and $C_P=C_P(n, s,\rho,\ell) $ is the constant in
    the Poincar\'e inequality \eqref{eq:poincare}. 
\end{corollary}
Notice that neither in this corollary nor in the preceding 
\thref{lem:tangents-close}
anything is required about the $C^0$-distance between the curves, so in principle
$\g$ and $\eta$ could be very far apart in $\R^n$.
\begin{proof}
    By the Morrey embedding, \thref{thm:morrey}, the curve
    $\g$ is of class $C^1$, so that
    $\BiLip(\g)>0$. Moreover, again by \thref{thm:morrey},
    in combination
    with the Poincar\'e inequality, \thref{proposition:PoincareInequality}, we estimate using our assumption on $\eta$
    \begin{align*}
    \|\eta'-\g'\|_{C^0(\R/\ell\Z,\R^n)} & \overset{\eqref{eq:morrey-embedding}}{\le}
    C_E\|\eta'-\g'\|_{W^{s,\rho}(\R/\ell\Z,\R^n)}          
     \overset{\eqref{eq:poincare}}{\le} C_EC_P\halfnorm[s,\rho]{\eta'-\g'}
    {\le}  \BiLip(\g)/2,
    \end{align*}
    so that we can apply  \thref{lem:tangents-close} to finish the proof.
Similarly, again by 
     Morrey's embedding \eqref{eq:morrey-embedding} and Poincar\'e's inequality
         \eqref{eq:poincare} one finds
	         $v_{\eta}  \ge  v_\g-\|\eta'-\g'\|_{C^0(\R/\Z,\R^n)}
		     {\ge} v_\g-C_EC_P
		         \halfnorm[\frac32 p-3,2]{\eta'-\g'}
			      {>}   
    v_\g- \BiLip(\g)/2\ge  v_\g/2>0.
$
    {}
\end{proof}

{\small
 \section*{Acknowledgments}
D. Steenebr\"ugge is partially supported
 by the 
 Graduiertenkolleg \emph{Energy, Entropy, and Dissipative Dynamics (EDDy)} of the Deutsche Forschungsgemeinschaft (DFG) – project no. 
 320021702/GRK2326. H. Schumacher and H. von der Mosel
 gratefully acknowledge support 
 from DFG-project 282535003: {\it Geometric curvature functionals: 
 energy landscape and discrete methods.}
This work is also partially funded by the
Excellence Initiative of the German federal and state
			                governments.
Parts of Section \ref{sec:2} are contained in J. Knappmann's Ph.D.
 thesis \cite{knappmann_2020},
  and some of the content  of Sections
   \ref{section:ShortTimeExistence}--\ref{sec:6}
    will appear
     in D. Steenebr\"ugge's Ph.D. thesis
      \cite{steenebruegge_2020}.

\renewcommand{\bibname}{References}
\phantomsection\addcontentsline{toc}{section}{\bibname}
\markboth{\itshape\bibname}{\itshape\bibname}
\bibliographystyle{acm}
\bibliography{Paper.bib}

\begin{thebibliography}{10}

\bibitem{AbrahamMarsdenRatiu:1988:ManifoldsTensoAnalysis}
{\sc Abraham, R., Marsden, J.~E., and Ratiu, T.}
\newblock {\em Manifolds, tensor analysis, and applications}, second~ed.,
  vol.~75 of {\em Applied Mathematical Sciences}.
\newblock Springer-Verlag, New York, 1988.

\bibitem{Alt:2016:LinearFunctionalAnalysis}
{\sc Alt, H.~W.}
\newblock {\em Linear functional analysis}.
\newblock Universitext. Springer-Verlag London, Ltd., London, 2016.
\newblock An application-oriented introduction, Translated from the German
  edition by Robert N\"{u}rnberg.

\bibitem{ashton-etal_2011}
{\sc Ashton, T., Cantarella, J., Piatek, M., and Rawdon, E.~J.}
\newblock Knot tightening by constrained gradient descent.
\newblock {\em Exp. Math. 20}, 1 (2011), 57--90.

\bibitem{bartels-reiter_2018}
{\sc {Bartels}, S., and {Reiter}, {\relax Ph}.}
\newblock {Stability of a simple scheme for the approximation of elastic knots
  and self-avoiding inextensible curves}.
\newblock {\em arXiv e-prints\/} (Apr 2018), arXiv:1804.02206.
\newblock to appear in Math. Comp.

\bibitem{bartels-etal_2018}
{\sc Bartels, S., Reiter, {\relax Ph}., and Riege, J.}
\newblock A simple scheme for the approximation of self-avoiding inextensible
  curves.
\newblock {\em IMA J. Numer. Anal. 38}, 2 (2018), 543--565.

\bibitem{blatt_2009a}
{\sc Blatt, S.}
\newblock Note on continuously differentiable isotopies, 2009.
\newblock Preprint no. 34, RWTH Aachen.

\bibitem{blatt_2012b}
{\sc Blatt, S.}
\newblock The gradient flow of the {M}\"obius energy near local minimizers.
\newblock {\em Calc. Var. Partial Differential Equations 43}, 3-4 (2012),
  403--439.

\bibitem{blatt_2013a}
{\sc Blatt, S.}
\newblock A note on integral {M}enger curvature for curves.
\newblock {\em Math. Nachr. 286}, 2-3 (2013), 149--159.

\bibitem{blatt_2018a}
{\sc Blatt, S.}
\newblock The gradient flow of {O}'{H}ara's knot energies.
\newblock {\em Mathematische Annalen 370}, 3 (Apr 2018), 993--1061.

\bibitem{blatt_2020a}
{\sc Blatt, S.}
\newblock The gradient flow of the {M}\"{o}bius energy:
  {$\varepsilon$}-regularity and consequences.
\newblock {\em Anal. PDE 13}, 3 (2020), 901--941.

\bibitem{BlattReiter:2015:Menger}
{\sc Blatt, S., and Reiter, {\relax Ph}.}
\newblock Towards a regularity theory for integral {M}enger curvature.
\newblock {\em Ann. Acad. Sci. Fenn. Math. 40}, 1 (2015), 149--181.

\bibitem{Cartan:1971:Differentialcalculus}
{\sc Cartan, H.}
\newblock {\em Differential Calculus}.
\newblock {Hermann [u.a.]}, {Paris}, 1971.

\bibitem{crowell-fox_1977}
{\sc Crowell, R.~H., and Fox, R.~H.}
\newblock {\em Introduction to knot theory}.
\newblock Springer-Verlag, New York-Heidelberg, 1977.
\newblock Reprint of the 1963 original, Graduate Texts in Mathematics, No. 57.

\bibitem{denne-sullivan_2008}
{\sc Denne, E., and Sullivan, J.~M.}
\newblock Convergence and isotopy type for graphs of finite total curvature.
\newblock In {\em Discrete differential geometry}, vol.~38 of {\em Oberwolfach
  Semin.} Birkh\"{a}user, Basel, 2008, pp.~163--174.

\bibitem{DiNezzaPalatucciValdinoci:2012:HitchhikersGuide}
{\sc Di~Nezza, E., Palatucci, G., and Valdinoci, E.}
\newblock Hitchhiker's guide to the fractional {S}obolev spaces.
\newblock {\em Bull. Sci. Math. 136}, 5 (2012), 521--573.

\bibitem{drabek-milota_2013}
{\sc Dr\'{a}bek, P., and Milota, J.}
\newblock {\em Methods of nonlinear analysis}, second~ed.
\newblock Birkh\"{a}user Advanced Texts: Basler Lehrb\"{u}cher. [Birkh\"{a}user
  Advanced Texts: Basel Textbooks]. Birkh\"{a}user/Springer Basel AG, Basel,
  2013.
\newblock Applications to differential equations.

\bibitem{gilsbach_2018}
{\sc Gilsbach, A.}
\newblock {\em On symmetric critical points of knot energies}.
\newblock PhD thesis, RWTH Aachen University, 2018.
\newblock http://publications.rwth-aachen.de/record/726186/files/726186.pdf.

\bibitem{grafakos_2008}
{\sc Grafakos, L.}
\newblock {\em Classical {F}ourier analysis}, second~ed., vol.~249 of {\em
  Graduate Texts in Mathematics}.
\newblock Springer, New York, 2008.

\bibitem{gromov_1983}
{\sc Gromov, M.}
\newblock Filling {R}iemannian manifolds.
\newblock {\em J. Differential Geom. 18}, 1 (1983), 1--147.

\bibitem{hermes_2014}
{\sc {Hermes}, T.}
\newblock {Analysis of the first variation and a numerical gradient flow for
  integral Menger curvature}.
\newblock {\em ArXiv e-prints\/} (Aug. 2014).

\bibitem{MR163318}
{\sc Hudson, J. F.~P., and Zeeman, E.~C.}
\newblock On combinatorial isotopy.
\newblock {\em Inst. Hautes \'{E}tudes Sci. Publ. Math.}, 19 (1964), 69--94.

\bibitem{knappmann_2020}
{\sc Knappmann, J.}
\newblock {\em On the second variation of integral {M}enger curvature}.
\newblock PhD thesis, RWTH Aachen University, 2020.
\newblock online available from
  \href{https://publications.rwth-aachen.de/record/802770/}{https://publications.rwth-aachen.de/record/802770/}.

\bibitem{KunzeHoffman:1971:LinearAlgebra}
{\sc Kunze, R., and Hoffman, K.}
\newblock {\em Linear {{Algebra}}}, second~ed.
\newblock {Prentice-Hall, Inc.}, {Englewood Cliffs, New Jersey}, 1971.

\bibitem{lin-schwetlick_2010}
{\sc Lin, C.-C., and Schwetlick, H.~R.}
\newblock On a flow to untangle elastic knots.
\newblock {\em Calc. Var. Partial Differential Equations 39}, 3-4 (2010),
  621--647.

\bibitem{Neuberger:1997:SobolevGradients}
{\sc Neuberger, J.~W.}
\newblock {\em Sobolev gradients and differential equations}, vol.~1670 of {\em
  Lecture Notes in Mathematics}.
\newblock Springer-Verlag, Berlin, 1997.

\bibitem{MR2244940}
{\sc Nocedal, J., and Wright, S.~J.}
\newblock {\em Numerical optimization}, second~ed.
\newblock Springer Series in Operations Research and Financial Engineering.
  Springer, New York, 2006.

\bibitem{doi:10.1111/1467-8659.t01-1-00587}
{\sc Redon, S., Kheddar, A., and Coquillart, S.}
\newblock Fast continuous collision detection between rigid bodies.
\newblock {\em Computer Graphics Forum 21}, 3 (2002), 279--287.

\bibitem{reiter_2005}
{\sc Reiter, {\relax Ph}.}
\newblock All curves in a {$C^1$}-neighbourhood of a given embedded curve are
  isotopic.
\newblock
  \href{http://www.instmath.rwth-aachen.de/Preprints/reiter20051017.pdf}{Report}~4,
  Institute for Mathematics, RWTH Aachen, October 2005.

\bibitem{2005.07448}
{\sc Reiter, {\relax Ph}., and Schumacher, H.}
\newblock Sobolev gradients for the m{\"o}bius energy.
\newblock {\em arXiv e-prints\/} (2020), arXiv:2005.07448.

\bibitem{Rudin:1973:FunctionalAnalysis}
{\sc Rudin, W.}
\newblock {\em Functional analysis}.
\newblock McGraw-Hill Book Co., New York-D\"{u}sseldorf-Johannesburg, 1973.
\newblock McGraw-Hill Series in Higher Mathematics.

\bibitem{knotplot_2017}
{\sc Scharein, R.}
\newblock Knot{P}lot, 2017.
\newblock {P}rogram for drawing, visualizing, manipulating, and energy
  minimizing knots. {V}ersion from {M}arch 10, 2017. {S}ee
  http://www.knotplot.com.

\bibitem{ScholtesSchumacherWardetzky:2019:DiscreteElasticae}
{\sc Scholtes, S., Schumacher, H., and Wardetzky, M.}
\newblock {Variational convergence of discrete elasticae}.
\newblock {\em IMA Journal of Numerical Analysis\/} (2020).

\bibitem{steenebruegge_2020}
{\sc Steenebr{\"u}gge, D.}
\newblock {\em Gradient flows and regularity of critical points geometric
  curvature energies on curves}.
\newblock PhD thesis, RWTH Aachen University, 2021.
\newblock preliminary title, work in progress.

\bibitem{strzelecki-etal_2010}
{\sc Strzelecki, P., Szuma{\'n}ska, M., and von~der Mosel, H.}
\newblock Regularizing and self-avoidance effects of integral {M}enger
  curvature.
\newblock {\em Ann. Sc. Norm. Super. Pisa Cl. Sci. (5) 9}, 1 (2010), 145--187.

\bibitem{StrzekeckivonderMosel:2013:MengerSurvey}
{\sc Strzelecki, P., and von~der Mosel, H.}
\newblock Menger curvature as a knot energy.
\newblock {\em Phys. Rep. 530}, 3 (2013), 257--290.

\bibitem{vonbrecht-blair_2017}
{\sc von Brecht, J.~H., and Blair, R.}
\newblock Dynamics of embedded curves by doubly-nonlocal reaction-diffusion
  systems.
\newblock {\em J. Phys. A 50}, 47 (2017), 475203, 57.

\bibitem{Whitney:1957:GeometricIntegrationTheory}
{\sc Whitney, H.}
\newblock {\em Geometric integration theory}.
\newblock Princeton University Press, Princeton, N. J., 1957.

\bibitem{2006.07859}
{\sc Yu, {\relax Ch}., Schumacher, H., and Crane, K.}
\newblock Repulsive curves.
\newblock {\em arXiv e-prints\/} (2020), arXiv:2006.07859.
\newblock to appear in Transactions of Graphics.

\bibitem{Zeidler:1993:NonlinearfunctionalanalysisanditsapplicationsFixedpointtheorems}
{\sc Zeidler, E.}
\newblock {\em Nonlinear Functional Analysis and Its Applications:
  {{Fixed}}-Point Theorems}, 2., corr. print~ed.
\newblock No.~1 in Nonlinear Functional Analysis and Its Applications.
  {Springer}, {New York, NY}, 1993.

\end{thebibliography}
\markright{\itshape\bibname}
}
\end{document}